\documentclass[10pt]{article}

\usepackage[utf8]{inputenc}
\usepackage[english] {babel}
\usepackage{easyReview}

\usepackage{amsmath}
\usepackage{amsfonts}
\usepackage{amssymb}
\usepackage{amsthm}
\usepackage{thmtools}
\usepackage{mathtools}
\usepackage{bbm}

\usepackage{braket}
\usepackage{enumerate}
\usepackage{esint}
\usepackage{hyperref}
\usepackage{enumitem}

\numberwithin{equation}{section}

\usepackage{cleveref}

\newcommand{\mres}{\mathbin{\vrule height 1.6ex depth 0pt width
0.13ex\vrule height 0.13ex depth 0pt width 1.3ex}}  

\newtheorem{theorem}{Theorem}[section]
\newtheorem{corollary}[theorem]{Corollary}
\newtheorem{lemma}[theorem]{Lemma}
\newtheorem{proposition}[theorem]{Proposition}

\theoremstyle{definition} \newtheorem{definition}[theorem]{Definition}
\theoremstyle{definition}  
\theoremstyle{definition}  \newtheorem{example}[theorem]{Example}
\theoremstyle{definition}  \newtheorem{remark}[theorem]{Remark}

\theoremstyle{definition}  

\theoremstyle{definition}  





\mathtoolsset{
	showonlyrefs
}

\def\Xint#1{\mathchoice
	{\XXint\displaystyle\textstyle{#1}}%
	{\XXint\textstyle\scriptstyle{#1}}%
	{\XXint\scriptstyle\scriptscriptstyle{#1}}%
	{\XXint\scriptscriptstyle\scriptscriptstyle{#1}}%
	\!\int}
\def\XXint#1#2#3{{\setbox0=\hbox{$#1{#2#3}{\int}$ }
		\vcenter{\hbox{$#2#3$ }}\kern-.6\wd0}}

\def\dashint{\Xint-}

\makeatletter
\let\save@mathaccent\mathaccent
\newcommand*\if@single[3]{%
	\setbox0\hbox{${\mathaccent"0362{#1}}^H$}%
	\setbox2\hbox{${\mathaccent"0362{\kern0pt#1}}^H$}%
	\ifdim\ht0=\ht2 #3\else #2\fi
}

\newcommand*\rel@kern[1]{\kern#1\dimexpr\macc@kerna}

\newcommand*\widebar[1]{\@ifnextchar^{{\wide@bar{#1}{0}}}{\wide@bar{#1}{1}}}

\newcommand*\wide@bar[2]{\if@single{#1}{\wide@bar@{#1}{#2}{1}}{\wide@bar@{#1}{#2}{2}}}
\newcommand*\wide@bar@[3]{%
	\begingroup
	\def\mathaccent##1##2{%
		\let\mathaccent\save@mathaccent
		\if#32 \let\macc@nucleus\first@char \fi
		\setbox\z@\hbox{$\macc@style{\macc@nucleus}_{}$}%
		\setbox\tw@\hbox{$\macc@style{\macc@nucleus}{}_{}$}%
		\dimen@\wd\tw@
		\advance\dimen@-\wd\z@
		\divide\dimen@ 3
		\@tempdima\wd\tw@
		\advance\@tempdima-\scriptspace
		\divide\@tempdima 10
		\advance\dimen@-\@tempdima
		\ifdim\dimen@>\z@ \dimen@0pt\fi
		\rel@kern{0.6}\kern-\dimen@
		\if#31
		\overline{\rel@kern{-0.6}\kern\dimen@\macc@nucleus\rel@kern{0.4}\kern\dimen@}%
		\advance\dimen@0.4\dimexpr\macc@kerna
		\let\final@kern#2%
		\ifdim\dimen@<\z@ \let\final@kern1\fi
		\if\final@kern1 \kern-\dimen@\fi
		\else
		\overline{\rel@kern{-0.6}\kern\dimen@#1}%
		\fi
	}%
	\macc@depth\@ne
	\let\math@bgroup\@empty \let\math@egroup\macc@set@skewchar
	\mathsurround\z@ \frozen@everymath{\mathgroup\macc@group\relax}%
	\macc@set@skewchar\relax
	\let\mathaccentV\macc@nested@a
	\if#31
	\macc@nested@a\relax111{#1}%
	\else
	\def\gobble@till@marker##1\endmarker{}%
	\futurelet\first@char\gobble@till@marker#1\endmarker
	\ifcat\noexpand\first@char A\else
	\def\first@char{}%
	\fi
	\macc@nested@a\relax111{\first@char}%
	\fi
	\endgroup
}

\usepackage[a4paper,top=3.2cm,bottom=2.7cm,left=4.5cm,right=4.5cm, bindingoffset=0mm]{geometry}
\usepackage{a4wide}
\usepackage{microtype}

\title{Generalized minimizers for the capillarity energy outside convex obstacles}
\author{Gabriele Nani and Simone Vincini}
\date{\today}

\begin{document}
	\maketitle
    \begin{abstract}
        In general, the capillarity energy outside an unbounded obstacle may fail to admit volume-constrained minimizers, as infimizing sequences may escape to infinity.

        We study this minimization problem outside a closed convex obstacle $C$. 
        Via a $\Gamma$-convergence result for the capillarity energies outside sequences of closed convex obstacles converging in the Hausdorff sense, we prove that the infimum of the functional is attained by a \emph{generalized minimizer}, namely, a finite union of regions, each of which minimizes the capillarity energy outside a limit of translations of $C$ in its volume class, with the total volume satisfying the prescribed constraint. Each generalized minimizer can be suitably approximated by an infimizing sequence, thus providing a description of its asymptotic behaviour.

        As an application, we study the isoperimetric profile associated with the capillarity energy and investigate how the asymptotic geometry of the obstacle affects both this profile and generalized minimizers and their connectedness.
    \end{abstract}
	\tableofcontents
	\section{Introduction}
        The classical capillarity problem originates in the study of equilibrium configurations of liquid droplets inside a container, under the combined effects of surface tension and adhesion to the boundary. 
        Its mathematical formulation can be traced back to the works of Young, Laplace, and Gauss, and was later developed in a rigorous variational framework by Finn \cite{Finn}, whose monograph remains a standard reference on the subject. A modern treatment from the viewpoint of geometric measure theory and geometric variational problems can be found in the book by Maggi \cite{Ma12}.
        
        In this paper we consider the equivalent formulation in which the droplet is constrained to lie in the complement of a given closed convex set $C$, which plays the role of the obstacle. 
        This setting naturally arises in several applied contexts, such as models for nanowires studied by Fonseca, Fusco, Leoni and Morini \cite{FoFuLeMo23}; see the references therein for further discussion of the modelling aspects.
        
        In variational terms, the problem is formulated as the minimization, under a volume constraint, of the capillarity energy functional $\mathcal{F}_{\lambda, C}$.
        When the obstacle has full dimension, i.e.~it is a convex body, its capillarity energy is defined as
        \begin{equation}\label{eq: energia}
             \mathcal{F}_{\lambda, C}(E) \coloneqq P(E;\mathbb{R}^N\setminus C) - \lambda \,\mathcal{H}^{N-1}(\partial^*E \cap \partial C),
        \end{equation}
        where $C \subset \mathbb{R}^N$ is a convex body, $P(\cdot;\mathbb{R}^N\setminus C)$ denotes the relative perimeter outside $C$, $\partial^*E$ is the reduced boundary of $E$, and $\lambda \in (-1,1)$ is the adhesion coefficient;
        When the obstacle has codimension 1, the wet part $\mathcal{H}^{N-1}(\partial^*E \cap \partial C)$ is replaced by the integral of the traces on both sides, while with larger codimension it vanishes completely, see \eqref{eq: energy}.
        In the following, for readability, we will just write $\mathcal{F}_{C}$ to refer to the capillarity energy.
        
        The main goal of the paper is to study the existence of (generalized) minimizers for the capillarity energy outside closed convex obstacles in presence of a volume constraint, to describe the associated isoperimetric profile and study the properties of the (generalized) minimizers.

        \paragraph{Background.}
        Similar problems have already been thoroughly studied.
        In the special case $\lambda = 0$, the functional reduces to the relative perimeter outside $C$, and the problem becomes the relative isoperimetric problem in $\mathbb{R}^N \setminus C$.
        In this context, Choe, Ghomi, and Ritoré \cite{ChoeGhomiRitore} proved the relative isoperimetric inequality in the exterior of a convex body, showing that, with the same volume constraint, the energy of the minimizers outside of a flat obstacle is less than the energy of the minimizers outside a general closed convex obstacle.
        They furthermore characterized the equality case in the assumption of smoothness of the boundary of the obstacle, showing that it implies the existence of flat facets in $\partial C$; the smoothness assumption was later dropped by Fusco and Morini \cite{FuMo23}. 
        More recently, this results have been extended to the full range of $\lambda \in (-1, 1)$ by Fusco, Julin, Morini, and Pratelli \cite{IsopCap}.

        The first part of this paper will be devoted to the study of the existence of (possibly generalized) minimizers for the energy \eqref{eq: energia} outside closed convex obstacles.
        When the obstacle $C$ is compact, the existence of minimizers with a volume constraint is known and can be proved e.g.~via the Direct Method of the Calculus of Variations and some clever slicing arguments, similar to the ones in \cite[Proposition 3.2]{ritoreExistenceCharacterizationRegions2004}.
        The situation changes drastically when $C$ is unbounded. 
        In this case, minimizers may fail to exist even for simple geometries and minimizing sequences typically escape to infinity.
        This loss of compactness provides one of the main difficulties in the study of isoperimetric and capillarity problems in unbounded settings.

        A closely related approach to this loss of compactness was developed by Nardulli \cite{Nardulli2014} in the study of volume-constrained minimizers of the perimeter on complete noncompact Riemannian manifolds.
        There, the same lack of compactness for minimizing sequences is tamed via a concentration--compactness--type analysis: under suitable assumptions, minimizing sequences are shown to decompose into a finite number of components, some of which may drift to infinity.        
        Taking pointed limits of the ambient manifold, Nardulli introduced a notion of generalized isoperimetric region, living partly in suitable limit manifolds at infinity.
        Via this notion, a generalized existence theorem is proved in this enlarged framework. Since the introduction of the ideas of Nardulli, related concentration--compactness and asymptotic decomposition techniques have been developed and applied in several geometric and variational settings, especially in the study of isoperimetric problems on noncompact Riemannian manifolds and, more recently, in the setting of $\mathrm{RCD}(K,N)$ spaces.
        There, analogous decomposition results have been obtained by Antonelli, Nardulli, and Pozzetta \cite{refId0}.
        We refer to the survey by Pozzetta \cite{PozzettaSurvey} and references therein for an overview of these developments in the Riemannian and synthetic setting.  

        The study of the corresponding loss of compactness phenomenon for the capillarity energy outside unbounded convex bodies in the Euclidean space, where some mass of minimizing sequences may escape to infinity is already present in \cite{FoFuLeMo23}.
        There the authors study a volume constrained capillarity energy functional outside certain unbounded convex bodies and obtain an existence result by exploiting a particular cylindrical structure of the obstacle at infinity, which allows one to recover compactness by translating escaping components along the invariant directions.
        In particular, given $C\subset \mathbb R^N$ a convex body such that $C\cap \{x_N <0\}$ is a semi-cylinder with compact section, $v>0$ and $\lambda\in(-1, 1)$, there exists $F\subset \mathbb{R}^N\setminus C$ minimizing $\mathcal{F}_C$ among all sets $E\subset \mathbb{R}^N\setminus C$ such that $|E|=v$.
		The proof is a more or less standard application of the direct method of the calculus of variations: up to a perturbation, which deals with the volume constraint, the functional is shown to have a minimum $F_R$ in any ball $B(0, R)$; a regularity argument proves that the number of connected components of $F_R$ is equibounded, so that every component escaping at infinity can be slid back along the cylinder and kept at finite distance; finally, by compactness and semicontinuity, a minimum is found.
		
		It is then natural to ask what happens when instead one imposes no condition on the obstacle $C$, except for being a convex body. 
        Since $C$ does not enjoy any invariance by translations, the escaping connected components of the sequence $F_R$ cannot be moved without changing the obstacle. 
        However, the energy is still invariant by translation of both the obstacle and the competitor: for all $x\in\mathbb R^N$
		\begin{equation}
			\mathcal F_C(E)=\mathcal F_{C+x}(E+x).
		\end{equation}
		A way to describe the asymptotic behaviour of a sequence $E_n$ escaping at infinity is then to translate back to the finite the sets $E_n$ and deal with the sequence of translations of $C$ by a diverging sequence of vectors $x_n$. 
        It is then natural to look at limits $\tilde{C}=\lim\limits_{n \to +\infty} C+x_n$, where the limit is to be intended in the Kuratowski convergence.
		This sequences have been studied by Leonardi, Ritoré and Vernadakis in \cite{LeRiVe22} in order to deal with similarly escaping sets in the study of the relative isoperimetric problem but in the inside of convex bodies. The limits can be shown to be convex cylinders, with possibly empty interior. 
		In order to understand what happens in the limit, thus,  we have to extend the definition of $\mathcal F_C$ to the case in which $C$ has positive codimension. If the codimension is bigger than one, we just neglect the contribution of the wet part $\mathcal \partial^* E\cap C$, as it is natural since $\mathcal H^{n-1}(C)=0$. If instead $C$ has codimension 1 (and thus we may assume $C\subset \{x_n=0\}$), we substitute $P(E; C)$ with $P(E\cap\{x_N>0\}; C) +P(E\cap\{x_N<0\}; C)$.

        \paragraph{Main results.}
        Our first main result is to prove that the asymptotic behaviour of minimizing sequences can be fully described by looking at the capillarity energy on asymptotic cylinders of the obstacle $C$. 
        We do so by generalizing competitors outside $C$ to asymptotic configurations, which include competitors outside the asymptotic cylinders.
        \begin{definition}\label{def: asymptotic configuration}
			Let $C\subset\mathbb R^N$ be a closed convex set.
            We call the couple $(\mathbf{E}, \mathbf{C})$ an \emph{asymptotic configuration} for $C$ if there exists $k\in \mathbb{N}\cup\{0\}$ such that
            \begin{equation}
				\mathbf E=(E_0, \ldots, E_k), \qquad \mathbf C=(C,C_1, \ldots, C_k),
			\end{equation} 
            where $C_1, \ldots, C_k$ are asymptotic cylinders for $C$, $E_0\subset \mathbb R^N\setminus C$ is a set of finite perimeter, $E_j\subset\mathbb R^N\setminus C_j$ with $|E_j|>0$, $j=1, \ldots, k$ are sets of finite perimeter.
            
			We say that $(\mathbf{E}, \mathbf{C})$ has length $k$ and volume $|\mathbf E|=\sum_{j=0}^k |E_j|$.
        \end{definition}  
		Given an asymptotic configuration, euristically the collection of limits of connected components of a minimizing sequence, we can compute its energy
        \begin{equation}
			\mathcal{G}_C(\mathbf E; \mathbf C)=\mathcal F_C(E_0)+\sum_{j=1}^k \mathcal F_{C_j} (E_j).
		\end{equation}

        \begin{definition}\label{def: generalized minimizer}
			Let $C\subset\mathbb R^N$ be a closed convex set and $v>0$. 
            
            An \emph{asymptotic} (or \emph{ generalized}) \emph{minimizer} for $C$ with volume constraint $v$ is an asymptotic configuration $(\mathbf{E}, \mathbf{C})$ such that
            \begin{equation}
                \mathcal{G}_C(\mathbf{E}, \mathbf{C}) \leq \mathcal{G}_C(\tilde{\mathbf{E}}, \tilde{\mathbf{C}})
            \end{equation}
            among all of the asymptotic configurations $(\tilde{\mathbf{E}}, \tilde{\mathbf{C}})$ for $C$ with volume $v$.
        \end{definition} 
        
        We are able to prove that the infimum of the capillarity energy outside $C$ is attained as a minimum when the energy is extended to asymptotic configurations.
		\begin{theorem}\label{teo: generalized}
			Let $C$ be a closed convex set, $\lambda\in \left(-1, 1\right)$ and $v>0$.		
			There exist a generalized minimizer $(\mathbf E, \mathbf C)$ for $C$ with volume constraint $v$. 
            Moreover:
			\begin{equation}
				\mathcal{G}_C(\mathbf E, \mathbf C)=\min \mathcal{G}_C( \tilde{\mathbf E}, \tilde{\mathbf C})=\inf \mathcal{F}_C(F),
			\end{equation}
			where $|\mathbf E|=v$ and $( \tilde{\mathbf E}, \tilde{\mathbf C})$ varies among all asymptotic configurations for $C$ of volume $v$ and $F\subset\mathbb R^N\setminus C$ ranges among all sets of finite perimeter with volume $v$.
		\end{theorem}
        For geometrically simple obstacles, this result allows to prove the existence of minimizers, generalizing the existence result in \cite{FoFuLeMo23} (see \Cref{cor: nanowires}).
        In order to prove \Cref{teo: generalized}, we first show an interesting intermediate result in \Cref{teo: gamma conv}: the $\Gamma$-convergence of the capillarity energy outside equibounded closed convex obstacles when such obstacles converge in the Hausdorff distance, no matter the possible loss of dimension.

        The existence of generalized minimizers follows a simple recipe.
        We first relax the volume constraint to a volume penalization with a parameter $\Lambda$ and minimize the penalized functional in large balls.
        These ``local'' minimizers satisfy good uniform density estimates (\Cref{prop: density est}), allowing us to identify and separate connected components.
        We then track each sequence of connected component as the radius diverges: if the sequence remains bounded, we consider its limit outside $C$; if the sequence diverges, we translate it back to the finite and consider its limit outside a corresponding asymptotic cylinder, giving rise to an asymptotic configuration.
        Finally, $\Gamma$-convergence allows us to control the asymptotic behaviour of the energy and well-known arguments (\Cref{prop: esposito fusco unb}) conclude that, for a large enough penalization, the minimizer attains the volume constraint.
        
        Furthermore, if $(\mathbf E, \mathbf C)$ is an asymptotic minimizer, all of its components $E_j$ are actual minimizers of the capillarity energy outside the corresponding obstacle $C_j$ within their volume class (\Cref{prop: gli E_i sono Lambda minimi}).
        This provides a lot of tools in order to begin a finer study of the generalized minimizers and of the isoperimetric profile $I_C$ associated to $\mathcal F_C$, defined by
        \begin{align*}
            I_C(v):= \inf\left \{\mathcal F_C(E): E\subset\mathbb R^N\setminus C,\ |E|=v\right \},
        \end{align*}
        which are the next goals of the paper. First, by exploiting the existence of generalized minimizers and their first properties, we study the isoperimetric profile $I_C$, which we prove to be strictly increasing and locally Lipschitz away from $0$, with a characterization of its derivative at differentiability points (\Cref{prop: proprietà del profilo isoperimetrico}).
        The same $\Gamma$-convergence result also yields an upper semicontinuity result for the isoperimetric profile with respect to the Kuratowski convergence of the obstacle, as shown in \Cref{prop: semicontinuity isop profile}.

        We then move on to study the effect of the asymptotic geometry of the obstacle on the isoperimetric profile and on the connectedness of minimizers.
        In \Cref{teo: principale d*(C)} we prove a dichotomy with respect to the asymptotic dimension $d^*(C)$, which roughly quantifies the number of directions where the obstacle is unbounded.
        
        First, using the characterization of the equality case in \cite{IsopCap}, we show that when $d^*(C)$ is at least $N-1$, where $N$ is the ambient dimension, the isoperimetric profile is the lowest possible, namely equal to that of the half-space. Moreover, generalized minimizers consist of a spherical cap sitting on a flat part of the boundary of $C$ or on a flat part of the boundary of an asymptotic cylinder of $C$. When the asymptotic dimension is strictly less than $N-1$, we show that the isoperimetric profile approaches asymptotically the perimeter of the sphere, showing that the wet part contributes less and less as the volume grows.
        In this case, via a blow-up and a blow-down procedure, we are able to prove that generalized minimizers consist of a single connected component in the large-volume and small-volume regimes.

        Whether connectedness holds for intermediate volumes in the case $d^*(C)<N-1$ remains open in general (except for some particular cases, see e.g.~\cite{maderbaumdicker2025quantitative}). 
        However, in \Cref{prop: bound componenti e diametro} we show that the number of connected components of a generalized minimizer is bounded by a universal constant, independent of the obstacle and of the volume.
		
        \paragraph{Organization of the paper.}
		In \Cref{section: preliminari} we collect the preliminary material and recall the main tools from geometric measure theory and convex geometry that will be needed in the sequel. 
        \Cref{section: Gamma} is devoted to the proof of the $\Gamma$-convergence result, \Cref{teo: gamma conv}, for the capillarity functional under Hausdorff convergence of equibounded closed convex obstacles. In \Cref{section: existence} we introduce the notion of asymptotic configurations and prove the main existence theorem for generalized minimizers, \Cref{teo: generalized}. In \Cref{section: profilo} we establish a series of fundamental properties of the isoperimetric profile $I_C$ associated to the capillarity energy functional. Finally, \Cref{section: connessione} is devoted to the study of the connectedness properties of generalized minimizers and to their dependence on the asymptotic geometry of the obstacle. 
				
	\section{Preliminaries}\label{section: preliminari}
    \subsection{Notation and basic definitions}
    \paragraph{Notation.}
		Whenever not otherwise specified, any set will be a subset of $\mathbb{R}^N$, where $N>1$ is fixed once and for all throughout the paper. Any dependence on the dimension $N$ will be left implicit. 
        
		The notation $B\left(x, R\right)$ will denote the open ball of radius $R$ centered in $x\in \mathbb R^N$, whereas writings such as $\{f\left(x\right)>\left(=, <, \ldots\right)g\left(x\right)\}$ will denote the subset of $\mathbb R^N$ where the relation is satisfied.
        The notation $A\subset B$ means that $A$ is a subset of $B$, while $A\Subset B$ means $\bar{A} \subset B$, where $\bar{A}$ denotes the topological closure of $A$. The topological interior of a set $A$ is denoted with ${\rm int}(A)$.
        Every set considered (again, if not mentioned explicitly) will be assumed to be Borel measurable. Given a set $A$, we denote its Lebesgue measure by $|A|$ and its cardinality by $\# A$. We will relabel subsequences whenever convenient, without further mention, to simplify the notation.
        With a quite usual abuse of notation, we will say that a sequence of sets converges in $L^1$ (or in an $L^1$ sense/fashion) if that applies to the characteristic functions of the sets. The same will apply for $L^1_{loc}$.
        Given $(\mu_n)_{n\in \mathbb N}$ a sequence of $\mathbb R^k$-valued Radon measures and $\mu$ another $\mathbb R^k$-valued Radon measure, we say that $\mu_n\xrightharpoonup{*} \mu$ if for all $\varphi\in C^0_c(\mathbb R^N; \mathbb R^k)$ it holds
        \begin{equation}
            \int_{\mathbb R^N}\varphi\cdot d\mu_n\to \int_{\mathbb R^N}\varphi\cdot d\mu.
        \end{equation}

        We will call convex body a closed convex set of full dimension, i.e.~with non-empty interior.

        \paragraph{Functions of bounded variation and sets of finite perimeter.}
        A function $u\in L^1\left(\Omega\right)$ is said to be of bounded variation if its distributional derivative $Du$ is a finite Radon measure. Given a domain $\Omega$, we call $BV(\Omega)$ the space of functions of bounded variation on $\Omega$.
        Similarly, if $u\in L^1_{loc}(\Omega)$ and $Du$ is a Radon measure, we write $u\in BV_{loc}(\Omega)$. 
        
        Given a measurable set $E\subset \Omega$, we say that $E$ is a set of finite (resp. locally finite) perimeter in $\Omega$ if $D \mathbbm{1}_E$ is a finite Radon measure (resp. $D \mathbbm{1}_E$  is a Radon measure). Given $B\subset \Omega$ a Borel set, we write $P(E; B)=|D \mathbbm{1}_E|(B)$, $P(E)=P(E; \mathbb R^N)$. In particular, if $B$ is open, 
			\begin{equation}\label{eq: perimetro}
				P(E; B)=\sup\left\{ \int_E {\rm div}\, T\;dx:\; T\in C^1_c(B),\, |T|\leq 1\right\}.
			\end{equation}
			Given a sequence $(u_n)_{n\in \mathbb N}\subset BV(\Omega)$ and $u\in BV(\Omega)$, we say that $u_n\to u$ strictly if $u_n\to u$ in $L^1$ and $|Du_n|(\Omega)\to |Du|(\Omega)$. Similarly, if $(E_n)_{n\in \mathbb N}$ is a sequence of sets of finite perimeter, we say that $E_n$ converges strictly to a set $E$ if the characteristic functions $\mathbbm{1}_{E_n}$ converge strictly to $\mathbbm{1}_E$, i.e.\ if $E_n\to E$ in $L^1$ and $P(E_n)\to P(E)$.
			
			For a function $u\in BV(\Omega)$ there is a well defined notion of trace on rectifiable sets in the domain or on the boundary of the domain.
			\begin{theorem}[{\cite[Theorem 3.77]{AmFuPa00}}]\label{teo: trace 1}
				Let $u\in BV(\Omega)$ and let $\Gamma \subset \Omega$ be a $\mathcal H^{N-1}$--rectifiable set oriented by $\nu$. Then, for $\mathcal H^{N-1}$--a.e. $x\in\Gamma$ there exist $\operatorname{Tr}^+_\Gamma u(x), \operatorname{Tr}^-_\Gamma u(x)$ such that
				\begin{equation}
					\lim_{\rho \to 0^+} \dashint_{B^\pm(x, \rho, \nu(x))}|u(y)-\operatorname{Tr}^\pm_\Gamma u(x)|\;dy=0,
				\end{equation}
				where $B^\pm(x, \rho, \nu(x))=B(x, \rho)\cap\{\pm x\cdot \nu> 0\}$.
				Moreover, $Du_{|\Gamma}=(\operatorname{Tr}^+_\Gamma u-\operatorname{Tr}^-_\Gamma u)\nu\,\mathcal H^{N-1}_{|\Gamma}$.
			\end{theorem}
			
			\begin{theorem}[{\cite[Theorem 3.87]{AmFuPa00}}] \label{teo: trace 2}
				Let $\Omega\subset\mathbb R^N$ be an open domain with bounded Lipschitz boundary. Let $u\in BV(\Omega)$. Then, for $\mathcal H^{N-1}$--a.e. $x\in\partial\Omega$ there exists $\operatorname{Tr}_\Omega u(x)$ such that
				\begin{equation}
					\lim_{\rho \to 0^+}\dashint_{B(x, \rho)\cap\Omega }|u(y)-\operatorname{Tr}_\Omega u(x)|\;dy=0.
				\end{equation}
				The trace operator which sends $u$ to $\operatorname{Tr}_\Omega u$ is continuous from $BV(\Omega)$ to $L^1(\partial \Omega)$ where the first is endowed with the topology of strict convergence and the second with the strong topology.
			\end{theorem}
            \begin{remark}
                If $\Omega\subset\mathbb R^N$ is an open domain with bounded Lipschitz boundary, $\partial \Omega$ is a rectifiable set oriented by the exterior normal $\nu$.
                In this case, if $u\in BV(\mathbb R^N)$, then $u_{|\Omega}\in BV(\Omega)$  and $\operatorname{Tr}_\Omega  u_{|\Omega} =\operatorname{Tr}^-_{\partial \Omega} u$.
            \end{remark}
			
			One of the most fundamental results in the theory of sets of finite perimeter is the so called De Giorgi structure theorem, that allows to represent the Gauss--Green measures $\mu_E\coloneqq D\mathbbm{1}_E$ in terms of the Hausdorff measure on the reduced boundary:
			\begin{theorem}[De Giorgi Structure Theorem, {\cite[Theorem 15.5 and Theorem 15.9]{Ma12}}]
				Let $E\subset\mathbb R^N$ be a set of (locally) finite perimeter. Then there exist a (locally) $\mathcal H^{N-1}$--rectifiable set $\partial^* E\subset \partial E$, called \emph{reduced boundary} of $E$, and a function $\nu_E:\partial^* E\to\mathbb{S}^{N-1}$, called \emph{outward generalized normal}, such that:
				\begin{itemize}
					\item the polar decomposition of $\mu_E$ reads as
					\begin{equation}
						\mu_E=-\nu_E\; \mathcal H^{N-1}\mres\partial^* E;
					\end{equation}
					\item $\nu_E (x)\perp T_x \partial^* E$ for $\mathcal H^{N-1}$ every $x$ in $\partial^* E$;
					\item $\lim\limits_{r\to 0^+}\left(\frac{E-x}{r}\right)= \{y\in \mathbb{R}^N \, : \, y\cdot \nu_E(x)<0\}$ in $L^1_{loc}$ for every $x\in\partial^* E$.
				\end{itemize}
			\end{theorem}
            In some occasions it is useful to replace the reduced boundary of a set with its so called essential boundary.
            \begin{definition}
                Let $E$ be a Borel set. For $\alpha\in[0, 1]$, we write
                \begin{equation}
                    E^{(\alpha)}\coloneqq\left\{x\in\mathbb R^N:\; \lim_{r\to 0^+} \frac{|E\cap B(x, r)|}{|B(x, r)|}=\alpha\right\}.
                \end{equation}
                
                The essential boundary of $E$ is defined as $\partial^e E\coloneqq \mathbb{R}^N\setminus (E^{(0)}\cup E^{(1)})$.
            \end{definition}
			By the De Giorgi structure theorem, $\partial^*E\subset E^{\left(\frac{1}{2}\right)}\subset \partial^e E$;
            Federer's Theorem improves the inclusion by stating that the two boundaries are $\mathcal H^{n-1}$-equivalent.
			\begin{theorem}[Federer's Theorem, {\cite[Theorem 16.2]{Ma12}}] \label{teo: Federer}
				Let $E$ be a set of locally finite perimeter. Then
				\begin{equation}
					\mathcal H^{N-1}(\partial^e E\setminus \partial ^*E)=0.
				\end{equation}
			\end{theorem}
			\begin{remark}
				The reduced boundary somewhat defines what is the smallest boundary for a set in the same equivalence class as $E$. Indeed, it is possible to show that $\partial E^{(1)}=\widebar{\partial^* E}$.
			\end{remark}

        \paragraph{Capillarity energy.}
		Given a convex body $C$ and a parameter (the \emph{adhesion coefficient}) $\lambda\in(-1, 1)$, the capillarity energy of a set $E\subset \mathbb R^N\setminus C$ of finite perimeter is
			\begin{equation}
            \mathcal{F}_C\left(E\right)=P\left(E;\mathbb{R}^N\setminus C\right)-\lambda P\left(E; C\right).
			\end{equation}
        In the following, we will extend the definition of $\mathcal F_C$ to possibly lower dimensional closed convex sets $C$. 
        In particular, given $E\subset \mathbb R^N\setminus C$ and $\lambda\in(-1, 1)$ we define
		\begin{equation}\label{eq: energy}
			\mathcal F_C(E)=
			\begin{cases}
				P\left(E;\mathbb{R}^N\setminus C\right)-\lambda P\left(E; C\right)\quad&\text{if dim}(C)=N,\\
				P\left(E;\mathbb{R}^N\setminus C\right)-\lambda \int_C \left(\operatorname{Tr}_{C}^+ \mathbbm{1}_E+ \operatorname{Tr}_{C}^- \mathbbm{1}_E\right)\,d\mathcal H^{N-1}
				&\text{if dim}(C)=N-1,\\
				P\left(E;\mathbb{R}^N\setminus C\right) &\text{if dim}(C) \leq N-2.
				\end{cases}
			\end{equation}
        We extend the functional $\mathcal F_C$ to all Borel sets by setting it to be $+\infty$ outside the class of sets of finite perimeter contained in $\mathbb R^N\setminus C$. 
        We will refer to the quantity $\mathcal{F}_{C}(E)-P(E,\mathbb{R}^N\setminus C)$ as the \emph{wet contribution} of $E$.
        
        Given a Borel set $\Omega \subset \mathbb{R}^N$, we also define a ``localized'' version of the functional
            \begin{equation}
				\mathcal F_C(E,\Omega)=
				\begin{cases}
					P(E, (\mathbb{R}^N\setminus C)\cap \Omega) - \lambda P(E,C \cap \Omega)\hspace{3.66cm}&\text{if dim}(C)=N,
                    \\
					P\left(E,(\mathbb{R}^N\setminus C)\cap \Omega\right)-\lambda \int_{C\cap \Omega} \left(\operatorname{Tr}_{C}^+ \mathbbm{1}_E+ \operatorname{Tr}_{C}^- \mathbbm{1}_E\right)\,d\mathcal H^{N-1} &\text{if dim}(C)=N-1,\\
					P\left(E,(\mathbb{R}^N\setminus C)\cap \Omega\right) &\text{if dim}(C) \leq N-2.
				\end{cases}
			\end{equation}
            In above definitions, for the case $\dim (C)=N-1$, we are fixing any orientation of $C$, since the value of the integral doesn't depend on this choice.
            Similarly to above, we extend the functional to $+\infty$ outside the class of sets of finite perimeter contained in $\mathbb R^N\setminus C$. 
            
			In order to exploit the regularity properties of $\Lambda$--minimizers in the study of volume constrained minima of $\mathcal{F}_C$, we will be interested in a penalized version of the functional: given a fixed volume $v>0$, a parameter $\Lambda>0$ and a Borel set $E$, we denote
			\begin{equation}
				\mathcal{F}_{C, \Lambda}^{v}\left(E\right)= \mathcal{F}_C\left(E\right)+\Lambda \left||E|-v\right|.
			\end{equation} 
            
		\subsection{Local existence and regularity}
        The existence of minimizers for the capillarity problem in the case of a compact convex obstacle does not present significant difficulties.
        In this subsection we show that minimizers exist and satisfy good regularity properties, that will be instrumental in the study of the more involved noncompact case. The following useful lemma is a generalization of an argument present in \cite[Proposition 19.3]{Ma12}.
			\begin{lemma}\label{lemma: local bulk wet estimate}
				Let $C\subset \mathbb R^N$ be a set of locally finite perimeter such that $\mathcal H^{N-1}(\partial C\setminus \partial^* C)=0$. 
                Then for all $\eta, \delta>0$ there exists $m=m\left(\eta, \delta, C\right)>0$ such that for all $E\subset \mathbb{R}^N\setminus C$ finite perimeter sets it holds
				\begin{equation}
					P\left(E; \partial^* C\right)\leq P\left(E; I_\delta\right)+m|E\cap I_\delta|+\eta,
				\end{equation}
				where $I_\delta=\{x \in \mathbb{R}^N\setminus C:\; 0 < d\left(x, C\right)<\delta \}$.
		\end{lemma}
		\begin{proof}
			Let $\eta, \delta>0$ and $C, E$ as above. 
            Let $R>0$, by \eqref{eq: perimetro} and a cut-off argument, there exists $T\in C^{\infty}_c\left(B(0,R);\mathbb R^N\right)$ such that $|T|\leq 1$, $\int_{\partial^*C \cap B(0,R)}\left(1+T\cdot \nu_C\right)\, d\mathcal{H}^{N-1}<\eta$ and $\text{supp}( T)\subset \bar{C}\cup I_\delta$. 
            By choosing $m=\|\nabla T\|_\infty$ and recalling $\mathcal H^{N-1}(\partial C\setminus \partial^* C)=0$, we get
			\begin{equation}
					\int_E \operatorname{div}T\, dx\leq m |E\cap I_\delta|. 
			\end{equation}
			On the other hand, by the Gauss--Green formula 
			\begin{align}
				\int_E \operatorname{div}T\, dx&= \int_{\partial^*E\cap \partial^*C \cap B(0,R)}  T\cdot \nu_E\;d\mathcal{H}^{N-1}
				+ \int_{\partial^*E\setminus \bar{C}} T\cdot \nu_E\;d\mathcal{H}^{N-1} \\
				&\geq P\left(E;\partial^* C\cap B(0,R)\right)-\eta-P\left(E;I_\delta\right).
			\end{align}
            Putting everything together and letting $R\to +\infty$ yields the desired inequality.
			\end{proof}
        A version of the same estimate is in place also when the obstacle has dimension $N-1$.
        \begin{lemma}\label{lemma: local bulk wet estimate N-1}
			Let $C\subset \{x_N=0\}$ be a bounded set 
            such that $\mathcal H^{N-1}(\partial C\cap C)=0$, where by $\partial C$ we mean the boundary of $C$ in $\{x_N=0\}$. 
            Then for all $\eta, \delta>0$ there exists $m=m\left(\eta, \delta, C\right)>0$ such that for all sets $E\subset \mathbb{R}^N$ of finite perimeter it holds
				\begin{equation}\label{eq: local bulk wet estimate N-1}
                \int_C {\rm Tr}^+_C \mathbbm{1}_E\,d\mathcal H^{N-1}\leq P(E; C\times (0, \delta))+m|E\cap\; C\times(0, \delta)|+\eta.
            \end{equation}
            The same holds replacing ${\rm Tr}^+_C$ with ${\rm Tr}^-_C$ and $C\times (0, \delta)$ with $C\times (-\delta, 0)$.
		\end{lemma}
        \begin{proof}
            The proof is similar to the one of \Cref{lemma: local bulk wet estimate}. Let $\eta, \delta>0$ and $C, E$ as above.

            Let us start by considering $F=E\cap \{x_N\geq 0\}$.
            Notice that ${\rm Tr}^+_C \mathbbm{1}_E={\rm Tr}^+_C \mathbbm{1}_F$, $P(E; C\times (0, \delta)) = P(F; C\times (0, \delta))$ and $|E\cap\; C\times(0, \delta)|=|F\cap\; C\times(0, \delta)|$.
            In particular, it is enough to prove the statement for $F$.

            Fix $\sigma>0$.
            Let $C_\sigma=\{x\in C: d(x, \{x_N=0\}\setminus C)>\sigma\}$ and notice that, by assumption, $\mathcal H^{N-1}(C_\sigma)\to \mathcal H^{N-1}(C)$ as $\sigma\to 0$.
            Let also $\varphi_\sigma\in C^1_c(C\times (-\delta, \delta))$ be such that $-1\leq \varphi_\sigma\leq 0$ and $\varphi\equiv-1$ on $C_\sigma$.
            Finally, let $e_N$ be the $N$-th versor. By the Gauss--Green formula 
			\begin{align}
				\int_F \operatorname{div}(\varphi_\sigma e_N)\, dx&= -\int_{\partial^*F\cap C}  \varphi_\sigma \;d\mathcal{H}^{N-1}
				+ \int_{\partial^* F\setminus C} \varphi_\sigma \,(e_N\cdot \nu_F)\;d\mathcal{H}^{N-1} \\
				&\geq \int_{C_\sigma} {\rm Tr}^+_C \mathbbm{1}_F\,d\mathcal H^{N-1}-P\left(F;C\times(0, \delta)\right).
			\end{align}
            On the other hand, 
            \begin{equation}
                \int_F \operatorname{div}(\varphi_\sigma e_N)\, dx\leq \|\nabla \varphi_\sigma\|_\infty |F\cap \,C\times (-\delta, \delta)|=\|\nabla \varphi_\sigma\|_\infty |F\cap \,C\times (0, \delta)|.
            \end{equation}

            Taking $\sigma$ small enough that
            \begin{equation}
                \int_{C_\sigma} {\rm Tr}^+_C \mathbbm{1}_F\,d\mathcal H^{N-1}\geq \int_{C}{\rm Tr}^+_C \mathbbm{1}_F\,d\mathcal H^{N-1}-\eta
            \end{equation}
            concludes the proof.
        \end{proof}
            
		Another well known and fundamental result is that, when the obstacle is convex, the wet part is always controlled by the outside perimeter. 
		\begin{proposition}[{\cite[Lemma 5.1]{FuMo23}}]\label{prop: bulk contact inequality}
			Let $C\subset \mathbb R^N$ be a convex body and $E \subset \mathbb{R}^N\setminus C$ a set with finite measure and finite perimeter. Then
			\begin{equation}
				P\left(E; \partial C\right)\leq P\left(E; \mathbb{R}^N\setminus C\right)
			\end{equation}
		\end{proposition}
		\begin{remark}\label{rmk: bulk contact inequality in the case N-1}
			A proof different from the one in \cite{FuMo23} can be obtained in the case of regular sets by noticing that $T=\nabla d(C, \cdot)$ is an extension of the outer normal of $C$ and has positive divergence: integrating by parts $\operatorname{div} T$ in $E$ yields the inequality. By approximation (for instance, with $C^\eta\coloneqq \{d(x, C)<\eta\}$), the inequality can be shown to hold for general $C$'s. In particular, if $C$ has dimension $N-1$ (and we may assume $C\subset \{x_N=0\}$), it can be strengthened to 
			\begin{equation}
				\int_C \operatorname{Tr}_{C}^+ \mathbbm{1}_E \, d\mathcal H^{N-1}\leq P(E; \{x_N>0\}), 
			\end{equation} 
			where we assume that $C$ is oriented by $e_N$.
		\end{remark}

        As a consequence of \Cref{prop: bulk contact inequality} and \Cref{rmk: bulk contact inequality in the case N-1}, we have

        \begin{corollary}\label{cor: coercivity}
            Let $C$ be a closed convex set.
            Then for every set of finite perimeter $E\subset \mathbb R^N\setminus C$ it holds
            \begin{equation}
                \mathcal{F}_C(E)\geq \frac{1-\lambda}{2} P(E).
            \end{equation}
        \end{corollary}
             
		The proof of the semicontinuity of the energy functional follows the one of \cite[Lemma 3.3]{FoFuLeMo23} (see also \cite[Proposition 19.3]{Ma12}).
		\begin{proposition}\label{prop: semicontinuity}
			Let $C$ be a compact convex set. Let $\left(E_n\right)_{n\in \mathbb N}$ be a sequence of sets of finite perimeter in $\mathbb R^N$ such that $E_n\to E$ in $L^1$ for some set $E\subset \mathbb R^N$. Then
			\begin{equation}
				\mathcal{F}_C \left(E\right)\leq \liminf_n \mathcal{F}_C \left(E_n\right)
			\end{equation}
		\end{proposition}
        \begin{proof}
            Let first $C$ be a convex body, i.e.\ $\dim C=N$.
            Let $(E_n)_{n\in \mathbb N}$ and $E$ be as in the statement. We can assume that $ \liminf\limits_n \mathcal{F}_C \left(E_n\right) <+\infty$ and that $E$ has finite perimeter.
            Fix $\delta, \eta>0$ and set $I_\delta\coloneqq \{x \in \mathbb{R}^N\setminus C:\;  d\left(x, C\right)<\delta \}$.
            Notice that, by Lemma \ref{lemma: local bulk wet estimate},
            \begin{align}
              |P(E_n; \partial C)-P(E;\partial C)|&\leq P(E\Delta E_n; \partial C)\\
              &\leq P(E\Delta E_n; I_\delta)+m|E\Delta E_n\cap I_\delta|+\eta.
            \end{align}
            In particular, 
            \begin{align}
                \mathcal F_C(E_n)-\mathcal F_C(E)&\geq P(E_n, \mathbb{R}^N\setminus C)-P(E, \mathbb{R}^N \setminus C)-\lambda\left( P(E_n; \partial C)-P(E;\partial C) \right)\\
                &\geq P(E_n; \mathbb R^N\setminus \widebar{I}_\delta)+P(E_n; I_\delta)-P(E, \mathbb{R}^N \setminus C)-|P(E_n; \partial C)-P(E;\partial C)|\\
                &\geq P(E_n; \mathbb R^N\setminus \widebar{I}_\delta)+P(E_n; I_\delta)-P(E, \mathbb{R}^N \setminus C)\\
                &\quad-P(E\Delta E_n; I_\delta)-m|E\Delta E_n\cap I_\delta|-\eta\\
                &\geq P(E_n; \mathbb R^N\setminus \widebar{I}_\delta)-P(E, \mathbb{R}^N \setminus C)-P(E; I_\delta)-m|E\Delta E_n\cap I_\delta|-\eta.
            \end{align}
            Sending $n\to \infty$, 
            \begin{equation}
                \liminf_n \mathcal F_C(E_n)-\mathcal F_C(E)\geq P(E; \mathbb R^N\setminus \widebar{I}_\delta)-P(E, \mathbb{R}^N \setminus C)-P(E; I_\delta)-\eta.
            \end{equation}
            Now, sending $\delta\to0$ and $\eta\to 0$, it follows
            \begin{equation}
                \mathcal F_C(E)\leq \liminf_n \mathcal F_C(E_n).
            \end{equation}

            If $\dim C<N-1$, the thesis follows from the semicontinuity of the perimeter.

            Finally, the case $\dim C=N-1$ follows as in the previous case, by applying \Cref{lemma: local bulk wet estimate N-1} instead of \Cref{lemma: local bulk wet estimate}.
        \end{proof}
            
		\begin{corollary}\label{cor: esistenza}
			Let $C$ be a closed convex set and let $R>0$. 
            Then there exists $E_R\subset B\left(0, R\right)\setminus C$ such that
			\begin{equation}\label{eq: minimi locali}
				\mathcal F_{C, \Lambda}^{v}\left(E_R\right)=\min \{\mathcal F_{C, \Lambda}^{v}\left(E\right): \; E\subset B\left(0, R \right) \setminus C\}  .
			\end{equation}
		\end{corollary}
        \begin{definition}[$\Lambda$-minimizer]
            Let $C$ be a closed convex set and 
            let $E\subset \mathbb{R}^N \setminus C$.
            For $R>0$ we say that $E$ is a \emph{$\Lambda$-minimizer of $\mathcal{F}_C$ in $B(0,R)$} if $E\subset B(0, R)$ and
            \begin{align}
            \mathcal{F}_C\left(E\right)\leq \mathcal{F}_C\left(F\right)+\Lambda|E\Delta F| \label{eq: lambda minimality in B(0,R)}
            \end{align}
            holds for all $F\subset B(0,R)\setminus C$.
            In general, we say that $E$ is a \emph{$\Lambda$-minimizer of $\mathcal{F}_C$} if for every open ball $B$
            \begin{equation}
                \mathcal{F}_C\left(E, B\right)\leq \mathcal{F}_C\left(F, B\right)+\Lambda|E\Delta F| \label{eq: lambda minimality}
            \end{equation}  
            holds for every set $F \subset \mathbb{R}^N \setminus C$ such that $F \Delta E \Subset B$.
        \end{definition}
		Thanks to the volume perturbation, if $E$ is a local minimizer as in \eqref{eq: minimi locali}, then it is a $\Lambda$-minimizers of $\mathcal{F}_C$ in $B(0,R)$. Indeed, for every set $F \subset B(0,R) \setminus C$ of finite perimeter,
    \begin{equation}\label{eq: Lambda min non loc}
    \begin{split}
        \mathcal{F}_{C}(E)
        &\le \mathcal{F}_{C}(F)
        + \Lambda \Bigl( \bigl||F| - v\bigr| - \bigl||E| - v\bigr| \Bigr) \\
        &\le \mathcal{F}_{C}(F)
        + \Lambda \bigl||F| - |E|\bigr| \\
        &\le \mathcal{F}_{C}(F) + \Lambda |F \Delta E|.
    \end{split}
    \end{equation}
    It is standard that $\Lambda$-minimizers satisfy useful density estimates, as recalled in the following proposition.
    \begin{proposition}\label{prop: density est}
        Let $C$ be a closed convex set and
        let $E\subset \mathbb{R}^N\setminus C$ be such that it is either a $\Lambda$-minimizer of $\mathcal{F}_C$ or a $\Lambda$-minimizer of $\mathcal{F}_C$ in $B(0,R)$ for an $R>0$.
        Then  
        \begin{enumerate}[label=(\arabic*)]
            \item\label{it: regularity min} $E$ is $\mathcal L^N$--equivalent to the open set $\Omega= E^{(1)} \subset \mathbb R^N$ that satisfies $\partial \Omega=\partial^e\Omega$.
            \item\label{it: density est} There exist $c_0=c_0\left(\lambda\right)>0, r_0=r_0\left(\lambda, \Lambda\right)\in\left(0, 1\right)$ such that if $\Omega'$ is any connected component of $\Omega$, one has \begin{equation} \label{eq: density estimates 2}
                |\Omega'\cap B\left(x, r\right)|\geq c_0 r^N
            \end{equation}
            for all $x\in \widebar{ \Omega'}$ and $r\in \left(0, r_0\right)$.
        \end{enumerate}
            In the case where $E$ is a $\Lambda$-minimizer of $\mathcal{F}_C$ in $B(0,R)$, we have
        \begin{enumerate}[label=(\arabic*)]
        \setcounter{enumi}{2}
            \item\label{it: perim est} Let $x\in B(0, R)$ and $0<r<1$ be such that $\widebar{B}(x, r)\cap C=\emptyset$ and $B(x,r) \Subset B(0,R)$. Then $P(E, B(x, r))\leq (N+\Lambda)\omega_Nr^{N-1}.$
            \item\label{it: height est} There exists $h=h(\lambda, \Lambda)>0$ such that, for every connected component $\Omega'$ of $\Omega$ and $x\in \widebar{\Omega}'$ it holds $\sup\limits_{z\in \widebar{\Omega}'\cap B(x, \frac{r_0}{2})} d(z, C)\geq h$.
        \end{enumerate}
    \end{proposition}
    \begin{proof}
        The proof of the first two items is identical to that of \cite[Lemma 3.6]{FoFuLeMo23}. The case $\dim(C)=N-1$ is treated exactly as the others, except that one has to use the strengthened version of \Cref{prop: bulk contact inequality} stated in \Cref{rmk: bulk contact inequality in the case N-1}.

        Let $x\in B(0, R)$ and $r>0$ be such that $\widebar{B}(x, r)\cap C=\emptyset$ and consider $F=E\setminus B(x, r)$. 
        Since the right-hand side of the inequality in point \ref{it: perim est} is continuous in $r$ and the left-hand side lower semicontinuous, by foliaton lemma it is enough to prove the statement in the hypothesis that $P(E; \partial B(x, r))=0$.
        In particular, 
        \begin{align}
            \mathcal{F}_C\left(F\right)&=\mathcal{F}_C\left(E\right)-P(E; B(x, r))+P(B(x, r); E^{(1)})\\
            &\leq \mathcal{F}_C\left(E\right)-P(E; B(x, r))+N\omega_N r^{N-1}.
        \end{align}
        Plugging the estimate in \eqref{eq: lambda minimality in B(0,R)}, we get
        \begin{equation}
            P(E; B(x, r))\leq N\omega_N r^{N-1}+\Lambda |E\Delta F|\leq (N+\Lambda) \omega_N r^{N-1}.
        \end{equation}

        Finally, let us prove item \ref{it: height est}.
        Let $r=\frac{r_0}2$, $x\in \widebar{\Omega}'$ and $H=\sup\limits_{z\in \Omega'\cap B(x, r)} d(z, C)$.
        Notice that, by the density estimates, 
        \begin{equation}\label{eq: misura intersezione}
            c_0 r^N\leq |\Omega'\cap B(x, r)|\leq |B(x, r)\cap C^{H}|, 
        \end{equation}
        where $C^H=\{z\in \mathbb R^N\setminus C:0\leq d(z, C)\leq H\}$.
        Let us estimate $|B(x, r)\cap C^{H}|$.
        By the Coarea formula, 
        \begin{equation}
            |B(x, r)\cap C^{H}|=\int_0^H \mathcal H^{N-1}\left(\{z\in B(x, r):d(z, C)=s\}\right)\,ds.
        \end{equation}
        Now 
        \begin{equation}
            \{z\in B(x, r):d(z, C)=s\}\subset \partial \{z\in B(x, r): d(z, C)\leq s\}\eqqcolon \partial C_s
        \end{equation}
        for all $s\in [0, H]$.
        Moreover, the sets $C_s$ are convex, so by Proposition \ref{prop: bulk contact inequality}, 
        \begin{equation}
            \mathcal H^{N-1}(\{z\in B(x, r):d(z, C)=s\})\leq \mathcal H^{N-1}(\partial C_s)\leq \mathcal{H}^{N-1}(\partial B(x, r))= N\omega_Nr^{N-1}.
        \end{equation}
        In particular, by \eqref{eq: misura intersezione} we get
        \begin{equation}
            H\geq \frac{c_0 r}{N\omega_N}=\frac{c_0 r_0}{2 N\omega_N}\eqqcolon h.
        \end{equation}
    \end{proof}

\begin{remark}\label{rmk: componenti connesse e scelta del rappresentante}
    By \cite[Theorem 2]{AmbrosioCaselles}, If $E \subset \mathbb{R}^N$ is such that $E^{(1)}$ is open and $\partial E^{(1)} = \partial^e E$, then the $M$-connected components (see \cite[Definition 1]{AmbrosioCaselles} for the definition) of $E$ coincides with the (topological) connected components of $E^{(1)}$. Hence, by \Cref{prop: density est}, if $E \subset \mathbb{R}^N$ is a $\Lambda$-minimizer of $\mathcal{F}_C$, we can unambiguously refer to the connected components of $E$ as those of its representative $E^{(1)}$. Moreover, in the rest of the paper, unless otherwise specified, whenever $E$ is a $\Lambda$-minimizer of $\mathcal{F}_C$, we will always choose $E^{(1)}$ as its precise representative.
\end{remark}

            Thanks to the density estimate in \Cref{prop: density est}, it is possible to prove that there is a bound on the number of connected components of the minimizers of the relaxed functional and on the diameter of its connected components
            \begin{corollary} \label{cor: stima sulle comp connesse e loro diametro}
                Let $v>0$ and $\Lambda>0$. There exist $M=M(\lambda, \Lambda, v)>0$ and $d=d(\lambda, \Lambda, v)>0$ such that if $C$ is a closed convex set, $R>0$ and $E\subset\mathbb{R}^N\setminus C$ is a minimizer of  $\mathcal{F}^{v}_{C,\Lambda}$ among subsets of $B(0,R) \setminus C$ (or among subsets of $\mathbb{R}^N\setminus C$) then $E$ that has less than $M$ connected components and every connected component has diameter less than $d$.
            \end{corollary}
            \begin{proof} 
            Denote with $E$ the $\mathcal{L}^N$-representative of it given by \Cref{prop: density est}.
            Let $\Omega_1, \dots , \Omega_{\widebar{n}}$ be the connected components of $E$. Let $r=r(\lambda, \Lambda)=\frac{r_0(\lambda, \Lambda)}{2}$ and $c_0=c_0(\lambda)$, where $r_0$ and $c_0$ are the constants of \Cref{prop: density est}.
                Then:
                \begin{equation}
                    2 v\geq | E | \geq \Bigl\vert \sum_{i=1}^{\widebar{n}} \Omega_i \cap B(x_i,r) \Bigr\vert = \sum_{i=1}^{\widebar{n}} \Bigl\vert \Omega_i \cap B(x_i,r) \Bigr\vert \geq \widebar{n}c_0 r^N.
                \end{equation}
            So the bound on the number of connected components follows. By an application of Vitali's Covering Lemma, we may find a subset $A\subset E$ such that the closed balls of the family $\left\{ \widebar{B(x,r)} : x \in A \right\}$ are pairwise disjoint and 
            \begin{equation}
                E \subset \bigcup_{x \in A} \widebar{B(x,5r)}.
            \end{equation}
            
            Thanks to the density estimate in \Cref{prop: density est} we have $2v \geq \vert E \vert \geq \#A c_o r^N$.
            If $\Omega$ is a connected component of $E$, since $\Omega \subset \bigcup_{x \in A} \widebar{B(x,5r)} $, we have that
            \begin{equation}
                \text{diam}(\Omega) \leq 10 r \#A \leq \frac{20 v}{c_0 r^{N-1}}.
            \end{equation}
            \end{proof}

\subsection{An approximation result}
In order to prove the approximation result below, we shall use the following lemma, first proved in \cite{Ivanov} and later mentioned in \cite[Point 2 of the Notes for Section 2.3]{Schneider_2013}.

\begin{lemma}\label{lem: Ivanov}
Let $C$ be a convex body. Then, for $\mathcal{L}^N$-a.e. $x\in\mathbb{R}^N$, and for every line $L$ passing through $x$, one has $\mathcal{H}^1(L\cap \partial C)=0$.
\end{lemma}

The previous lemma yields the following proposition.

\begin{proposition}\label{prop: misura bordo conv e conv dilatato}
Let $C$ be a convex body. Then, for $\mathcal{L}^N$-a.e. $x\in\mathbb{R}^N$, there exists a countable set $S\subset (0,+\infty)$ such that
\begin{equation}
\mathcal{H}^{N-1}\bigl(\partial(\lambda (C-x)) \cap \partial (C-x)\bigr)=0
\end{equation}
for every $\lambda\in (0,+\infty)\setminus S$.
\end{proposition}

\begin{proof}
We distinguish the cases $x\in {\rm int}(C)$ and $x\notin C$. Assume first that $x\in {\rm int}(C)$. In this case, by convexity, the family
\begin{equation}
\bigl(\partial(\lambda (C-x))\bigr)_{\lambda>0}
\end{equation}
is pairwise disjoint. Hence, by the foliation lemma,
\begin{equation}
\mathcal{H}^{N-1}\bigl(\partial(\lambda (C-x))\cap \partial(C-x)\bigr)=0
\end{equation}
for every $\lambda>0$ except at most countably many values of $\lambda$.

Assume now that $x\in \mathbb{R}^N\setminus C$. By \Cref{lem: Ivanov}, for $\mathcal{L}^N$-a.e. such $x$, every line $L$ passing through $x$ satisfies
\begin{equation}
\mathcal{H}^1(L\cap \partial C)=0.
\end{equation}
Let $x$ be such that previous property is satisfied, since $C$ is convex, every line through $x$ intersects $\partial C$ in at most two points. We then decompose $\partial C$ into the following three disjoint sets:
\begin{equation}
C_1:=\Bigl \{ y\in \partial C \ :\ \exists \ \nu\in \mathbb{S}^{N-1}, \ \exists \ 0<t<s
\ \text{such that } y=x+t\nu \text{ and } x+s\nu\in \partial C \Bigr \},
\end{equation}

\begin{equation}
C_2:=\Bigl \{ y\in \partial C \ :\ \exists \ \nu\in \mathbb{S}^{N-1}, \ \exists \ 0<t<s
\ \text{such that } y=x+s\nu \text{ and } x+t\nu\in \partial C \Bigr \},
\end{equation}
and
\begin{align}
C_3:=\Bigl \{ y\in \partial C \ :\ &\exists \ \nu\in \mathbb{S}^{N-1}, \ \text{such that } (x+\mathbb{R}\nu)\cap \partial C= \{y\} \Bigr \}.
\end{align}
Thus, $C_1$ and $C_2$ consist respectively of the first and second intersection points of rays starting from $x$ with $\partial C$, while $C_3$ contains the points corresponding to rays having a unique intersection with $\partial C$. Fix $\lambda>0$ with $\lambda\neq 1$. By construction,
\begin{equation}
\lambda(C_1-x)\cap (C_1-x)=\emptyset,
\qquad
\lambda(C_2-x)\cap (C_2-x)=\emptyset, 
\end{equation}
and moreover
\begin{equation}
\lambda(C_3-x)\cap \partial(C-x)=\emptyset,
\end{equation}
since points of $C_3$ correspond to rays intersecting $\partial C$ at exactly one point. In the case $0<\lambda<1$ (the case $\lambda>1$ being analogous), we also have, again by construction, that
\begin{equation}
\lambda(C_1-x)\cap (C_2-x)=\emptyset,
\end{equation}
Therefore, if
\begin{equation}
\mathcal{H}^{N-1}\bigl(\partial(\lambda(C-x))\cap \partial(C-x)\bigr)>0,
\end{equation}
then necessarily
\begin{equation}
\mathcal{H}^{N-1}\bigl(\lambda(C_2-x)\cap (C_1-x)\bigr)>0.
\end{equation}
On the other hand, the family
\begin{equation}
\bigl(\lambda(C_2-x)\bigr)_{0<\lambda<1}
\end{equation}
is pairwise disjoint. Hence, applying once again the foliation lemma, we conclude that
\begin{equation}
\mathcal{H}^{N-1}\bigl(\lambda(C_2-x)\cap (C_1-x)\bigr)=0
\end{equation}
for every $\lambda\in(0,1)$ except at most countably many values. This concludes the proof.
\end{proof}
As a consequence of the previous proposition, we prove the following lemma.

\begin{lemma}\label{lm: quasi ovunque disgiunti}
Let $(C_k)_{k\in\mathbb{N}}$ be a sequence of convex bodies. For $\mathcal{L}^N$-a.e. $x\in\mathbb{R}^N$ there exists a set $A\subset (0,+\infty)$ such that for every $\lambda,\mu\in A$, with $\lambda\neq\mu$, one has
\begin{equation}
\mathcal{H}^{N-1}\bigl(\lambda(\partial C_k - x)\cap \mu(\partial C_k - x)\bigr)=0
\end{equation}
for every $k\in\mathbb{N}$. Moreover, for every $\delta>0$ the set $A\cap (1-\delta,1)$ is uncountable.
\end{lemma}

\begin{proof}
As a consequence of \Cref{prop: misura bordo conv e conv dilatato}, for $\mathcal{L}^N$-a.e. $x\in\mathbb{R}^N$ there exists a countable set $S\subset (0,+\infty)$ such that
\begin{equation}
\mathcal{H}^{N-1}\bigl(\lambda(\partial C_k-x)\cap \partial (C_k-x)\bigr)=0
\end{equation}
for every $k\in\mathbb{N}$ and every $\lambda\in (0,+\infty)\setminus S$. Fix such an $x \in\mathbb{R}^N$ and the corresponding set $S$.

Without loss of generality, we may assume that $\mathbb{Q}^+\subset S$. Let $\Gamma$ be the multiplicative subgroup of $(0,+\infty)$ generated by $S$, namely
\begin{equation}
\Gamma=\left\{\frac{t_1\cdots t_l}{s_1\cdots s_m}:\; l,m\in \mathbb{N}, \, \, t_i,s_j\in S, \, \, i=1,\dots, l \, \, j= 1, \dots, m\right\}.
\end{equation}
Since $S$ is countable, also $\Gamma$ is countable.

We define an equivalence relation on $(0,+\infty)$ by
\begin{equation}
\lambda \sim \mu
\qquad \Longleftrightarrow \qquad
\frac{\lambda}{\mu}\in \Gamma.
\end{equation}

Each equivalence class is countable and dense in $(0,+\infty)$.

If $\mu \notin [\lambda]_{\sim}$, then $\frac{\lambda}{\mu}\notin S$, hence 
\begin{equation}
\mathcal{H}^{N-1}\bigl(\lambda(\partial C_k-x)\cap \mu(\partial C_k -x)\bigr) =0
\end{equation}
for every $k\in\mathbb{N}$.

Let $L$ denote the set of equivalence classes. Since $(0,+\infty)$ has cardinality $2^{\aleph_0}$ and each class is countable, $L$ has cardinality $2^{\aleph_0}$. Hence we can write
\begin{equation}
L=\bigsqcup_{n\in\mathbb{N}} L_n
\end{equation}
where the union is disjoint and each $L_n$ has cardinality $2^{\aleph_0}$.

For every $n\in\mathbb{N}$ and for every equivalence class $[\lambda]_{\sim}\in L_n$, choose a representative of the class $[\lambda]_{\sim}$ in the interval
$\left(1- \frac{1}{n},\,1 \right)$,
which is possible since each equivalence class is dense. Let $A_n$ denote the set of all representatives chosen in this way. Then
\begin{equation}
A_n \subset \left(1-\frac{1}{n},\,1 \right),
\qquad |A_n|=2^{\aleph_0}.
\end{equation}

Finally, set
\begin{equation}
A=\bigcup_{n\in\mathbb{N}} A_n.
\end{equation}
By construction, $A$ is the desired set.
\end{proof}
            
			The next approximation result will be used extensively in the following.
			\begin{lemma} \label{lm: approssimazione}
				Let $(C_k)_{k\in\mathbb{N}}$ be a sequence of convex bodies and let $E\subset\mathbb R^N\setminus C_1$ be a set of finite perimeter with $|E|<+\infty$. Then there exist a sequence of bounded open sets with $C^\infty$ boundary $(E_n)_{n\in \mathbb N}$ such that
				\begin{itemize}
					\item $E_n\setminus C_1\xrightarrow{n\to +\infty} E$ strictly;
					\item $P(E_n; \partial C_k)=0$ for every $n,k \in \mathbb{N}$;
					\item $\operatorname{Tr}_{\partial C_1}^+ \mathbbm{1}_{E_n}\xrightarrow{n\to +\infty} \operatorname{Tr}_{\partial C_1}^+ \mathbbm{1}_{E}$ in $L^1(\partial C_1, \mathcal H^{N-1})$.
				\end{itemize}
			\end{lemma}
			\begin{proof}
                By \Cref{lm: quasi ovunque disgiunti} there exists $x\in{\rm int}(C_1)$ and a set $A\subset (0,+\infty)$ such that the conclusions of the previous Lemma hold.  
                Without loss of generality we can assume $x=0$.
				By standard arguments (considering the intersection with balls of diverging radii), we can also assume that there exists $R>0$ such that $E\Subset B(0, R)$.
				
				\textbf{Step 1:} Approximation with sets $F$ such that $P(F, \partial C_k)=0$ for every $k\in\mathbb{N}$.\\
				Let us set, for $\lambda\in(0, 1)$, $F_\lambda=\lambda E$. It is immediate to see that $F_\lambda\to E$ strictly as $\lambda \to 1$. Let us show that there exist a sequence $\lambda_n\to 1^{-}$ such that
                \begin{equation}
                    P(F_{\lambda_n}; \partial C_k)=0 \ \text{for every} \ n,k\in\mathbb{N} \label{eq: foliation nel lemma di approx}
                \end{equation} 
                Indeed, since for every $k\in\mathbb{N}$ the family $(\frac{1}{\lambda} \partial C_k)_{\lambda \in A}$ is $\mathcal{H}^{N-1}$-pairwise disjoint, by the Foliation Lemma we have that there exist a countable set $B \subset A$ such that $P(E, \frac{1}{\lambda} \partial C_k) = 0$ for every $k\in\mathbb{N}$ and every $\lambda \in A \setminus B$.
                By \Cref{lm: quasi ovunque disgiunti}, $1$ is an accumulation point of $A \setminus B$, hence we can find a sequence $\lambda_n \to 1^{-}$ such that \eqref{eq: foliation nel lemma di approx} holds.
                
                Notice moreover that $\mathbbm{1}_{F_{\lambda_n}\setminus C_1}\to \mathbbm{1}_E$ strictly in $BV(\mathbb{R}^N \setminus C_1)$. Indeed,
                \begin{equation}
                P(F_{\lambda_n}, \mathbb{R}^N \setminus C_1)
                = \lambda_n^{N-1} P\!\left(E, \mathbb{R}^N \setminus \frac{1}{\lambda_n} C_1 \right)
                = \int \lambda_n^{N-1}\mathbbm{1}_{\mathbb{R}^N \setminus (\frac{1}{\lambda_n} C_1)} \, d|\mu_E|.
                \end{equation}
                
                Since the integrand converges pointwise to $\mathbbm{1}_{\mathbb{R}^N \setminus C_1}$ and $P(E) < +\infty$, by the Dominated Convergence Theorem we obtain
                \begin{equation}
                P(F_{\lambda_n}, \mathbb{R}^N \setminus C_1) \to P(E, \mathbb{R}^N \setminus C_1).
                \end{equation}
                
                Therefore, \Cref{teo: trace 2} yields
                \begin{equation}
                \operatorname{Tr}_{\partial C_1}^{+}\mathbbm{1}_{F_{\lambda_n}}
                \to
                \operatorname{Tr}_{\partial C_1}^{+}\mathbbm{1}_E
                \qquad \text{in } L^1(\partial C_1,\mathcal H^{N-1}).
                \end{equation}
                
                In particular, this also implies that
                \begin{equation}
                P(F_{\lambda_n} \setminus C_1)\to P(E).
                \end{equation}
				
			\textbf{Step 2:} Approximation by smooth sets.\\
				It is now sufficient to approximate a set $F$ such that $P(F;\partial C_k)=0$ for every $k\in\mathbb{N}$ and $F\Subset B(0,R)$ for some $R>0$.
            By the standard theory of sets of finite perimeter, there exists a sequence $(F_n)_{n\in \mathbb N}$ of open sets with smooth boundary such that $F_n\to F$ strictly in $BV$ and $F_n\Subset B(0,R)$ for every $n\in \mathbb N$. 
            By the same argument as in Step 1, we may assume that $P(F_n;\partial C_k)=0$ for all $n,k\in\mathbb{N}$. Since $F_n \to F$ in $L^1$, we have $\mu_{F_n}\xrightharpoonup{*}\mu_F$.
            Moreover, since $P(F_n)\to P(F)<+\infty$, by \cite[Proposition 4.30]{Ma12} we obtain $|\mu_{F_n}|\xrightharpoonup{*}|\mu_F|$. Since $P(F,\partial C_1)=0$, by \cite[Proposition 4.26]{Ma12} it follows that
            $
            P(F_n,\mathbb{R}^N\setminus C_1)\to P(F,\mathbb{R}^N\setminus C_1).
            $
            Hence $\mathbbm{1}_{F_n\setminus C_1}\to \mathbbm{1}_{F\setminus C_1}$ strictly in $BV(\mathbb{R}^N\setminus C_1)$, which, by \Cref{teo: trace 2}, yields the conclusion.
                \end{proof}

		\subsection{Asymptotic geometry of convex sets}
			In order to study the asymptotic geometric properties of convex sets it is useful to recall the notion of Kuratowski convergence.
			\begin{definition}
				Let $(F_n)_{n\in \mathbb N}$ be a sequence of sets in $\mathbb R^N$. We say that $F_n$ converges in the sense of Kuratowski to a (possibly empty) set $F$ if
				\begin{itemize}
					\item Every point $x\in F$ is the limit of a sequence $(x_n)_{n\geq1}$ such that $x_n\in F_n$ eventually,
					\item any limit point of any sequence $(x_n)_{n\in \mathbb N}$, with $x_n\in F_n$, must be in $F$.
				\end{itemize}
			\end{definition} 
			It is easy to see that the class of closed convex sets is closed with respect to this convergence. We recall that the family of closed subsets of $\mathbb{R}^N$ can be endowed with a metrizable topology whose sequential convergence coincides with Kuratowski convergence. We shall refer to this topology as the Kuratowski topology, since it is the unique metrizable topology that induces Kuratowski convergence.
			Kuratowski convergence has good compactness properties:
			\begin{theorem}[Blaschke's Theorem]
				Let $(F_n)_{n\in \mathbb N}$ be a sequence of closed sets. Then there exists some subsequence which converges in the Kuratowski sense.
			\end{theorem}
			We recall that for an equibounded sequence of closed sets Kuratowski convergence and convergence with respect to Hausdorff metric are equivalent.
            The following useful result allows to interpret Kuratowski convergence as Hausdorff convergence on (some) bounded open sets. Since we were not able to find a reference for this, we present a simple proof.
			\begin{lemma}\label{lemma: local hausdorff}
				Let $F\subset\mathbb R^N$ a closed set and $(F_n)_{n\in \mathbb N}$ be a sequence of closed sets in $\mathbb R^N$ such that $F_n\to F$ in the sense of Kuratowski.
				Let $G$ be a bounded open set such that $F\cap G\neq \emptyset$ and assume that $\widebar{F\cap G}=F\cap \widebar G$. 
                Then  $F_n\cap\widebar{G}\to F\cap\widebar{G}$ in the Hausdorff metric. This in particular holds if both $F, G$ are convex.
			\end{lemma}
			\begin{proof}
    Let us show that any subsequence of $F_n\cap\widebar{G}$ has a further subsequence converging in the Hausdorff metric to $F\cap \widebar{G}$.
    Let $(F_{n_j}\cap\widebar{G})_{j\in \mathbb N}$ be such a subsequence. By Kuratowski convergence, given $x\in F\cap G$, there exists a sequence $(x_{n_j})_{j\in \mathbb N}$ such that $x_{n_j}\in F_{n_j}$ eventually and $x_{n_j}\to x$. Hence, $F_{n_j}\cap\widebar{G}\neq \emptyset$ eventually. By Blaschke's theorem and the fact that $G$ is bounded, there exist a compact set $K\subset\mathbb R^N$ and a further subsequence
    $(F_{n_{j_k}}\cap\widebar{G})_{k\in \mathbb N}$ such that $F_{n_{j_k}}\cap\widebar{G}\to K$ in the Hausdorff metric. Let us show that $K=F\cap \widebar{G}$. Let $(x_{n_{j_k}})_{k\in \mathbb N}$ be a sequence such that $x_{n_{j_k}}\in F_{n_{j_k}}\cap\widebar{G}$, and let $x$ be any limit point of that sequence. By Kuratowski convergence of $(F_n)_{n\in \mathbb N}$, it follows that $x\in F$. Moreover, since $\widebar{G}$ is closed, we have $x\in\widebar{G}$, and hence $x\in F\cap\widebar{G}$. This implies that $K\subset F\cap\widebar{G}$. Let $x\in F\cap G$. By Kuratowski convergence, there exists a sequence $(x_{n_{j_k}})_{k\geq1}$ such that $x_{n_{j_k}}\in F_{n_{j_k}}$ and $x_{n_{j_k}}\to x$. Since $G$ is open, we have
    $x_{n_{j_k}}\in F_{n_{j_k}}\cap G\subset F_{n_{j_k}}\cap\widebar{G}
    $
    eventually. This implies
    \begin{equation}
        F\cap G\subset K\subset F\cap\widebar{G}.
    \end{equation} Since $K$ is compact and, by assumption, $\widebar{F\cap G}=F\cap\widebar{G}$, the thesis follows. Finally, it follows easily from the properties of convex sets that if $F$ is a closed convex set and $G$ is an open convex set, then $\widebar{F\cap G}=F\cap\widebar{G}$.
\end{proof}
		
			\noindent We start the study of asymptotic geometry of convex sets by recalling some structural properties proved in \cite[Lemma 3.1]{LeRiVe22}
			\begin{lemma}
				Let $C\subset\mathbb R^N$ be an unbounded closed convex set and let $(x_n)_{n\in \mathbb N}\subset C$ be a diverging sequence. Then, modulo subsequences, there exists a closed convex set $K$ such that $C-x_n\to K$ in the Kuratowski sense. Moreover, $K$ contains a line and thus it is a convex cylinder.
			\end{lemma}
			Following the terminology in \cite{LeRiVe22}, we say:
			\begin{definition}\label{def: asymptotic cylinders}
				Let $C\subset\mathbb R^N$ be a closed convex set.
                We say that $K\subset \mathbb R^N$ is an \emph{asymptotic cylinder} for $C$ if there exists $(x_n)_{n\in \mathbb N}\subset C$ such that $|x_n|\to +\infty$ and $C-x_n\to K$ in the sense of Kuratowski.
                We denote by $\mathcal{K}(C)$ the set of asymptotic cylinders of $C$.
			\end{definition}
			We have introduced the notion of asymptotic cylinder in order to provide a setting in which look for minimizers of the capillarity energy. We know that convexity is preserved by convergence in the Hausdorff metric and even in the convergence of Kuratowski. What is not preserved, in general,  is the property of being a convex body, as the following example shows:
			\begin{example}[\cite{LeRiVe22}]\label{es: loss of dim}
				Let $Q\coloneqq\{(x, y, z)\in\mathbb R^3:\, x^2+y^2\leq 1 , \;z\geq 0\}$ the unitary half cylinder and $P\coloneqq \{(x, y, z)\in\mathbb R^3:\, z=(x-1)^2 ,\; y=0,\; x\geq 1\}$ a parabolic curve. Let $C$ be the closed convex envelope of $Q\cup P$. Let us fix a diverging sequence $(x_n)_{n\in \mathbb N}\subset [1, +\infty)$ and set $p_n=(x_n, 0, (x_n-1)^2)$ the point on $P$ with abscissa $x_n$ for all $n\in \mathbb N$.
				Then any asymptotic cylinder of $C$ defined by subsequences of $p_n$ must have dimension at most $2$.
			\end{example}
			The situation depicted in \Cref{es: loss of dim} is not even the worst possible: if we think about $C\times C$ in $\mathbb R^6$, we see that the dimensions lost in the limit may be more than one.
			We conclude with a simple remark on the properties of Kuratowski converging convex bodies.
			\begin{remark}[{\cite[Remarks 2.1, 2.2]{FoFuLeMo23}}]\label{rmk: convergenza convessi}
				Let $C_n$ be a sequence of convex bodies and $C$ be a convex body. Then $C_n\to C$ in the sense of Kuratowski if and only if $\mathbbm{1}_{C_n}\to \mathbbm{1}_{C}$ pointwise a.e.. Moreover, if this holds, $|\mu_{C_n}|=\mathcal H^{N-1} \mres \partial C_n \xrightharpoonup{*} |\mu_{C}|  =\mathcal H^{N-1} \mres \partial C$ in duality with $C^0_c(\mathbb R^N)$. In particular, if the sets are equibounded, $\mathbbm{1}_{C_n}\to \mathbbm{1}_{C}$ strictly.
			\end{remark}

            It will be useful in the following to observe that for a closed convex set $C$, an asymptotic cylinder of another asymptotic cylinder of $C$ is itself an asymptotic cylinder of $C$. 

\begin{proposition}\label{prop: i cilindri asint dei cilindri asint sono cilindri asint}
    Let $C$ be a closed convex set. If $K \in \mathcal{K}(C)$ and $L \in \mathcal{K}(K)$, then $L \in \mathcal{K}(C)$.
\end{proposition}

\begin{proof}
    Throughout the proof, all limits of sets are taken in the Kuratowski sense. The argument is a standard diagonalization. We can find a diverging sequence $(x_n)_{n\in\mathbb{N}}$ in $C$ and a diverging sequence $(y_m)_{m\in\mathbb{N}}$ in $K$ such that
    \begin{equation}
        K = \lim_{n \to +\infty} (C - x_n) \, \, \, \, \, \text{and} \, \, \, \, \, L = \lim_{m \to +\infty} (K - y_m).
    \end{equation}
    Since the collection of closed sets equipped with the Kuratowski topology is first countable, we can find $(V_k)_{k \in \mathbb{N}}$, a countable basis of neighborhoods of $L$. For every $k \in \mathbb{N}$, we can find a strictly increasing sequence of indices $(m_k)_{k\in\mathbb{N}}$ such that $K - y_{m_k} \in V_k.$ By the convergence $C - x_n \to K$, we can then choose a strictly increasing sequence of indices $(n_k)_{k\in\mathbb{N}}$ such that
    $C - x_{n} - y_{m_k} \in V_k$ for every $n \geq n_k$. Since $y_{m_k} \in K = \lim\limits_{l \to +\infty} (C - x_{n_l})$, for every $k$ there exists a sequence $(z_l^k)_{l\in\mathbb{N}}$ in $C$ such that $y_{m_k} = \lim\limits_{l \to +\infty} (z_l^k - x_{n_l}).$  By a diagonal argument, we can find a subsequence $(n_{l_k})_{k\in\mathbb{N}}$ of indices such that, letting $z_k = z^{k}_{l_k}$, we have
    \begin{align}
        z_k - (x_{n_{l_k}} + y_{m_k}) \to 0. \label{eq1: diagonalizzazione conv kuratowski}
    \end{align}
    Since the Kuratowski topology is Hausdorff and $C - x_{n_{l_k}} - y_{m_k} \in V_k$ for every $k$, we have
    \begin{align}
        L = \lim_{k \to +\infty} (C - x_{n_{l_k}} - y_{m_k}). \label{eq2: diagonalizzazione conv kuratowski}
    \end{align}
    
    Using \eqref{eq1: diagonalizzazione conv kuratowski} and \eqref{eq2: diagonalizzazione conv kuratowski}, we end up with $C - z_k \to L$, which proves that $L \in \mathcal{K}(C)$.
\end{proof}

	\section{\texorpdfstring{$\Gamma$}{Gamma}--convergence of the functionals}\label{section: Gamma}
		A key tool in our study of minimizing sequences is the $\Gamma$--convergence of the capillarity energies when the underlying obstacles do converge with respect to the Hausdorff metric. in order to prove this, we split the argument in two parts, namely we distinguish whether the dimension is preserved or not in the limit.
		\subsection{No loss of dimension}
			When the full dimensionality is preserved in the limit, more and simpler tools are at disposal to compare sets outside the converging obstacles (and their energy) to ones outside the limit one.
			We start by proving a simple lemma that will be useful in this regard: up to homotheties (with dilation constant which goes to 1), we can always assume that the convergence of the obstacles is ``monotone'', either upwards or downward as needed.
			\begin{lemma}\label{lemma: aggiustamento convessi}
				Let $C\subset\mathbb{R}^N$ be a bounded convex body such that $0\in {\rm int}(C)$. Let $(C_n)_{n\in \mathbb N}$ be a sequence of convex bodies such that
				$C_n\to C$ in the Hausdorff metric.
				Then there exists a sequence $(\varepsilon_n)_{n\in \mathbb N}$ such that:
				\begin{itemize}
					\item $\varepsilon_n\to 0$;
					\item $(1-\varepsilon_n) C_n\Subset {\rm int}(C)$ eventually;
					\item $C\Subset (1+\varepsilon_n)C_n$ eventually.
				\end{itemize}
			\end{lemma}
			\begin{proof}
				By the Hausdorff convergence and convexity, there exists $r>0$ such that
				\begin{equation}
					B(0,r) \Subset {\rm int}(C) \quad \text{and} \quad B(0, r) \Subset {\rm int}(C_n)\subset C_n 
				\end{equation}
				eventually.
				By the definition of Hausdorff convergence, there exist an infinitesimal sequence $(a_n)_{n\in \mathbb N}$ such that
				\begin{equation}
					C_n\Subset C+B(0, a_n)\subset \left( 1+\frac{a_n}{r}\right)C
				\end{equation}
				eventually and 
				\begin{equation}
					C\Subset C_n+ B(0, a_n) \subset \left( 1+\frac{a_n}{r}\right)C_n
				\end{equation}
				eventually.
				By setting $\varepsilon_n=\max\left\{\frac{a_n}{r} , 1-\left( 1+\frac{a_n}{r}\right)^{-1}\right\}$, the thesis follows.
			\end{proof}
			This dilation argument is key in the proof of the $\Gamma$--convergence result, as it allows to cut sets in a controlled fashion, providing an approximating sequence for the $\Gamma-\limsup$ and a nicer sequence for the $\Gamma-\liminf$.
			\begin{theorem}
				Let $C$ be a bounded convex body and let $(C_n)_{n\in\mathbb{N}}$, be a sequence of convex bodies converging to $C$ in the Hausdorff metric. Then the following statements hold:
				\begin{itemize}
					\item Let $E\subset\mathbb{R}^N\setminus C$ be a set of finite perimeter such that $|E|<+\infty$. Then there exists a sequence $(E_n)_{n\in \mathbb N}$ of sets of finite perimeter such that $E_n\to E$ strictly, $E_n \subset \mathbb{R}^N \setminus C_n$ and 
					\begin{equation}\label{eq: limsup}
						\mathcal F_{C_n}(E_n)\to \mathcal{F}_C(E).
					\end{equation}
					Moreover, if $E\Subset B(x, R)$ for some $x\in \mathbb R^N$, $R>0$, the sets $E_n$ can be chosen to be such that $E_n\Subset B(x, R)$.
					\item Let $E\subset\mathbb{R}^N\setminus C$ be a set of finite perimeter such that $|E|<+\infty$ and $(E_n)_{n\in \mathbb N}$ be a sequence of sets of finite perimeter such that $E_n\to E$ in $L^1$ and $E_n \subset \mathbb{R}^N \setminus C_n$. Then
					\begin{equation}\label{eq: liminf}
						\mathcal F_C(E)\leq \liminf_{n \to +\infty} \mathcal F_{C_n}(E_n).
					\end{equation}
				\end{itemize}
			\end{theorem}
			\begin{proof}
				Let's start by proving the $\liminf$ inequality. 	
				Let $(E_n)_{n\in \mathbb{N}}$ be a sequence of sets of finite perimeter such that $E_n\to E$ in $L^1$ and $E_n \subset \mathbb{R}^N \setminus C_n$.
                In order to more easily prove \eqref{eq: liminf}, we distinguish the case in which $\lambda$ is bigger or equal than 0 from the case in which it is smaller. 
                First, by \Cref{lm: approssimazione,lemma: aggiustamento convessi}, we may suppose that $C_n\Subset C$ and that $P(E_n;\partial C)=0$ for all $n$. 
                Let us set $F_n=E_n\setminus C$ and $\tilde{F}_n=E_n\cap C$. Notice that $F_n\to E$ in $L^1$. 
                Assume that $\lambda\geq 0$. By \Cref{prop: bulk contact inequality} on $\tilde F_n$ we get that 
				\begin{equation}
					P(E_n;\partial C_n)\leq P(\tilde F_n;\mathbb R^N\setminus C_n)=P(F_n; \partial C)+P(E_n; {\rm int}(C)\setminus C_n).
				\end{equation}
				This yields
				\begin{align}
					\mathcal F_C(F_n)&=P(E_n;\mathbb R^N\setminus C)- \lambda P(F_n;\partial C)\\
					&\leq P(E_n;\mathbb R^N\setminus C)-\lambda P(E_n;\partial C_n)+\lambda P(E_n; {\rm int}(C)\setminus C_n)\\
					&\leq P(E_n;\mathbb R^N\setminus C)+ P(E_n; {\rm int}(C)\setminus C_n)-\lambda P(E_n;\partial C_n)= \mathcal F_{C_n}(E_n).
				\end{align}
				Since the capillarity energy functional is lower semicontinuous (\Cref{prop: semicontinuity}), we get
				\begin{equation}
					\mathcal F_C(E)\leq \liminf_n \mathcal F_C(F_n)\leq \liminf_n \mathcal F_{C_n}(E_n).
				\end{equation}
				
				Finally, assume $\lambda<0$. Let us apply \Cref{lemma: local bulk wet estimate} on the sets $\tilde F_n\subset C$ with the role of the obstacle played by $\mathbb R^N\setminus C$. Fix $\delta, \eta>0$ and set $I_\delta\coloneqq \{x\in {\rm int}(C)\text{ s.t. }d(x, \mathbb R^N\setminus C)<\delta\}$.
				For $n\in \mathbb{N}$ we get:
				\begin{align}
					P(F_n; \partial C)=P(\tilde F_n; \partial C)&\leq P(E_n; I_\delta)+m(\delta, \eta)| E_n\cap C|+\eta\\
					&\leq P(E_n;\partial C_n)+P(E_n; {\rm int} (C)\setminus C_n)\\
					&+ m(\delta, \eta)|E_n\cap C|+\eta.
				\end{align}
				By semicontinuity, this yields
				\begin{align}
					\mathcal F_C(E)&\leq \liminf_n \mathcal F_C(F_n)\\
					&=\liminf_n\left[P(E_n; \mathbb R^N\setminus C)-\lambda P(F_n; \partial C)\right]\\
					&\leq \liminf_n \Big[ P(E_n; \mathbb R^N\setminus C)\\
					&\quad-\lambda \left( P(E_n;\partial C_n)+P(E_n; {\rm int} (C)\setminus C_n)+ m(\delta, \eta)| E_n\cap C|+\eta\right)\Big]\\
					&\leq \liminf_n \Big[ \left(P(E_n; \mathbb R^N\setminus C)+P(E_n; {\rm int} (C)\setminus C_n)\right) -\lambda P(E_n;\partial C_n)\\
					&\quad+ m(\delta, \eta)| E_n\cap C|+\eta\Big]\\
					&=\liminf_n\mathcal F_{C_n}(E_n)+\eta.
				\end{align}
				By the generality of $\eta> 0$, we get \eqref{eq: liminf}.

                Let us now prove the $\limsup$ inequality.
                Notice that, by the general theory of $\Gamma$-convergence, see e.g.\ \cite[Proposition 8.3]{dalmasoIntroductionGConvergence1993}, it is enough to prove it for every subsequence of a given subsequence.
                By \Cref{lemma: aggiustamento convessi}, we may assume, up to some dilations and a subsequence, that $C\subset C_{n+1}\subset C_{n}$ for all $n$. 
                Moreover, by \Cref{rmk: convergenza convessi}, $|\mu_{ C_n}|  \xrightharpoonup{*} |\mu_{ C}|$ in duality with $C^0_c(\mathbb R^N)$. 
                Finally, by \Cref{lm: approssimazione}, we may assume that $E=F\setminus C$, where $F$ is a bounded open set with $C^\infty$ boundary such that $P(F,\partial C)=P(F, \partial C_n)=0$ for all $n$, so, in particular, $|\mu_C|(\partial F)=0$.
                Let $E_n=F\setminus C_n$, then $P(E_n; C_n)=|\mu_{C_n}|(F)$ (see \cite[Theorem 16.3]{Ma12}). By \Cref{rmk: convergenza convessi}, $E_n\to E$ in $L^1$.
                Moreover, since $|\mu_C|(\partial F)=0$, it follows $P(E_n; C_n)=|\mu_{C_n}|(F)\to |\mu_C|(F)=P(E;C)$. Finally, 
				$P(E;\mathbb R^N\setminus C)-P(E_n;\mathbb R^N\setminus C_n)=P(F;C_n\setminus C)$: since $C_n\setminus C\searrow\emptyset$, the difference decreases to 0, yielding \eqref{eq: limsup}. 
			\end{proof}
			In particular, since the perturbation is $L^1$--continuous, the penalized functional $\Gamma$--converges:
			\begin{corollary}\label{cor: gamma limite dimensione piena}
				Let $\Lambda>0$ and $C\subset \mathbb R^N$ be a bounded convex body and $(C_n)_{n\in \mathbb N}$ be a sequence of bounded convex bodies converging to $C$ in the Hausdorff metric. Then:
				\begin{itemize}
					\item for every set $E\subset\mathbb R^N$ of finite perimeter and every sequence $(E_n)_{n\in \mathbb N}$ such that $E_n\to E$ in $L^1$ and $E_n \subset \mathbb{R}^N \setminus C_n$, 
					\begin{equation}
						 \mathcal F^{v}_{C, \Lambda}(E)\leq \liminf_n\mathcal F^{v}_{C_n, \Lambda}(E_n);
					\end{equation}
					\item for every set $E\subset\mathbb R^N$ of finite perimeter and volume there exist a sequence $(E_n)_{n\in \mathbb N}$ such that $E_n\to E$ strictly, $E_n \subset \mathbb{R}^N \setminus C_n$ and 
					\begin{equation}
						\mathcal F^{v}_{C, \Lambda}(E)=\lim_n \mathcal F^{v}_{C_n, \Lambda}(E_n).
					\end{equation}
				\end{itemize}
				In other words, $\mathcal F^{v}_{C, \Lambda}=\Gamma \text{-}\lim\limits_{n\to +\infty} \mathcal F^{v}_{C_n, \Lambda}$ in $L^1$.
			\end{corollary}
		\subsection{Loss of dimension}
			When the full dimensionality is lost in the limit, it is no longer possible to resort to \Cref{lemma: aggiustamento convessi}: in particular there is no hope in having $C_n\nearrow C$, but it seems not clear even if it is possible to have (up to homotheties which tend to the identity) $C_n\searrow C$. It is easy to note that if more than one dimension is lost in the limit, the wet part of the functional converges uniformly to 0, so it is immediate that $\mathcal F_{C_n}\xrightarrow{\Gamma}P$. Let us focus on the case $\operatorname{dim} C=N-1$.
			 
			In order to study what happens when one single dimension is lost in the limit, we have to somehow describe effectively the collapse of the perimeter measure on the two sides of the limit. While it is immediate to see that the limit of the Gauss--Green measures vanishes in the limit due to cancellation, the convexity allows to show that the perimeter measures do converge to exactly the double of the $(N-1)$--Hausdorff measure of the limit. In order to describe what happens in the limit, then, it is useful to split the convex bodies $C_n$ in two halves whose Hausdorff measure converge to those of $C$, while the corresponding Gauss--Green measures converge to cancelling limits.
			
			A na\"ive idea would be to cut $C_n$ in half by intersecting with a sequence of hyperplanes $\{x_N=c_n\}$ for suitably chosen $c_n$, but a sequence of progressively rotating and thinning rectangles $C_n$ converging to a segment in $\mathbb R^2$ provides an easy counterexample. A further refining of this idea would be to find a sequence of suitably tilted and translated hyperplanes converging to $\{x_N=0\}$ and use those to split $C_n$: while this works in $\mathbb R^2$, it is not at all clear if it is possible to perform such a choice in higher dimension. 
			
			A perhaps unexpectedly simpler way out is to consider a ``nonlinear'' splitting. Given $K\subset \mathbb R^N$ a compact convex body, we can consider its projection $\tilde{K}=p(K)$, where $p:\mathbb R^N\to \{x_N=0\}$ is the orthogonal projection. In what follows, we will denote by $\partial \tilde{K}$ the boundary of $\tilde{K}$ in the topology of $\{x_N=0\}$. 
			By convexity, there exist $f^u, f^l:\tilde{K}\to \mathbb R$ such that $K=\{f^l(x')\leq x_n\leq f^u(x')\}$, with $f^l$ convex and $f^u$ concave. In particular, $\partial K$ can be split in two caps $K^u\coloneqq \{(x', f^u(x')): x'\in\tilde{K}\}$ and $K^l\coloneqq \{(x', f^l(x')): x'\in\tilde{K}\}$ and possibly a lateral component $K^s=\partial K\setminus (K^u\cup K^l)$. Finally, in view of the study of perimeters, notice that $K^s\subset \partial \tilde{K}\times \mathbb R$, so $\mathcal H^{N-1}(K^s)$ can easily be bounded in terms of $d_{\mathcal{H}}(K, \tilde{K})$ and $\mathcal H^{N-2}(\partial\tilde{K})$.
			
			In order to study the limit behaviour of the wet part, we will ``straighten'' the boundary of the converging obstacles. For $\varepsilon>0$ and $C\subset \{x_N=0\}$, set $C_\varepsilon\coloneqq \{x\in C: d(x, \{x_N=0\}\setminus C)> \varepsilon\}$.
			\begin{lemma}\label{lemma: funzioni calotte} 
				Let $(C_n)_{n\in\mathbb{N}}$ be a sequence of convex bodies converging in the sense of Hausdorff to $C$, a bounded convex body in $\{x_N=0\}$. Then, for each $\varepsilon>0$, it holds:
				\begin{itemize}
					\item $f_n^u, f_n^l$ are eventually well defined on the whole of $C_\varepsilon$ and are concave and convex, respectively;
					\item $f_n^u, f_n^l\to 0$ uniformly on $\tilde{C}_n$ (in particular, eventually on $C_\varepsilon$);
					\item $\operatorname{Lip} f_n^u, \operatorname{Lip} f_n^l\to 0$ in $C_\varepsilon$.
				\end{itemize}
			\end{lemma}
			\begin{proof}
				Notice that since $C_n\to C$ with respect to Hausdorff distance, $\tilde{C}_n\coloneqq p(C_n)\to C$. In particular $C_\varepsilon \subset \tilde{C}_n$ eventually. Moreover, it is immediate to check that $\sup |f^u_n|, \sup |f^l_n|\leq d_{\mathcal{H}}(C_n, C)\to 0$, thus $\operatorname{osc} f_n^u, \operatorname{osc} f_n^l\leq 2 d_{\mathcal{H}}(C_n, C)\to 0$. The third item follows by well known properties of convex functions.
			\end{proof}
			This simple splitting satisfies the properties that we were looking for:
			\begin{corollary}\label{cor: convergence measures}
				Let $C$ be a bounded convex body in $\{x_N=0\}$ and let $(C_n)_{n \in \mathbb{N}}$ be a sequence of convex bodies in $\mathbb R^N$ Hausdorff converging to $C$. Set $\tilde{C_n}=p(C_n)$ and $f_n^u, f_n^l:\tilde{C_n}\to \mathbb R$, $C_n^u, C_n^l, C_n^s$ as above. Then
				\begin{itemize}
					\item for the Gauss--Green measures it holds
					\begin{equation}
						\left(\mu_{C_n}\right) \mres C_n^u, \left(\mu_{C_n}\right)\mres C_n^l\xrightharpoonup{*} \pm e_N \mathcal H^{N-1} \mres C;
					\end{equation}
					\item for the Hausdorff measures of the caps:
					\begin{equation}
						\mathcal H^{N-1}\mres  C_n^u, \mathcal H^{N-1} \mres C_n^l \xrightharpoonup{*} \mathcal H^{N-1}\mres C 
					\end{equation}
					and $\mathcal H^{N-1}(C_n^s)\to 0$.
				\end{itemize}
			\end{corollary}
			\begin{proof}
				Since $\tilde{C_n}\to \tilde{C}$ in Hausdorff distance, by \Cref{rmk: convergenza convessi} it holds that $\mathcal H^{N-2}(\partial \tilde{C}_n)\to \mathcal H^{N-2}(\partial \tilde{C})$, where the boundaries are intended as before in $\{x_N=0\}$. In particular, this implies
				\begin{equation}
					\mathcal{H}^{N-1}(C_n^s)\leq 2\mathcal H^{N-2}(\partial \tilde{C}_n)d_{\mathcal{H}}(C_n, C)\leq 2\mathcal H^{N-2}(\partial \tilde{C})d_{\mathcal{H}}(C_n, C)+o(1),
				\end{equation} which vanishes. The other results follow easily from \Cref{lemma: funzioni calotte}, the definition of the cap functions and coarea formula.
			\end{proof}
			The main tool that we will exploit to pass to the limit in the energies is a family of bilipschitz diffeomorphism in order to ``straighten'' part of the boundaries. By \Cref{lemma: funzioni calotte}, given $\varepsilon>0$ it is easy to build (for $n$ sufficiently large) a sequence $\Phi^\varepsilon_n:\mathbb R^N\to\mathbb R^N$ such that 
			\begin{equation}
				\Phi^\varepsilon_n(x', x_N)=\begin{cases}
					(x', x_N-f_n^l(x'))\qquad &\text{if } x'\in C_\varepsilon\\
					(x', x_N)\qquad &\text{if } x'\notin C,
				\end{cases}
			\end{equation}
			with $\Phi^\varepsilon_n$ being a sequence of bi--Lipschitz maps, uniformly converging to the identity and such that 
			\begin{equation}\label{eq: bilipschitz}
				\operatorname{Lip}(\Phi^\varepsilon_n), \operatorname{Lip}\left((\Phi^\varepsilon_n)^{-1}\right)\to 1
			\end{equation} 
			as $n$ goes to infinity.
			Analogously, $\Psi_n^\varepsilon$ can be built such that the same properties are satisfied, up to the value assumed in $C_\varepsilon\times\mathbb R$, which we set to $\Psi^\varepsilon_n(x', x_N)=(x', x_N-f_n^u(x'))$.
			
			\begin{proposition}\label{prop: gamma limsup perdita dimensione}
				Let $C$ be a bounded convex body in $\{x_N=0\}$ and $E\subset\mathbb{R}^N\setminus C$ a set of finite perimeter such that $|E|<+\infty$. Let $(C_n)_{n\in \mathbb N}$ be a sequence of convex bodies such that $C_n\to C$ in the Hausdorff metric. Then there exists a sequence $(E_n)_{n\in \mathbb N}$ of sets of finite perimeter such that $E_n \subset \mathbb{R}^N\setminus C_n$, $E_n\to E$ strictly and  $\mathcal F_{C_n}(E_n)\to \mathcal{F}_C(E)$. 
				Moreover, if $E\Subset B(x, R)$ for some $x\in \mathbb R^N$, $R>0$, the sets $E_n$ can be chosen to be such that $E_n\Subset B(x, R)$.
			\end{proposition}
			\begin{proof}
                Fix $\varepsilon>0$ and set $C_n^\varepsilon\coloneqq \Phi^\varepsilon_n(C_n)$, $E_n^\varepsilon\coloneqq E\setminus C_n^\varepsilon$.
				Set 
				\begin{align}
						\partial C_n^\varepsilon&= \left( C_n^\varepsilon\cap C_\varepsilon\right)\cup \left(\partial C_n^\varepsilon\cap \left(C_\varepsilon\times(0, +\infty)\right)\right) \cup \left(\partial C_n^\varepsilon\setminus \left(C_\varepsilon\times\mathbb R\right)\right)\\
						&\eqqcolon C_{n, f}^\varepsilon\cup C_{n, c}^\varepsilon\cup C_{n, r}^\varepsilon.
				\end{align} 
				Notice that, for $n$ sufficiently large, $C_{n, f}^\varepsilon=C_\varepsilon$.
                
				By \Cref{lm: approssimazione}, considering the sequence of convex bodies given by $\{x_N \leq 0 \}$ and $(C_n^{\varepsilon}\cap \left(C_\varepsilon\times(0, +\infty)\right))_{n\in\mathbb{N}}$, we may assume that $E\cap\{x_N>0\}=F\cap\{x_N>0\}$, where $F$ is a bounded open set with $C^\infty$ boundary such that $P(F, \{x_N=0\})=0$ and such that $P(F, C_{n,c}^{\varepsilon})=0 $ for every $n\in\mathbb{N}$.

                Let us study $P(E_n^\varepsilon; \partial C_n^\varepsilon)$.
				For what concerns the flat part, by \Cref{teo: trace 1} it holds
				\begin{equation}
					P(E_n^\varepsilon; C^\varepsilon_{n, f})=\int_{C^\varepsilon_{n, f}} \operatorname{Tr}_{\partial C_n^\varepsilon}^+ \mathbbm{1}_{E_n^\varepsilon}\,d\mathcal H^{N-1}= \int_{C_\varepsilon} \operatorname{Tr}_{C}^- \mathbbm{1}_E\,d\mathcal H^{N-1},
				\end{equation} for $n$ large enough, where $C_n^\varepsilon$ is oriented by the external normal and $C_\varepsilon$ by $e_N$.

				By \Cref{cor: convergence measures} and \eqref{eq: bilipschitz}, $\mathcal{H}^{N-1} \mres C^\varepsilon_{n, c}\xrightharpoonup{*} \mathcal{H}^{N-1} \mres C_\varepsilon$.  By \cite[Theorem 16.3]{Ma12} and the Portmanteau Lemma,
				\begin{align}
					\lim_n P(E_n^\varepsilon, C^\varepsilon_{n, c})&=\lim_n \mathcal{H}^{N-1} \mres  C^\varepsilon_{n, c}(F)=\mathcal{H}^{N-1} \mres C_\varepsilon (F)\\
					&=\int_{C_\varepsilon} \operatorname{Tr}_{C}^+ \mathbbm{1}_F\,d\mathcal H^{N-1}=\int_{C_\varepsilon} \operatorname{Tr}_{C}^+ \mathbbm{1}_E\,d\mathcal H^{N-1}.
				\end{align}
				Finally, for the rest of the wet part it is easy to estimate 
				\begin{equation}
					\limsup_n P(E_n^\varepsilon, C^\varepsilon_{n, r})\leq\limsup_n \mathcal H^{N-1}(C^\varepsilon_{n, r})\leq 2 \mathcal{H}^{N-1}(C\setminus C_\varepsilon).
				\end{equation}
				Putting all together, 
				\begin{equation}
					\limsup_n \left|P(E_n^\varepsilon; C_n^\varepsilon)-\int_C \left(\operatorname{Tr}_{C}^+ \mathbbm{1}_E+ \operatorname{Tr}_{C}^- \mathbbm{1}_E\right)\,d\mathcal H^{N-1}\right|=o_\varepsilon(1).
				\end{equation}
				Moreover, by Hausdorff convergence and the properties of $\Phi_n^\varepsilon$, it is immediate to see that $P(E_n^\varepsilon, \mathbb R^N\setminus C_n^{\varepsilon})\xrightarrow{n \to +\infty} P(E, \mathbb R^N\setminus C)$. From here, we conclude that there exists a vanishing sequence $\varepsilon_n\to 0$ such that $\left(\Phi_n^{\varepsilon_n}\right)^{-1}(E_n^{\varepsilon_n})$ does the job.
			\end{proof}
			
			\begin{proposition}\label{prop: Gamma liminf caso perdita di dimensione}
				Let $C$ be a bounded convex body in $\{x_N=0\}$ and $E\subset\mathbb{R}^N\setminus C$ a set of finite perimeter such that $|E|<+\infty$. Let $(C_n)_{n\in \mathbb N}$ be a sequence of convex bodies such that $C_n\to C$ in the Hausdorff convergence sense and $(E_n)_{n\in \mathbb N}$ be a sequence of sets of finite perimeter such that $E_n \subset \mathbb{R}^N\setminus C_n$ and $E_n\to E$ in $L^1$. Then 
				\begin{equation}
					\mathcal F_C(E)\leq \liminf_{n \to +\infty} \mathcal F_{C_n}(E_n).
				\end{equation}
			\end{proposition}
			\begin{proof}
				Fix $\varepsilon>0$. Set $C_n^{+, \varepsilon}\coloneqq \Phi^\varepsilon_n(C_n)$, $E_n^{+, \varepsilon}\coloneqq \Phi^\varepsilon_n(E_n)$, $C_n^{-, \varepsilon}\coloneqq \Psi^\varepsilon_n(C_n)$, $E_n^{-, \varepsilon}\coloneqq \Psi^\varepsilon_n(E_n)$. In order to simplify some computations, with a slight abuse of notation, we will use the notation
                \begin{equation}
                	\mathcal F_K(E; B):=P(E; B\setminus K)-\lambda P(E; \partial K\cap B)
                \end{equation}
                for the localized functional also for sets $K\subset\mathbb{R}^N$ with $\dim(K)=N$ and Borel sets $B$.
				Notice that, by \eqref{eq: bilipschitz} it follows
				\begin{equation}
					\liminf_{n \to +\infty}\mathcal{F}_{C_n}(E_n; \{(x', x_N):x'\in C_\varepsilon, x_N\geq f_n^u(x')\})=\liminf_{n \to +\infty} \mathcal{F}_{C_n^{-, \varepsilon}}(E_n^{-, \varepsilon}; C_\varepsilon\times [0,+\infty))
				\end{equation}
				and 
				\begin{equation}
					\liminf_{n \to +\infty} \mathcal{F}_{C_n}(E_n; \{(x', x_N):x'\in C_\varepsilon, x_N\leq f_n^l(x')\})= \liminf_{n \to +\infty} \mathcal{F}_{C_n^{+, \varepsilon}}(E_n^{+, \varepsilon}; C_\varepsilon\times(-\infty,0]).
				\end{equation}
				By the fact that
				\begin{align}
					 \mathcal{F}_{C_n}(E_n; C_\varepsilon\times\mathbb R)&= \mathcal{F}_{C_n}(E_n; \{(x', x_N):x'\in C_\varepsilon, x_N\geq f_n^u(x')\})\\
					 &\qquad+\mathcal{F}_{C_n}(E_n; \{(x', x_N):x'\in C_\varepsilon, x_N\leq f_n^l(x')\}), 
				\end{align}
				it follows
				\begin{align}
					\liminf_{n \to +\infty} \mathcal{F}_{C_n}(E_n; C_\varepsilon\times\mathbb R)\geq& \liminf_{n \to +\infty} \mathcal{F}_{C_n^{-, \varepsilon}}(E_n^{-, \varepsilon}; C_\varepsilon\times[0,+\infty))\\
					&\qquad+\liminf_{n \to +\infty} \mathcal{F}_{C_n^{+, \varepsilon}}(E_n^{+, \varepsilon}; C_\varepsilon\times(-\infty,0]).
				\end{align}
				Notice moreover that
				$E_n^{+, \varepsilon}, E_n^{-, \varepsilon}\to E$ in $L^1$.
				
				By computations similar to those in the proof of \Cref{prop: semicontinuity}, it follows that
				\begin{equation}
					P(E; C_\varepsilon\times(0, +\infty))-\lambda\int_{C_\varepsilon} \operatorname{Tr}_{C}^+ \mathbbm{1}_E \, d\mathcal H^{N-1}\leq \liminf_{n \to +\infty} \mathcal{F}_{C_n^{-, \varepsilon}}(E_n^{-, \varepsilon}, C_\varepsilon\times[0, +\infty)).
				\end{equation}
				Repeating these arguments on the other side of $C$ we get 
				\begin{equation}
					P(E; C_\varepsilon\times(-\infty, 0))-\lambda\int_{C_\varepsilon} \operatorname{Tr}_{C}^- \mathbbm{1}_E \, d\mathcal H^{N-1}\leq \liminf_{n \to +\infty} \mathcal{F}_{C_n^{+, \varepsilon}}(E_n^{+, \varepsilon}, C_\varepsilon\times(-\infty, 0])
				\end{equation}
				where $C$ is always oriented by $e_N$. Finally, notice that
				\begin{equation}
					P\left(E; \mathbb R^N\setminus \left(\bar{C_\varepsilon}\times\mathbb R\right)\right)\leq \liminf_{n \to +\infty} P\left(E_n; \mathbb R^N\setminus \left(\bar{C_\varepsilon}\times\mathbb R\right)\right),
				\end{equation}
				and that, by \Cref{cor: convergence measures}, we have
				\begin{equation}
					\limsup_{n \to +\infty} P\left(E_n; C_n\setminus\left( C_\varepsilon\times\mathbb R\right)\right) \leq\limsup_{n \to +\infty} \mathcal H^{N-1}  (\partial C_n\setminus\left( C_\varepsilon\times\mathbb R)\right)=o_\varepsilon(1)
				\end{equation}
				and that, similarly, 
				\begin{equation}
					\int_{C\setminus C_\varepsilon} \left(\operatorname{Tr}_{C}^+ \mathbbm{1}_E+ \operatorname{Tr}_{C}^- \mathbbm{1}_E\right) \, d\mathcal H^{N-1}\leq 2\mathcal{H}^{N-1}(C\setminus C_\varepsilon)=o_\varepsilon(1).
				\end{equation}
				Summing up, we get
				\begin{equation}
					\mathcal{F}_C(E)\leq \liminf_{n \to +\infty} \mathcal F_{C_n}(E_n)+o_{\varepsilon}(1).
				\end{equation}
				By the arbitrariness of $\varepsilon$, the thesis follows.
			\end{proof}
			In conclusion, we can now state the $\Gamma$--convergence result in full generality. In particular, we may also allow the volumes to vary, namely $v_n \to v >0$, without affecting the result:
			\begin{theorem}\label{teo: gamma conv}
				Let $\Lambda>0$, $C\subset \mathbb R^N$ be a compact convex set, $(C_n)_{n\in \mathbb N}$ be a sequence of closed convex sets converging to $C$ in the Hausdorff metric and let $v_n \to v >0$ be a sequence of positive numbers. Then:
				\begin{itemize}
					\item for every set $E\subset\mathbb R^N$ of finite perimeter and every sequence $(E_n)_{n\in \mathbb N}$ such that $E_n \subset \mathbb{R}^N\setminus C_n$ and $E_n\to E$ in $L^1$, 
					\begin{equation}
						\mathcal {F}^{v}_{C, \Lambda}(E)\leq \liminf_n\mathcal {F}^{v_n}_{C_n, \Lambda}(E_n);
					\end{equation}
					\item for every set $E\subset\mathbb R^N$ of finite perimeter and finite volume there exists a sequence $(E_n)_{n\in \mathbb N}$ such that $E_n \subset \mathbb{R}^N\setminus C_n$ and $E_n\to E$ strictly and 
					\begin{equation}
						\mathcal{F}^{v}_{C, \Lambda}(E)=\lim_n \mathcal{F}^{v_n}_{C_n, \Lambda}(E_n).
					\end{equation}
					Moreover, if $E\Subset B(x, R)$ for some $x\in \mathbb R^N$, $R>0$, the sets $E_n$ can be chosen to be such that $E_n\Subset B(x, R)$.
				\end{itemize}
                In other words, $\mathcal{F}^{v}_{C, \Lambda}=\Gamma\text{-}\lim\limits_{n \to +\infty} \mathcal{F}^{v_n}_{C_n, \Lambda}$, where the $\Gamma$-limit holds both with respect to the
                $L^1$ topology and with respect to the strict topology in $BV$.
			\end{theorem}
            \begin{proof}
                If $\dim C<N-1$, we have that  $\mathcal H^{N-1}(\partial C_n)\to 0$, so the thesis is immediate.
                Suppose that $\dim (C) \geq N-1$. By semicontinuity of the dimension, then, $\liminf\limits_{n \to +\infty} \dim(C_n)\geq N-1$. 
                Up to partitioning the set of indices, it is enough to prove the statement in the two cases where $\dim(C_n)\equiv N-1$ and $\dim(C_n)\equiv N$. In the case where $\dim(C_n)\equiv N$, the thesis follows from \Cref{cor: gamma limite dimensione piena} and \Cref{prop: gamma limsup perdita dimensione,prop: Gamma liminf caso perdita di dimensione}.

                Let us finally deal with the case where $\dim(C_n)\equiv N-1=\dim(C)$.
                Let us start by proving the $\liminf$-inequality.
                Let $E$ be a set of finite perimeter and $(E_n)_{n\in \mathbb N}$ be a sequence of sets of finite perimeter such that $E_n\to E$ in $L^1$ and $E_n \subset \mathbb{R}^{N}\setminus C_n$.
                Without loss of generality, we can assume that $\sup \mathcal F_{C_n, \Lambda}^v(E_n)<+\infty$.
                
                For $h>0$, let $C_{n, h}=C_n\times[-h, h]$.
                By \Cref{prop: gamma limsup perdita dimensione} and a diagonal argument, there exists a vanishing sequence $(h_n)_{n\in \mathbb N}$ and a sequence $n\mapsto \tilde E_n\subset \mathbb{R}^N \setminus C_{n, h_n}$ such that 
                \begin{equation}\label{eq: approssimazione limsup}
                    |\mathcal F_{C_n, \Lambda}^v(E_n)-\mathcal F_{C_{n, h_n}, \Lambda}^v(\tilde E_n)|\to 0.
                \end{equation}
                Now, since $h_n\to 0$, $C_{n, h_n}\to C$ in the Hausdorff sense.
                In particular, by \Cref{prop: Gamma liminf caso perdita di dimensione} and \eqref{eq: approssimazione limsup}, 
                \begin{equation}
                    \liminf_n \mathcal F_{C_n, \Lambda}^v(E_n)= \liminf_n \mathcal F_{C_{n, h_n}, \Lambda}^v(\tilde E_n)\geq  \mathcal F_{C, \Lambda}^v(E).
                \end{equation}

                Finally, let us prove the $\limsup$-inequality.
                Up to a sequence of isometries converging to the identity, we can assume that $C_n\subset \{x_N=0\}$ for all $n$.
                Moreover, up to a sequence of rescalings as in \Cref{lemma: aggiustamento convessi}, we can assume that $C\subset C_n$ for all $n$.
                Notice that, by the Hausdorff convergence, this implies that $\mathcal H^{N-1}(C_n\setminus C)\to 0$. 
                Let $E$ be a set of finite perimeter and volume.
                Let $E_n=E\setminus C_n$.
                Then
                \begin{align}
                    \mathcal F_{C_n}(E_n)&= P\left(E;\mathbb{R}^N\setminus {C_n}\right)-\lambda \int_{C_n} \left(\operatorname{Tr}_{\{x_N=0\}}^+ \mathbbm{1}_E+ \operatorname{Tr}_{\{x_N=0\}}^- \mathbbm{1}_E\right)\,d\mathcal H^{N-1}\\
                    &\to P\left(E;\mathbb{R}^N\setminus C\right)-\lambda \int_C \left(\operatorname{Tr}_{C}^+ \mathbbm{1}_E+ \operatorname{Tr}_{C}^- \mathbbm{1}_E\right)\,d\mathcal H^{N-1}= \mathcal{F}_{C}(E),
                \end{align}
                concluding the proof.
            \end{proof}
            
	\section{Existence of generalized minimizers}\label{section: existence}
		In the following, let us fix $C\subset\mathbb R^N$ be a closed convex set, $v>0$, $\lambda\in(-1, 1)$ and $\Lambda>0$.
		We are finally ready to prove the existence of generalized minimizers for the capillarity energy. 
        In order to explain the precise notion of minimizers we are looking for, let us recall \Cref{def: asymptotic configuration}.
		\begin{definition}
			Let $C\subset\mathbb R^N$ be a closed convex set.
            We say that $(\mathbf{E}, \mathbf{C})$ is an \emph{asymptotic configuration} for $C$ if there exist $k\in\mathbb{N} \cup \{ 0 \}$, $C_1, \ldots, C_k$ asymptotic cylinders for $C$, $E_0\subset \mathbb R^N\setminus C$, $E_j\subset\mathbb R^N\setminus C_j$, $j=1, \ldots, k$ sets of finite perimeter such that \begin{equation}
				\mathbf E=(E_0, \ldots, E_k), \qquad \mathbf C=(C,C_1, \ldots, C_k).
			\end{equation}
			We say that $(\mathbf{E}, \mathbf{C})$ has length $k$ and volume $|\mathbf E|=\sum_{j=0}^k |E_j|$.
            
            In order to avoid trivial situations, we assume that $|E_j|>0$ for $j>0$.
		\end{definition}
		
		In order to prove the existence, we recall the picture coming from \Cref{cor: esistenza}, \Cref{prop: density est} (cfr. \cite[Lemma 3.6]{FoFuLeMo23}) and \Cref{cor: stima sulle comp connesse e loro diametro}.
		Let us fix $v>0$ and $\Lambda>0$.
		Up to the choice of a suitable sequence of monotonically diverging radii $(R_n)_{n\in \mathbb N}$, there exist $0\leq M_1\leq M\in\mathbb N$, $d>0$ and a sequence of sets of finite perimeter $(F_n)_{n\in \mathbb N}$ such that:
		\begin{itemize}
			\item $\mathcal {F}^{v}_{C, \Lambda}(F_n)=\min\limits_{F\subset B(x, R_n)\setminus C} \mathcal{F}^{v}_{C, \Lambda}(F)$;
			\item $F_n$ has exactly $M$ connected components $F_n^{(1)}, \ldots, F_n^{(M)}$ and \begin{equation}
				\mathcal{F}_C(F_n)=\sum_{j=1}^M \mathcal{F}_C(F_n^{(j)});
			\end{equation}
			\item each connected component $F_n^{(j)}$ is such that $\operatorname{diam}(F_n^{(j)})<d$;
			\item for all $1\leq j\leq M_1$, there exist diverging sequences $\left(x_n^{(j)}\right)_{n\in\mathbb{N}}\subset \partial C$ such that $x_n^{(j)}\in \widebar{F_n^{(j)}}\cap \partial C$ for all $n$ and $j$.
			\item for all $M_1+1\leq j\leq M$, $\widebar{F_n^{(j)}}\cap \partial C$ is uniformly bounded in $n$.
		\end{itemize}
        \begin{remark}
            One might expect the existence of connected components $\Omega$ such that $\widebar{\Omega}\cap C=\emptyset$.
            However, this cannot happen, as it can be shown arguing as in \cite[Theorem 3.2, step 5]{FoFuLeMo23} (see also \Cref{cor: le componenti devono toccare il convesso}).
        \end{remark}
       By a standard approximation argument, we see that 
        \begin{align}
            \lim\limits_{n \to +\infty} \mathcal{F}^{v}_{C, \Lambda}(F_n)=\inf \mathcal{F}^{v}_{C, \Lambda}(F), \label{eq: equazione degli F_n in Section 4}
        \end{align}
        where $F\subset\mathbb R^N\setminus C$ varies among all sets of finite perimeter.
		
		In analogy with what has been done with $\mathcal{F}_C$, we introduce the penalized functional $ \mathcal G^{v}_{C, \Lambda}$ defined as
		\begin{equation}
			 \mathcal G^{v}_{C, \Lambda}(\mathbf{E}; \mathbf{C})= \mathcal G_C(\mathbf{E}; \mathbf{C})+\Lambda \left| |\mathbf{E}|-v\right|
		\end{equation}
		Similar to what we have done for the capillarity energy, we look for asymptotic minimizers for the penalized functional. 
		\begin{definition} 
			Let $\Lambda>0$ and let $(\mathbf{E}, \mathbf{C})$ be an asymptotic configuration. We say that $(\mathbf{E}, \mathbf{C})$ is an \emph{asymptotic} (or  \emph{generalized}) \emph{minimizer} for $\mathcal G^{v}_{C, \Lambda}$ if 
			\begin{equation}
				\mathcal G^{v}_{C, \Lambda}(\mathbf E, \mathbf C)\leq \mathcal G^{v}_{C, \Lambda}( \tilde{\mathbf E}, \tilde{\mathbf C}),
			\end{equation}
			where $( \tilde{\mathbf E}, \tilde{\mathbf C})$ varies among all asymptotic configurations (regardless of the lengths) and $E_i \neq \emptyset$ for every $i=1,\dots, k$.
		\end{definition}
		\begin{theorem}\label{teo: penal unb}
			For every $\Lambda>0$  there exist an asymptotic minimizer $(\mathbf E, \mathbf C)$ for the functional $\mathcal G^{v}_{C, \Lambda}$.
			Moreover,
			\begin{equation}
				\inf \mathcal F^{v}_{C, \Lambda}(F)=\min \mathcal G^{v}_{C, \Lambda}(\tilde{\mathbf E}; \tilde{\mathbf C})=\mathcal G^{v}_{C, \Lambda}(\mathbf E; \mathbf C),
			\end{equation}
			where $F\subset \mathbb R^N\setminus C$ varies among all sets of finite perimeter and $(\tilde{\mathbf E}; \tilde{\mathbf C})$ varies among all asymptotic configurations.
			Finally, if $\mathbf{E}=(E_0, \ldots, E_k)$, $\mathbf{C}=(C,C_1, \ldots, C_k)$ and $C_j=\lim\limits_{n \to +\infty} \left(C-x_n^{(j)}\right)$, $j=1, \ldots, k$ in the sense of Kuratowski for some diverging sequences $\left(x_n^{(j)}\right)_{n\in \mathbb{N}}\subset\partial C$, then there exist a sequence $(\tilde F_n)_{n\in \mathbb{N}}$ of sets of finite perimeter and $k+1$ sequences $\left(\tilde F_n^{(j)}\right)_{n\in \mathbb{N}}$ in $\mathbb R^N\setminus C$ with the following properties:
			\begin{itemize}
				\item for all $n\in \mathbb N$, the sets $\tilde F_n^{(i)}$, $\tilde F_n^{(j)}$ have positive distance when $i\neq j$;
				\item for all $n\in \mathbb N$, $\tilde F_n=\bigcup\limits_{j=0}^{k} \tilde F_n^{(j)}$;
				\item $\tilde F^{(j)}_n-x^{(j)}_n\xrightarrow{n \to +\infty} E_j$ strictly;
				\item $\mathcal F_{C}(\tilde F_n^{(j)})\xrightarrow{n \to +\infty} \mathcal F_{C_j}(E_j)$.
			\end{itemize}
		\end{theorem}
		\begin{proof}
			\textbf{Step 1:} We begin by proving that the two infima coincide. First, note that the infimum of $\mathcal{G}^{v}_{C, \Lambda}$ is non--negative and bounded above by $P(B)$, where $B\subset\mathbb R^N$ is some ball of volume $v$.
			
			Let $\varepsilon>0$. Then there exist $k\in\mathbb N$ and an asymptotic configuration $(\mathbf E, \mathbf C)$  of length $k$, where $\mathbf E=(E_0, \ldots, E_k)$, with $E_j$ bounded for $j=0, \ldots, k$,  and $\mathbf C=(C,C_1, \ldots, C_k)$, such that
			\begin{equation}
				\mathcal{G}^{v}_{C, \Lambda}(\mathbf E; \mathbf C)\leq \inf \mathcal{G}^{v}_{C, \Lambda}(\tilde{\mathbf E}; \tilde{\mathbf C})+\varepsilon,
			\end{equation}
			where $(\tilde{\mathbf E}, \tilde{\mathbf C})$ varies among all possible asymptotic configurations.
			Let us now exploit the approximating sequences granted by \Cref{teo: gamma conv} in order to approximate the sets $E_j$ with sequences in $\mathbb R^N\setminus C$.
			
			Let $1\leq j\leq k$ be fixed. Let $R_j>0$ be such that $E_j\Subset B(0, R_j)$. 
			Now, by definition of asymptotic cylinder and \Cref{lemma: local hausdorff} (notice that $0\in C_j$ by definition), there exists $\left(x^{(j)}_n\right)_{n\in \mathbb{N}}$ such that $\left(C-x_n^{(j)}\right)\cap \widebar{B(0, R_j)}\xrightarrow{n \to + \infty} C_j\cap \widebar{B(0, R_j)}$ in Hausdorff distance. Moreover, since $E_j\Subset B(0, R_j)$, we have that
			\begin{equation}
				\mathcal F_{C_j}(E_j)=\mathcal F_{C_j\cap \widebar{B(0, R_j)}}(E_j).
			\end{equation}
			Now, by \Cref{teo: gamma conv}, there exists a sequence of sets of finite perimeter $\left(A_n^{(j)}\right)_{n\in \mathbb N}$ such that:
			\begin{itemize}
				\item $A^{(j)}_n\subset\mathbb R^N\setminus \left(C-x_n^{(j)}\right)$ and $A^{(j)}_n\Subset B(0, R_j)$;
				\item $A^{(j)}_n\to E_j$ strictly for $n \to +\infty$;
				\item $\mathcal F_{\left((C-x_n^{(j)})\cap \widebar{ B(0, R_j)}\right)} \left(A_n^{(j)} \right)\to \mathcal F_{C_j}(E_j)$ for $n \to +\infty$.
			\end{itemize} 
			
			Let us now choose our sets $\tilde F_n^{(j)}$. For $j=0$, we set $\tilde F_n^{(0)}=E_0$ for all $n\in \mathbb N$.
			Let now $j=1$. Since $x_n^{(1)}$ diverges, we can set $m_1^{(1)}$ as the first $n$ such that $A_n^{(1)}+x^{(1)}_n$ has positive distance from $E_0$. Let $2 \leq j \leq k$ and suppose to have defined $m_1^{(l)}$ for $l < j$. We define $m_1^{(j)}$ to be the first natural number $n$ such that $A_n^{(j)}+x^{(j)}_n$ has positive distance from $E_0, A_{m_1^{(l)}}^{(l)}+x^{(l)}_{m_1^{(l)}}$ for $l=1, \ldots, j-1$.
			We can then pass to $m_2^{(1)}$, by choosing it as the first $n$ such that $A_n^{(1)}+x^{(1)}_n$ has positive distance from all the previously selected sets.
			Proceeding recursively, we obtain k sequences $\left(m_n^{(j)} \right)_{n\in \mathbb N, j=1, \ldots, k}$ such that, setting
			\begin{equation}
				\tilde F_n^{(j)}=A_{m_n^{(j)}}+x^{(j)}_{m_n^{(j)}},
			\end{equation}
			and relabelling $x^{(j)}_{m_n^{(j)}}$ as $x_n^{(j)}$, we have that:
			\begin{itemize}
				\item $d(\tilde F_n^{(i)}, \tilde F_n^{(j)})>0$ when $i\neq j$;
				\item $E_j=\lim\limits_{n \to +\infty} \tilde F_n^{(j)}-x_n^{(j)}$ strictly;
				\item $\mathcal F_{C}(\tilde F_n^{(j)})\to \mathcal F_{C}(E_j)$.
			\end{itemize}
			Moreover, since the sets are at positive distance from each other, if we set $\tilde F_n=\bigcup\limits_{j=0}^k \tilde F_n^{(j)}$, we get that
			\begin{equation}
				\mathcal{F}_C(\tilde F_n)=\sum_{j=0}^k \mathcal{F}_C(\tilde F_n^{(j)})= \mathcal{F}_C(E_0)+\sum_{j=1}^k \mathcal F_{(C-x_n^{(j)})\cap B(0, R_j)}(A_n^{(j)}).
			\end{equation}
			By definition of the $A_n^{(j)}$, we have that 
			\begin{equation}
				\inf \mathcal{F}^{v}_{C, \Lambda}(F)\leq \lim_n \mathcal{F}_C(\tilde F_n)= \mathcal{G}^{v}_{C, \Lambda}(\mathbf E;\mathbf C)\leq \inf \mathcal{G}^{v}_{C, \Lambda}(\tilde{\mathbf E}; \tilde{\mathbf C})+\varepsilon,
			\end{equation}
			where $F\subset\mathbb R^N\setminus C$ varies among all sets of finite perimeter and $(\tilde{\mathbf E}, \tilde{\mathbf C})$ varies among all possible asymptotic configurations for $C$.
			Since $\varepsilon>0$ is arbitrary, it follows that $\inf \mathcal{F}^{v}_{C, \Lambda}(F)\leq\inf \mathcal{G}^{v}_{C, \Lambda}(\tilde{\mathbf E}; \tilde{\mathbf C})$. The other inequality is trivial, since all possible competitor sets for $\mathcal {F}^{v}_{C, \Lambda}$ yield an asymptotic configuration for $C$.
			
			\noindent\textbf{Step 2:} we prove that the infimum for $\mathcal{G}^{v}_{C, \Lambda}$ is indeed a minimum. Let $F_n^{(1)}, \ldots, F_n^{(M)}$ for $n\in \mathbb N$ and $(x_n^{(j)})_{n\in \mathbb N}$ be the sequences chosen at the beginning of \Cref{section: existence}.
			The behaviour of the sets $F_n^{(j)}$ for $j>M_1$ is easier to describe: up to subsequences, their union converge in $L^1$ to a set $E_0\subset\mathbb R^N\setminus C$ such that, by \Cref{prop: semicontinuity}, satisfies
			\begin{equation}
				\mathcal F^{v}_{C, \Lambda}(E_0)\leq \liminf_{n\to +\infty}\mathcal {F}^{v}_{C, \Lambda} \left( \bigcup_{j=M_1+1}^{M} F_n^{(j)}\right).
			\end{equation}
			
			The components $F^{(j)}_n$ with $j\leq M_1$ are treated differently. Up to extracting a subsequence, we may assume that $C-x_n^{(j)}\to C_j$ in the sense of Kuratowski and that $F_n^{(j)}-x_n^{(j)} \to E_j$ in $L^1$, where $E_j\subset\mathbb R^N\setminus C_j$, for $n\to \infty$ and $j=1, \ldots, M_1$.
			Now, since $ \left(C-x_n^{(j)}\right)\cap \widebar{B(0, 2d)}\to \left(C_j\cap \widebar{B(0, 2d)}\right)$ and $\left(F_n^{(j)}-x_n^{(j)}\right)\Subset \widebar{B(0, 2d)}$ for all $n$ and $j$, we may apply \Cref{teo: gamma conv} to obtain
			\begin{align}
				\mathcal {F}^{v}_{C_j, \Lambda}(E_j)&=\mathcal {F}^{v}_{C_j\cap \widebar{B(0, 2d)}, \Lambda}(E_j)\\
				&\leq \liminf_{n\to +\infty} \mathcal {F}^{v}_{ C-x_n^{(j)}\cap \widebar{B(0, 2d)}, \Lambda}\left(F_j-x_n^{(j)}\right)\\
				&= \liminf_{n \to +\infty} \mathcal {F}^{v}_{ C, \Lambda}(F_n^{(j)})
			\end{align}
			for all $j=1, \ldots, M_1$. By adding up in $j$ and setting $\mathbf{E}=(E_0, \ldots, E_{M_1})$, $\mathbf{C}=(C,C_1, \ldots, C_{M_1})$, we finally get that		
			\begin{equation}
				\mathcal{G}^{v}_{C, \Lambda}(\mathbf{E};\mathbf{C})\leq \liminf_n \mathcal{F}^{v}_{C, \Lambda}(F_n)=\inf\mathcal{F}^{v}_{C, \Lambda}(F)=\inf \mathcal{G}^{v}_{C, \Lambda}(\tilde{\mathbf{E}};\tilde{\mathbf{C}}), 
			\end{equation}
			where $F\subset \mathbb R^N\setminus C$ varies among all sets of finite perimeter and $(\tilde{\mathbf{E}}, \tilde{\mathbf{C}})$ varies among all possible asymptotic configurations for $C$: we have proved that $(\mathbf{E}, \mathbf{C})$ is an asymptotic minimizer for $\mathcal {G}^{v}_{C, \Lambda}$, so the thesis follows.
            
            \textbf{Step 3:} Finally, if $(\mathbf E, \mathbf C)$ is a generalized minimizer for $\mathcal{G}_{C,\Lambda}^v$, the existence of a sequence $(\tilde{F}_n)_{n\in\mathbb{N}}$ with the properties stated in the Theorem follows by exactly the same argument as in Step 1. Indeed, since $(\mathbf E, \mathbf C)$ is a generalized minimizer, each $E_i$ is a $\Lambda$-minimizer for $\mathcal{F}_{C_i}$ and, by the density estimates in \Cref{prop: density est}, it is bounded. Therefore, the construction of Step 1 can be applied.
		\end{proof}
		In order to conclude the existence of asymptotic minimizers with volume constraint (and not just volume penalization), it suffices to show that for sufficiently large values of $\Lambda$ the volume constraint is attained.
		
		Before doing so, we prove a little generalization of \cite[Lemma 2.1]{MoPoSp22}:
		\begin{lemma}\label{lm: concentration2}
			Let $a, b>0$. There exists $m=m\left(a, b\right)>0$ with the following property: if $E_1, \ldots, E_n\subset\mathbb R^N$ are sets of finite perimeter with $\sum_{j=1}^n|E_j|\geq a$ and $\sum_{j=1}^N P\left(E_j\right)\leq b$, then there exists $1\leq j\leq n$ and a cube $\tilde{Q}$ with sides of unitary length and parallel to the coordinate axes such that $|\tilde{Q}\cap E_j|\geq m$.
		\end{lemma}
		\begin{proof}
			Let $Q$ be the (open) unitary cube. If there exist $1\leq j\leq n$ and $z\in \mathbb Z$ such that $|E_j\cap\left(z+Q\right)|\geq \frac{1}{2}$, the thesis follow. If not, we have 
			\begin{equation}
				|E_j \cap \left(z+Q\right)| \leq |\left(z+Q\right) \setminus E_j|
			\end{equation} 
			for all $z\in\mathbb Z$. 
            
            Let $\gamma>0$ be the constant of the relative isoperimetric inequality, i.e.\ such that $|F\cap Q|^{\frac{N-1}{N}}\leq \gamma P\left(F, Q\right)$ for all sets of finite perimeter $F$ satisfying $|F\cap Q|\leq \frac 12$.
            Writing $V=\sup_{z\in\mathbb Z^N, \,j=1, \ldots, n} |E_j\cap\left(z+Q\right)|$, have:
			\begin{align}
				a&\leq \sum_{j=1}^n|E_j|=\sum_{z\in\mathbb Z^N}\sum_{j=1}^n |E_j\cap\left(z+Q\right)|\leq  V^{\frac{1}{N}} \sum_{j=1}^n\sum_{z\in\mathbb Z^N} |E_j\cap\left(z+Q\right)|^{\frac{N-1}{N}}\\
				&\leq \gamma V^{\frac{1}{N}}  \sum_{j=1}^n\sum_{z\in\mathbb Z^N} P\left(E_j, z+Q\right)\leq \gamma V^{\frac{1}{N}} \sum_{j=1}^n P\left(E_j\right)\leq \gamma V^{\frac{1}{N}} b.
			\end{align}
			This implies $\sup\limits_{z\in\mathbb Z^N, \,j=1, \ldots, n} |E_j\cap\left(z+Q\right)|\geq \min\left\{\left(\frac{a}{\gamma b}\right)^N, \frac{1}{2}\right\}$, which concludes.
		\end{proof}

\begin{lemma}[{\cite[Lemma 2.5]{RigLargeVol}, see also \cite[Theorem 1.1]{EspositoFusco}}]\label{lm: ef}
There exist $\varepsilon(N), c(N)>0$ with the following property. Let $E\subset\mathbb{R}^N$ be a set of finite perimeter with $|E|<v$ and let $B_R$ be a ball such that
$$
|E\cap B_{\frac{R}{2}}|\leq\varepsilon R^N,\quad |E\cap B_R|\geq \frac{\omega_NR^N}{2^{N+2}},
$$
with $\varepsilon\in(0,\varepsilon(N))$. Then there exists $\sigma_0=\sigma_0(E,v)>0$ such that for all $\sigma\in(0,\sigma_0)$ there exists a bi-Lipschitz map $\Phi_\sigma:\mathbb{R}^N\mapsto\mathbb{R}^N$, with $\Phi_\sigma(x)=x$ for all $|x|\geq R$, such that, setting $\widetilde E=\Phi_\sigma(E)$, one has
$$
|\widetilde E|<v, \quad P(E;\overline{B}_R)-P(\widetilde E;\overline{B}_R)\geq -2^NN\sigma P(E;\overline{B}_R), \qquad |\widetilde E|-|E|\geq c(N)\sigma R^N.
$$
\end{lemma}
		
		\begin{proposition}\label{prop: esposito fusco unb}
			Let $\lambda\in(-1,1)$ and $v>0$. Then there exists $\Lambda_0=\Lambda_0(\lambda,v)>0$ such that if $\Lambda\geq \Lambda_0$, $C\subset \mathbb R^N$ is a closed convex set and $(\tilde{\mathbf{E}};\tilde{\mathbf{C}})$ is a generalized minimizer for $\mathcal G^{v}_{C,\Lambda}$, then $|\tilde{\mathbf{E}}|=v$. Moreover, the constants can be chosen so that $\Lambda_0(\lambda,v)=v^{-\frac{1}{N}}\Lambda_0(\lambda,1)$.
        \end{proposition}
		\begin{proof}
			The proof is based on an argument similar to \cite[Theorem 3.2, step 6]{FoFuLeMo23}. 
            For completeness, we provide the full proof.

			We argue by contradiction and assume that there exist a diverging sequence $(\Lambda_n)_{n\in\mathbb{N}}$, a sequence $(C_n)_{n\in\mathbb{N}}$ of closed convex sets, and for each $n$ a generalized minimizer $(\mathbf E^{(n)}, \mathbf{C}^{(n)})$ for $\mathcal{G}_{C_n,\Lambda_n}^{v}$ such that the volume constraint is not attained. In particular, let
            \begin{equation}
            \mathbf{E}^{(n)}=(E^{(n)}_0,\ldots,E^{(n)}_{k_n}), \qquad \mathbf{C}^{(n)}=(C_n,C^{(n)}_1,\ldots,C^{(n)}_{k_n}).
            \end{equation} 
			We observe that necessarily $|\mathbf{E}^{(n)}|<v$ for every $n$. Indeed, if $E_i^{(n)}\neq \emptyset$ for some $i$, one can cut $E_i^{(n)}$ with a hyperplane not intersecting $C_i^{(n)}$ in such a way that the volume remains larger than $v$, thus producing a competitor with strictly lower energy and contradicting minimality.

			By \Cref{cor: coercivity}, $\mathcal F_\cdot\geq \frac{1-\lambda}{2}P$.
            In particular, there exists $M>0$, depending only on $\lambda$, $v$ and $N$, such that
            \begin{equation}
            \sum_{j=0}^{k_n} P(E_j^{(n)}) \leq M
            \end{equation}
            for all $n$. 
            Indeed,
            \begin{equation}
            \sum_{j=0}^{k_n} P(E_j^{(n)}) \leq \frac{2}{1-\lambda}\mathcal G^{v}_{C_n,\Lambda_n}(\mathbf E^{(n)};\mathbf C^{(n)}) \leq \frac{2}{1-\lambda}P(\Omega_n)=:M,
            \end{equation}
            where $\Omega_n$ is a ball disjoint from $C_n$ with volume $v$ (notice that the value $P(\Omega_n)$ doesn't depend on $n$). Moreover, $\Lambda_n(v-|\mathbf{E}^{(n)}|)\leq M\frac{1-\lambda}{2}$, hence, for $n$ big enough, we have
            \begin{equation}
            \sum_{j=0}^{k_n} |E_j^{(n)}| \geq \frac{v}{2}>0.
            \end{equation} 
			
            By \Cref{lm: concentration2}, there exist a sequence $(j_n)_{n\in \mathbb N}$ and translations $(z_n)_{n\in \mathbb N}\subset\mathbb R^N$ such that $|(Q+z_n)\cap E_{j_n}^{(n)}|>m>0$ for some fixed $m=m(v,\lambda,N)>0$. Since $P(E_{j_n}^{(n)})\leq M$, up to subsequences there exist a closed convex set $\tilde C$ and $E\subset\mathbb R^N\setminus \tilde C$ such that $C_{j_n}^{(n)}-z_n\to \tilde C$ in the sense of Kuratowski and $E_{j_n}^{(n)}-z_n\to E$ in $L^1_{loc}$. In particular, $|E\cap Q|\geq m$ and $|\tilde C\cap Q|\leq 1-m$, hence $\partial^*E\setminus \tilde C\neq \emptyset$. Let $\bar{x}\in \partial^*E\setminus \tilde C$ and $\bar r>0$ be such that $\overline{B}(\bar{x},\bar r)\cap \tilde C=\emptyset$. By Kuratowski convergence, for $n$ big enough we have
            \begin{equation}
            \overline{B}(\bar{x},\bar r)\cap (C_{j_n}^{(n)}-z_n)=\emptyset.
            \end{equation}
            
            Fix $\varepsilon>0$. By the definition of reduced boundary, there exists $r=r(\varepsilon)>0$ such that, setting $E_r:=\frac{E-\bar{x}}{r}$ and $y:=-\frac{\nu_E(\bar{x})}{2}$, we have
        \begin{equation}
        \left|E_r\cap B\left(y,\frac{1}{2}\right)\right|<\varepsilon, \qquad \left|E_r\cap B\left(y,1\right)\right|>\frac{\omega_N}{2^{N+2}}.
        \end{equation}
        Up to decreasing $r$, setting $x_0=\bar{x}+ry$, it follows that
        \begin{equation}
        B(x_0,r)\subset B(\bar{x},\bar r), \quad \left|E\cap B\left(x_0,\frac{r}{2}\right)\right|<\varepsilon r^N, \quad \left|E\cap B(x_0,r)\right|>\frac{\omega_N}{2^{N+2}}r^N.
        \end{equation}
        Moreover, $\overline{B}(x_0 + z_n,r)\cap C_{j_n}^{(n)}=\emptyset$ for $n$ large. Hence the assumptions of \Cref{lm: ef} are satisfied for the balls $B_n:=B(x_0+z_n,r)$ and the sets $E_{j_n}^{(n)}$.

        For each $n$, choose $\sigma_n<\sigma_0\big(E_{j_n}^{(n)},\,v-|\mathbf{E}^{(n)}|+|E_{j_n}^{(n)}|\big)$ and let $\Phi_n:=\Phi_{\sigma_n}$ be given by \Cref{lm: ef}. Let $\tilde{E}_{j_n}^{(n)}=\Phi_n\left(E_{j_n}^{(n)}\right)$ and $(\tilde{\mathbf{E}}^{(n)};\mathbf{C}^{(n)})$ be the asymptotic configuration where we replace $E_{j_n}^{(n)}$ with $\tilde{E}_{j_n}^{(n)}$. Then

        $$
        |\widetilde {\mathbf{E}}^{(n)}|<v, \quad P(E_{j_n}^{(n)},\overline{B}_n)-P(\widetilde {E}_{j_n}^{(n)};\overline{B}_n)\geq -C(N)\sigma_n P(E_{j_n}^{(n)};\overline{B}_n), \quad |\widetilde {E}_{j_n}^{(n)}|-|E_{j_n}^{(n)}|\geq C(N)\sigma_n .
        $$
        Hence we have that
        \begin{align}
            \mathcal{G}_{C_n}(\tilde{\mathbf{E}}^{(n)};\mathbf{C}^{(n)}) + \Lambda_n(v- |\tilde{\mathbf{E}}^{(n)}|) &< \mathcal{G}_{C_n}(\mathbf{E}^{(n)};\mathbf{C}^{(n)}) + \Lambda_n(v- |\mathbf{E}^{(n)}|) 
            \\&+ C(N)\sigma_n ( P(E_{j_n}^{(n)},\overline{B}_n) - \Lambda_n )
        \end{align}
        Since $\Lambda_n \to +\infty$ and $P\left(E_{j_n}^{(n)}, \overline{B}_n\right) \leq M$, it follows that 
        \begin{equation}
            \mathcal {G}_{C_n, \Lambda_n}^{v}(\tilde{\mathbf E}^{(n)}; \mathbf C^{(n)})< \mathcal {G}_{C_n, \Lambda_n}^{v}(\mathbf E^{(n)}; \mathbf C^{(n)})
        \end{equation}
        for $n$ large enough, which is the desired contradiction. The dependence of the constants can be arranged in the stated form by a standard scaling argument.
    \end{proof}

		We are finally able to prove \Cref{teo: generalized}:
		\begin{proof}[Proof of \Cref{teo: generalized}]
			The proof does not immediately follow by combining \Cref{teo: penal unb} with \Cref{prop: esposito fusco unb}, since we still need to construct an infimizing sequence $(\tilde{F}_n)_{n\in\mathbb{N}}$ for the functional $\mathcal{F}_C$ satisfying the volume constraint $v>0$ and whose energy converges to the minimum of the generalized functional, as the one provided by \Cref{teo: penal unb} may not satisfy the volume constraint.
            
            Fix $\Lambda>\Lambda_0(\lambda,v)$ and consider the corresponding sequence $(F_n)_{n\in\mathbb{N}}$ defined at the beginning of \Cref{section: existence}.
            Up to subsequences, as shown in Step 2 of \Cref{teo: penal unb}, 
            \begin{gather}
                \bigcup_{j=M_1+1}^{M} F_n^{(j)} \xrightarrow{n \to +\infty} E_0\quad \text{ in }L^1, \\
                F_n^{(j)} - x_n^{(j)} \xrightarrow{n \to +\infty} E_j \quad \text{ in }L^1 \text{ for every} \ j=1,\dots, M_1,\\
                C- x_n^{(j)} \xrightarrow{n \to +\infty} C_j \quad \text{ in the Kuratowski sense,}
            \end{gather} 
            where $E_j \subset \mathbb{R}^N \setminus C_j$ for every $j=1,\dots,M_1$, that $(\mathbf E, \mathbf C) = ((E_0,\dots,E_{M_1});(C,C_1,\dots, C_{M_1}))$ is a generalized minimizer for $C$ with volume constraint $v$ and that $\mathcal{F}_{C,\Lambda}^{v}(F_n) \to \mathcal{G}_C(\mathbf E, \mathbf C)$. 
            
            Let us now construct a sequence $(\tilde{F}_n)_{n\in\mathbb{N}}$ that satisfies the volume constraint $v$. 
            Since $|\mathbf E|=v$, $\sum_{j=1}^k \big|F^{(j)}_n\big|-v$ is infinitesimal in $n$. Let us suppose that $\sum_{j=1}^k \big|F^{(j)}_n\big|>v$ for all $n\in \mathbb N$, as the other case is simpler. 
            
            Let us fix $j$ such that $E_j \neq \emptyset$.
            By \Cref{prop: density est}, \Cref{it: density est,it: perim est,it: height est}, there exists a sequence $n\mapsto B_n\Subset \mathbb{R}^N \setminus C_j$ of balls with infinitesimal radii such that, for $n$ large enough,
			\begin{equation}
			    \big|B_n\cap F_n^{(j)}\big|=  \sum_{i=1}^k \big|F^{(i)}_n\big|-v 
			\end{equation}
			and
			\begin{equation}
				P\left(F_n^{(j)}\setminus B_n;\mathbb R^N\setminus C\right)- P\left(F_n^{(j)};\mathbb R^N\setminus C\right)\to 0.
			\end{equation}
			Now, if we call 
			\begin{equation}
	        \tilde{F}_n=\bigcup_{\overset{i=0}{i\neq j}}^k F_n^{(i)}\, \cup \left(F_n^{(j)}\setminus B_n\right),
			\end{equation}
			we see that $|\tilde{F}_n|=v$ for all $n$ big enough. 
            Finally, 
            \begin{equation}
                |\mathcal{F}_C(\tilde{F}_n) -\mathcal{F}_C(F_n)|\leq P\left(F_n^{(j)}\setminus B_n;\mathbb R^N\setminus C\right)- P\left(F_n^{(j)};\mathbb R^N\setminus C\right), 
            \end{equation}
            which implies $\mathcal{F}_C(\tilde{F}_n) \to \mathcal{G}_C(\mathbf E, \mathbf C)$.
            This concludes the proof.
		\end{proof}
        Here we collect some regularity results about minimizers.
        \begin{proposition}\label{prop: gli E_i sono Lambda minimi}
        Let $C$ be a closed convex set and let $\lambda \in (-1,1)$.
        \begin{itemize} 
            \item If $(\mathbf{E};\mathbf{C})$ is a generalized minimizer for $\mathcal{G}_{C}$ with volume constraint $v>0$, then $E_i$ is a minimizer for $\mathcal{F}_{C_i}$ with volume constraint $v_i = |E_i|$ for all $i=0, \ldots, k$, where $k$ is the length of $(\mathbf{E};\mathbf{C})$.
            Moreover, for $\Lambda \geq \Lambda_0$ given by \Cref{prop: esposito fusco unb}, $E_i$ is a $\Lambda$-minimizer for $\mathcal{F}_{C_i}$ for $i=0, \ldots, k$.
            \item Let $\Lambda>0$ and $v>0$. If $(\mathbf{E};\mathbf{C})$ is a generalized minimizer for $\mathcal{G}_{C, \Lambda}^{v}$, then $E_i$ is a minimizer for $\mathcal{F}_{C_i, \Lambda}^{v_i}$ where $v_i=|E_i|$ for all $i=0, \ldots, k$, where $k$ is the length of $(\mathbf{E};\mathbf{C})$.
            In particular, $E_i$ is a $\Lambda$-minimizer for $\mathcal{F}_{C_i}$ for $i=0, \ldots, k$.
        \end{itemize}
        In particular, if $(\mathbf{E};\mathbf{C})$ is as in either of the two items above, then for $i=1, \ldots, k$, $E^{(1)}_i$ is open and $\partial^e E_i=\partial E_i^{(1)}$.
    \end{proposition}
    \begin{proof}
        Let $\Lambda>0$ and $v>0$ and assume that $(\mathbf{E};\mathbf{C})$ is a generalized minimizer for $\mathcal{G}_{C, \Lambda}^{v}$. 
        Let $i\in \{ 0, \ldots, k \}$, where $k$ is the length of $(\mathbf{E};\mathbf{C})$.
        Let $B$ be a ball and $F_i\subset \mathbb R^N\setminus C_i$ be such that $E_i\Delta F_i\Subset B$ and let $\mathbf F=(E_0, \ldots, E_{i-1}, F_i, E_{i+1}, \ldots, E_k)$.
        By minimality, 
        \begin{align}
            \mathcal{G}_{C, \Lambda}^{v}(\mathbf{E};\mathbf{C})&\leq \mathcal{G}_{C, \Lambda}^{v}(\mathbf{F};\mathbf{C})\\
            &=\mathcal{G}_{C, \Lambda}^{v}(\mathbf{E};\mathbf{C})+\mathcal F_{C_i}(F_i)-\mathcal F_{C_i}(E_i)+\Lambda\Big|\sum_{j\neq i} |E_j|+|F_i|-v\Big| -\Lambda\Big|\sum_{j\neq i} |E_j|+|E_i|-v\Big|\\
            &\leq \mathcal{G}_{C, \Lambda}^{v}(\mathbf{E};\mathbf{C})+\mathcal F_{C_i}(F_i)-\mathcal F_{C_i}(E_i)+\Lambda \Big| |E_i|-|F_i|\Big|
        \end{align}
        Since $E\Delta F\Subset B$, this implies 
        \begin{equation}
            \mathcal F_{C_i}(E_i; B)\leq \mathcal F_{C_i}(F_i; B) +\Lambda |E_i\Delta F_i|,
        \end{equation}
        proving that $E_i$ is a $\Lambda$-minimizer for $\mathcal F_{C_i}$.
        Similarly, from the second-to-last line, 
        \begin{align}
            F_{C_i}(E_i)+\Lambda \Big| |E_i|-v_i\Big|=\mathcal F_{C_i}(E_i)&\leq \mathcal F_{C_i}(F_i) +\Lambda \Big| |F_i|- |E_i|\Big|\\
            &=\mathcal F_{C_i}(F_i) +\Lambda \Big| |F_i|- v_i\Big|,
        \end{align}
        so $E_i$ is also a minimizer for $\mathcal F^{v_i}_{C_i, \Lambda}$.

        Assume now that $(\mathbf{E};\mathbf{C})$ is a generalized minimizer for $\mathcal G_C$.
        The fact that the components $E_i$, $i=1, \ldots, k$, with $k$ being the length, are minimizer for $\mathcal F_C$ follows from computations analogue to the above.
        Moreover, by \Cref{prop: esposito fusco unb}, for $\Lambda\geq \Lambda_0$, $(\mathbf{E};\mathbf{C})$ is a generalized minimizer for $\mathcal G^v_{C, \Lambda}$, where $v=|\mathbf E|$.
        By what has been proved above, then, $E_i$ is a $\Lambda$-minimizer for $\mathcal{F}_{C_i}$ for $i=0, \ldots, k$.        
    \end{proof}
    If $(\mathbf{E};\mathbf{C})$ is as in either of the two items in \Cref{prop: gli E_i sono Lambda minimi} we may choose each $E_i$ to be its open representative $E_i^{(1)}$. Throughout this and the following sections, we always adopt this convention without further mention.
    \begin{corollary}\label{cor: le componenti devono toccare il convesso}
        Let $(\mathbf{E};\mathbf{C})=((E_0,\dots,E_k);(C,C_1,\dots,C_k))$ be a generalized minimizer for $C$ with volume constraint $v$. 
        Then, for every $i=0,\dots,k$ and every connected component $\Omega$ of $E_i$ the wet contribution of $\Omega$ is non zero, i.e.\ 
        \begin{equation}
            \mathcal F_{C_i}(\Omega)-P(\Omega; \mathbb R^N\setminus C_i) \neq 0.
        \end{equation}

        In particular, for every $i=1,\dots,k$, $\dim(C_i)\geq N-1$ 
    \end{corollary}
    \begin{proof}
        Suppose by contradiction that, for some $i$, $E_i$ has a connected component $\Omega$ whose wet contribution to the functional is zero.
        By \Cref{prop: gli E_i sono Lambda minimi}, $\Omega$ is a perimeter minimizer with volume constraint $w=|\Omega|$, and hence it is a ball. 
        
        Let then $(\tilde{\mathbf{E}}; \tilde{\mathbf{C}})$ be the asymptotic configuration with $\tilde{E}_j=E_j$ for every $j\notin \{0,i\}$, $\tilde{E}_i=E_i\setminus\Omega$
        and $\tilde{E}_0=E_0\cup B,$ where $B$ is an open ball contained in $\mathbb{R}^N \setminus(E_0\cup C)$ with volume $w$, such that
        either of the following hold:
        \begin{itemize}
            \item $\widebar{B}$ is disjoint from $C$ and $\widebar{B}\cap \widebar{E_0}$ is a point of the reduced boundary of $E_0$;
            \item $\widebar{B}$ is disjoint from $\widebar{E_0}$ and $\widebar{B}\cap C$ is a point of the reduced boundary of C.
        \end{itemize}

        Notice that $\mathcal G_C(\tilde{\mathbf{E}}; \tilde{\mathbf{C}})=\mathcal G_C(\mathbf{E}; \mathbf{C})$, so that $(\tilde{\mathbf{E}}; \tilde{\mathbf{C}})$ is a generalized minimizer for $\mathcal G_C$. 
        By \Cref{prop: gli E_i sono Lambda minimi}, for $\Lambda$ large enough $E_0\cup B$ is a $\Lambda$-minimizer for $\mathcal F_C$.
        
        Finally, it is easy to prove that, the complement $F=\mathbb R^N\setminus (E_0\cup B \cup C)$ is a $\Lambda$-minimizer for the capillarity energy with adhesion coefficient $-\lambda$.
        A computation shows that $F$ does not satisfy the density estimates of \Cref{prop: density est}, leading to a contradiction.

        In particular, for every $i=1,\dots,k$, $\dim (C_i) \geq N-1$ since otherwise the wet contribution of $E_i$ would be zero.
    \end{proof}

    The following remark follows from well known facts in the literature (see \cite{Ma12,DePhilMaggi,Taylor1977}).

    \begin{remark}[Regularity]
        Let $(\mathbf{E};\mathbf{C})=((E_0,\dots,E_k);(C,C_1,\dots,C_k))$ be a generalized minimizer for the closed convex set $C$ with volume constraint $v$. 
        For each $i=0,\dots,k$ define $M_i:=  \overline{\bigl(\partial E_i\cap(\mathbb{R}^N\setminus C_i)\bigr)}.$
        Then there exists a closed set $\Sigma_i\subset M_i$ such that $M_i \setminus (C_i \cup \Sigma_i) $ is an analytic hypersurface and $\text{dim}_{\mathcal{H}}(\Sigma_i \cap(\mathbb{R}^N \setminus C_i) ) \leq N-8$. 
        Moreover:
        \begin{itemize}
            \item if $\text{dim}(C_i)=N$ and $C_i$ has $C^{1,1}$ boundary, then $M_i \setminus \Sigma_i$ is of class $C^{1,\frac{1}{2}}$ as a hypersurface with boundary. Moreover, \emph{Young's law} holds, namely, setting $B_i = C_i$,
            \begin{equation}
                \nu_{C_i}(x) \cdot \nu_{E_i}(x) = -\lambda \ \text{for every } x\in (M_i\setminus \Sigma_i) \cap B_i 
            \end{equation}
            \item if $\text{dim} (C_i)=N-1$, then $ M_i \setminus (\Sigma_i \cup \partial C_i)$ is of class $C^{1,\frac12}$, where $\partial C_i$ denotes the relative boundary of $C_i$ in the hyperplane $H$ containing $C_i$. Let $H^\pm$ be the open half-spaces determined by $H$, and let $\mu^\pm\in\mathbb S^{N-1}$ denote the unit normal to $H$ pointing into $H^\pm$. Setting $M_i^\pm:= \overline{(\partial E_i\cap H^\pm)}$ and $B_i = C_i \setminus \partial C_i$, \emph{Young's law} holds at every point in the relative interior of $C_i$, namely
                \begin{equation}
                \nu_{E_i}(x)\cdot \mu^{\pm}=-\lambda \qquad \text{for every } x\in  (M_i^{\pm}\setminus \Sigma_i)\cap B_i.
                \end{equation}
        \end{itemize}
        Finally, if $C_i$ satisfies the assumptions of either of the previous two cases, then $\Sigma_i \cap B_i = \emptyset$ if $N=3$, $\Sigma_i \cap B_i$ is discrete if $N=4$, and $\mathcal{H}^{s}(\Sigma_i \cap B_i) = 0$ for every $s>N-4$ if $N\geq 5$.
    \end{remark}

    We conclude this section by discussing two simpler settings where \Cref{teo: generalized} yields existence of minimizers for $\mathcal F_C$.

    The first is when $C$ is a compact convex set.
    In that case, existence for the capillarity problem is already known via different arguments, but it follows at once from \Cref{teo: generalized} 
    \begin{corollary}
        Let $C$ be a compact convex set, $\lambda \in (-1, 1)$ and $v>0$.
        Then there exists a minimizer of $\mathcal F_C$ among sets $E\subset \mathbb R^N\setminus C$ with volume $v$.
    \end{corollary}
    \begin{proof}
        Let us start by noticing that $\mathcal K(C)=\emptyset$.
        Let $(\mathbf E, \mathbf C)$ be a generalized minimizer for $\mathcal G_C$.
        Then, by construction, the length of $(\mathbf E, \mathbf C)$ is 0, since $\mathcal K(C)=\emptyset$.

        By \Cref{prop: gli E_i sono Lambda minimi}, $E_0$ is a minimizer of $\mathcal F_C$ with volume constraint $v=|E_0|$.
    \end{proof}

    Furthermore, our existence result generalizes \cite[Theorem 3.2]{FoFuLeMo23}, in which existence was proved for minimizers of the capillarity energy outside obstacles coinciding at infinity with convex semi--cylinders. Indeed, a stronger result readily follows from \Cref{teo: generalized}:
		\begin{corollary}\label{cor: nanowires}
			Let $C\subset\mathbb R^N$ be a closed convex set such that all of its asymptotic cylinders have compact section. 
            Assume that for asymptotic cylinder $K$ for all $R>0$, there exist a diverging sequence $(x_n)_{n\in \mathbb N}\subset C$ and a sequence $y_n\in K$ such that there is $C\cap B(x_n, R)=\left(K\cap B(y_n, R)\right)+(x_n-y_n)$.
			Then for all $\lambda\in (-1, 1)$ and $v>0$ $\mathcal{F}_{C}$ admits a minimizer with volume constraint $v$.
		\end{corollary}
		\begin{proof}
			Let $(\mathbf E, \mathbf{C})$ be an asymptotic minimizer for the capillarity energy outside $C$, with 
			\begin{equation}
				\mathbf{E}=(E_0, \ldots, E_k), \qquad \mathbf{C}=(C,C_1, \ldots, C_k).
			\end{equation}
            
			Let $d=\max\limits_{j\geq 0}\operatorname{diam}(E_j)$ and let $D=\max\limits_{j\in \mathbb N}\operatorname{diam} \omega_j$, where $\omega_j$ is the section of $C_j$.
            Fix $R>2(d+D)$ and, for $j=1, \ldots, k$, let $(x^{(j)}_n)_{n\in \mathbb N}, (y^{(j)}_n)_{n\in \mathbb N}$ be as in the statement.
            
            Let $x_0\in E_0$ and $n_1, \ldots, n_k$ indexes such that $B(x_0, R)$ and $B\left(x_{n_j}^{(j)}, R\right)$ are all disjoint (it is possible to do so since the sequences $\left(x^{(j)}_n\right)_{n\in \mathbb N}$, $j=1, \ldots, k$ diverge).

            Notice that, by \Cref{cor: le componenti devono toccare il convesso}, $\widebar{E}_j\cap C_j\neq \emptyset$.
            By the invariance with respect to translations of the cylinders and the bounds on $R$, we can assume that $E_j\subset B(y_{n_j}^{(j)}, R)$ for $j=1, \ldots, k$.
			Let $F_0=E_0$, $F_j=E_j+x_{n_j}^{(j)}-y_{n_j}^{(j)}$, $j=1, \ldots, k$. We have that
			\begin{equation}
            \mathcal{F}_C(F_0) =\mathcal{F}_C(E_0), \qquad \mathcal{F}_C(F_j)=\mathcal F_{C_j}(E_j),
			\end{equation} 
			so, since by construction the sets $F_j$ are well-separated, $\mathcal F_C(F)=\mathcal G_C(\mathbf E;\mathbf C)$, where $F=\bigcup_{j=0}^k F_k$. 
            By \Cref{teo: generalized}, the thesis follows.
		\end{proof}

\section{The isoperimetric profile and its properties}\label{section: profilo}
This section is dedicated to the study of some general properties of the isoperimetric profile associated with the capillarity functional. We recall its definition.

\begin{definition}
    Let $C$ be a closed convex set.
    For $\lambda \in (-1,1)$, $v>0$, the isoperimetric profile $I_C$ associated to $\mathcal F_C$ is defined as 
    \begin{equation}
        I_C(v)=\inf \left\{ \mathcal{F}_C(E) : E \subset \mathbb{R}^N \setminus C,\ |E| = v \right\}.
    \end{equation}
    More in general, for $R>0$, we define the local isoperimetric profile as
    \begin{equation}
        I_{C,R}(v) := \inf \left\{ \mathcal{F}_C(E) : E \subset B(0,R) \setminus C,\ |E| = v \right\}.
    \end{equation}
\end{definition}
\begin{lemma}\label{lemma: first properties profile}
    Let $C$ be a closed convex set and $\lambda\in (-1, 1)$.
    Then for all $v>0$
    \begin{equation}\label{eq: asymptotic profile}
        I_C(v)=\min \mathcal{G}_{C}(\mathbf{E};\mathbf{C}),
    \end{equation}
    where the minimum is taken among asymptotic configurations of volume $v$.
    Moreover, for all $v>0$, $I_{C,R}(v)\to I_C(v)$ as $R\to \infty$.
\end{lemma}
\begin{proof}
    \eqref{eq: asymptotic profile} follows directly from \Cref{teo: generalized}.
    The convergence $I_{C,R}(v)\to I_C(v)$ follows from the monotonicity of $I_{C,R}(v)$ with respect to $R$ and the fact that the sets in the infimizing sequence constructed in the proof of \Cref{teo: generalized} are bounded for each fixed index. Hence, for each such set, its energy is attained for all sufficiently large $R$.
\end{proof}
Before studying the properties of the isoperimetric profile for the general convex body, let us recall some properties in case $C$ is a half-space.
To this end, for $\lambda \in (-1, 1)$ and $r>0$, we recall the definition of the $\lambda$-spherical cap
\begin{equation}
    B^{\lambda}_r = \{ x \in B(0,r) \mid x  \cdot e_N > \lambda \} - \lambda r e_N = B(-\lambda r e_N, r) \setminus \{ x_N \leq 0 \}.
\end{equation}
For $v>0$, $\lambda \in (-1, 1)$, we denote by $B^{\lambda}[v]$ the rescaling of $B^{\lambda}_r$ satisfying $|B^{\lambda}_r|=v$.
\begin{proposition}\label{prop: isop profile del semispazio}
    Let $H$ be a half-space and $C$ be a closed convex set.
    Then
    \begin{enumerate}
        \item\label{it: profile flat} For all $v>0$, $I_H(v)=v^\frac{N-1}N I_H(1)$.
        \item \label{it: strict subadditivity} For all $v>w>0$ and $n\geq 2$, 
        \begin{equation}
            I_H(v) < \inf \{ I_H(v_1)+\dots+I_H(v_n) : v_1+\dots+v_n=v,\ v_i \ge w \ \forall i=1,\dots,n \}.
        \end{equation}
        \item \label{it: minimality} For all $v>0$, $I_H(v)\leq I_C(v)$.
        Moreover, if there exist $v>0$ and $E\subset \mathbb R^N\setminus C$ with $|E|=v$ such that $\mathcal F_C(E)=I_H(v)$, then $E$ is isometric to $B^\lambda[v]$ and it lies on a flat part of $\partial C$.
    \end{enumerate}
\end{proposition}
\begin{proof}
    \Cref{it: profile flat} follows from a simple scaling argument and the homogeneity of $\mathcal F_C$.
    \Cref{it: strict subadditivity} follows from the homogeneity and subadditivity of $I_H$.
    \Cref{it: minimality} follows from \cite[Theorem 1.1]{IsopCap}.
\end{proof}
While it is not possible to give an explicit formula for $I_C$ and a characterization of minimizers as in the case $C=H$, we are still able to give some description of the isoperimetric profile and of minimizers for the general closed convex set.

\begin{proposition}\label{prop: proprietà del profilo isoperimetrico}
    Let $C$ be a closed convex set and $\lambda \in (-1, 1)$. 
    The following hold:
    \begin{enumerate}
        \item $I_C$ is strictly increasing and for every $a>0$ it is Lipschitz on $[a,+\infty)$, with constant $\Lambda_0(\lambda) a^{-\frac{1}{N}}$.
        \item If $v>0$ and $(\mathbf{E};\mathbf{C})=((E_0,\dots,E_k);(C,C_1,\dots,C_k))$ is a generalized minimizer for $I_C(v)$ there is a constant $H_{\mathbf{E}}$ such that the mean curvature $H_{E_i}$ of each $E_i$, $0,\dots,k$ is constant and equal to $ H_{\mathbf{E}}$ on $\partial^{*}E_i \setminus C$
        \item There exists a constant $\Lambda_0=\Lambda_0(\lambda)>0$ such that $0<H_{\mathbf{E}} \leq \Lambda_0 v^{-\frac{1}{N}}$.
        \item If $I_C$ is differentiable at $v>0$, then $I'_C(v)= H_{\mathbf{E}}$.
    \end{enumerate}
\end{proposition}
\noindent The number $H_{\mathbf{E}}$ is called \emph{mean curvature} of the free boundary of the generalized minimizer $(\mathbf{E};\mathbf{C})$.
\begin{proof}
    The proof of Lipschitz continuity follows by the $\Lambda$-minimality. Indeed, let $a>0$, $v_1, v_2 \in [a,+\infty)$, $\Lambda_0 = \Lambda_0(1,\lambda)>0$ be the constant of \Cref{prop: esposito fusco unb} and let $(\mathbf{E}^{(i)},\mathbf{C}^{(i)})$ be the generalized minimizer for $I_C(v_i)$ for every $i=1,2$. By \Cref{prop: esposito fusco unb}, $(\mathbf{E}^{(i)},\mathbf{C}^{(i)})$ is a minimizer of $\mathcal{G}^{v_i}_{C, \Lambda_0 v_i^{-1/N}}$. Hence we have
    \begin{align}
        I_C(v_1) &= \mathcal{G}^{v_1}_{C, \Lambda_0 v_1^{-1/N}}\left(\mathbf{E}^{(1)},\mathbf{C}^{(1)}\right) \leq \mathcal{G}_{C}\left(\mathbf{E}^{(2)},\mathbf{C}^{(2)}\right) + \Lambda_0 v_1^{-\frac{1}{N}} \left | v_1 - \left| \mathbf{E}^{(2)} \right| \right | 
        \\&\leq I_C(v_2) + \Lambda_0 a^{-\frac{1}{N}} | v_1 - v_2|
    \end{align}
    Exchanging the role of $v_1$ and $v_2$ gives the Lipschitzianity on $[a,+\infty)$.

    Let $v_0 > 0$ and let $(\mathbf{E}^{v_0}, \mathbf{C}^{v_0})=((E_0,\dots, E_k),(C,C_1,\dots,C_k))$ be a generalized minimizer for $I_C(v_0)$. From the standard first variation theory for the perimeter under a volume constraint (see \cite[Theorem 19.7]{Ma12}), it follows that for every \(i=0,\dots,k\) such that \(E_i \neq \emptyset\), the scalar mean curvature \(H_{E_i}\) is constant on the free boundary \((\partial^{*}E_i)\setminus C_i\); denote this constant by $\lambda_i$. We prove that for every $i,j\in\{0,\dots,k\}$ such that $E_i\neq\emptyset\neq E_j$, $\lambda_i=\lambda_j$. We choose two vector fields $X_i\in C_c^\infty(\mathbb R^N\setminus C_i;\mathbb R^N)$ and $X_j\in C_c^\infty(\mathbb R^N\setminus C_j;\mathbb R^N)$ such that
    \begin{equation}
        \int_{\partial^*E_i\setminus C_i} X_i\cdot \nu_{E_i}\,d\mathcal H^{N-1}\neq 0, \qquad \int_{\partial^*E_j\setminus C_j} X_j\cdot \nu_{E_j}\,d\mathcal H^{N-1}\neq 0.
    \end{equation}
    Let $\Phi_i$ and $\Phi_j$ be the local variation generated by $X_i$ and $X_j$, namely
\begin{equation}
    \begin{cases}
        \frac{\partial }{\partial s}\Phi_i(s,x)=X_i(\Phi_i(s,x)), \\
        \Phi_i(0,x)=x,
    \end{cases}
    \qquad
    \begin{cases}
        \frac{\partial}{\partial t}\Phi_j(t,x)=X_j(\Phi_j(t,x)), \\
        \Phi_j(0,x)=x.
    \end{cases}
\end{equation}
For $s,t$ sufficiently small we define a new asymptotic configuration $(\mathbf E^{s,t},\mathbf C)$ by setting
\begin{equation}
    E_h^{s,t}:=
    \begin{cases}
        \Phi_i(s,E_i), & h=i, \\
        \Phi_j(t,E_j), & h=j, \\
        E_h, & h\neq i,j,
    \end{cases}
    \qquad C_h^{s,t}:=C_h \quad \text{for every } h=0,\dots,k.
\end{equation}
Since $\operatorname{supp}(X_i)\Subset\mathbb R^N\setminus C_i$ and $\operatorname{supp}(X_j)\Subset\mathbb R^N\setminus C_j$, this is an asymptotic configuration for $C$. Define
\begin{equation}
    V(s,t):= |\mathbf{E}^{s,t}| =|\Phi_i(s,E_i)|+|\Phi_j(t,E_j)|+\sum_{h\neq i,j}|E_h|.
\end{equation}
Then $V(0,0)=v$ and, by the first variation formula for the volume (see \cite{Ma12}),
\begin{equation}
    \partial_t V(0,0)=\int_{\partial^*E_j\setminus C_j} X_j\cdot \nu_{E_j}\,d\mathcal H^{N-1}\neq 0, \qquad \partial_s V(0,0)=\int_{\partial^*E_i\setminus C_i} X_i\cdot \nu_{E_i}\,d\mathcal H^{N-1}\neq 0.
\end{equation}
Hence, by the implicit function theorem, there exist $\varepsilon>0$ and a $C^1$ function $t:(-\varepsilon,\varepsilon)\to\mathbb R$ with $t(0)=0$ such that
\begin{equation}
    V(s,t(s))=v \qquad \text{for every } s\in(-\varepsilon,\varepsilon).
\end{equation}
Moreover,
\begin{equation}
    t'(0)=-\frac{\partial_s V(0,0)}{\partial_t V(0,0)}=-\frac{\int_{\partial^*E_i\setminus C_i} X_i\cdot \nu_{E_i}\,d\mathcal H^{N-1}}{\int_{\partial^*E_j\setminus C_j} X_j\cdot \nu_{E_j}\,d\mathcal H^{N-1}}.
\end{equation}
By the minimality of $(\mathbf E,\mathbf C)$, the function $s\mapsto \mathcal G_C(\mathbf E^{s,t(s)},\mathbf C)$ has a minimum at $s=0$, hence
\begin{equation}
    0=\frac{d}{ds}\mathcal G_C(\mathbf E^{s,t(s)},\mathbf C)\bigg|_{s=0}.
\end{equation}
Again, by the standard first variation theory for the perimeter contained in \cite{Ma12}, we obtain
\begin{equation}
    0=\lambda_i \int_{\partial^*E_i\setminus C_i} X_i\cdot \nu_{E_i}\,d\mathcal H^{N-1}+\lambda_j t'(0)\int_{\partial^*E_j\setminus C_j} X_j\cdot \nu_{E_j}\,d\mathcal H^{N-1}.
\end{equation}
Substituting the expression for $t'(0)$ yields
\begin{equation}
    0=(\lambda_i-\lambda_j)\int_{\partial^*E_i\setminus C_i} X_i\cdot \nu_{E_i}\,d\mathcal H^{N-1}.
\end{equation}
Since the last integral is non-zero, we conclude that $\lambda_i=\lambda_j.$ In particular, we may denote this quantity by $H_{\mathbf{E}^{v_0}}$.

\noindent We now show that, if $I_C$ is differentiable at $v_0$, then $I_C'(v_0)=H_{\mathbf{E}^{v_0}}$. This follows from a standard first variation argument (see \cite{Sternberg-Zumbrun} and \cite[Theorem 1.2, Step 4]{FuMo23}), which we briefly recall for completeness. Fix $i$ such that $E_i \neq \emptyset$ and consider a vector field $X$ such that $\text{supp}(X) \Subset \mathbb{R}^N \setminus C_i$ and 
\begin{equation}\label{eq: eq2 dim prof isop}
    \int_{(\partial ^{*} E_i) \setminus C_i} X \cdot \nu_{E_i} >0.  
\end{equation}
Consider the local variation $\Phi$ generated by the field $X$:
\begin{equation}
    \begin{cases}
        \frac{\partial \Phi}{\partial t} (t,x) = X(\Phi (t,x)), \\
        \Phi(0,x)=x. \label{eq: eq3 dim prof isop} 
    \end{cases}
\end{equation}
Let $\mathbf{E}(t) = (E_0,\dots, E_{i-1},\Phi(t, E_i), E_{i+1}, \dots, E_k)$ and set $f(t) = \mathcal{G}_{C}(\mathbf{E}(t),C)$. By the first variation formula, we have
\begin{equation}\label{eq: eq4 dim prof isop}
    f'(0) = \frac{d}{dt} P(\Phi (t,E_i),\mathbb{R}^N \setminus C_i)_{|_{t=0}} =  \lambda_i \int_{(\partial^{*} E_i) \setminus C_i} X \cdot \nu_{E_i}, 
\end{equation}
and 
\begin{equation}\label{eq: eq5 dim prof isop}
    \frac{d}{dt} \left|\mathbf{E}(t) \right|_{|_{t=0}} = \frac{d}{dt} | \Phi(t, E_i) |_{|_{t=0}} = \int_{(\partial^{*}E_i)\setminus C_i} X \cdot \nu_{E_i} >0. 
\end{equation}
Since the map $t \mapsto | \mathbf{E}(t) |$ has positive derivative at $0$, it is locally invertible. Denote its inverse by $t(\cdot) : v \in (v_0 - \delta, v_0 + \delta) \mapsto t(v) \in (-\epsilon, \epsilon)$ and define $\tilde{f}(v) = f(t(v))$. We have $\tilde{f}(v_0) = I_C(v_0)$, $\tilde{f}(v) \geq I_C(v)$ for every $v \in (v_0 - \delta, v_0 + \delta)$, and both $I_C$ and $\tilde{f}$ are differentiable at $v_0$. This implies that their derivatives coincide, hence
\begin{equation}
    I_C'(v_0) = \tilde{f}'(v_0) = f'(0)\, t'(v_0) = H_{\mathbf{E}^{v_0}},
\end{equation}
as desired. We also obtain that  
\begin{equation}
    H_{\mathbf{E}^{v_0}} \leq \Lambda_0(\lambda)\, v_0^{-\frac{1}{N}}.
\end{equation}
Indeed, arguing as at the beginning of the proof and using the $\Lambda_0$-minimality, we have
\begin{equation}
    |f(t) - f(0)| \leq \frac{\Lambda_0}{v_0^{\frac{1}{N}}} \, |V(t) - v_0| .
\end{equation}
where $V(t)= | \mathbf{E}(t) |$. Hence,
\begin{equation}
    |f'(0)| \leq \Lambda_0(\lambda)\, v_0^{-\frac{1}{N}}\, |V'(0)|,
\end{equation}
which yields the desired conclusion using \eqref{eq: eq4 dim prof isop} and \eqref{eq: eq5 dim prof isop}.

\noindent We will now prove that $H_{\mathbf{E}^{v_0}} > 0$.
Let $x_0\in \partial^* C\cap \partial^* E_i$. 
For every $t\in \mathbb{R}$, define
\begin{equation}
    \Pi_t:=\{x\in \mathbb{R}^N : (x-x_0)\cdot \nu_C(x_0)=t\}.
\end{equation}
Let $t_*$ denote the first contact value for the translating hyperplanes, that is
\begin{equation}
    t_*:=\sup \{t\in \mathbb{R}: \overline{E}_i\cap \Pi_t \neq \emptyset\}.
\end{equation}
Then $t_*>0$ and there exists $y_0\in \Pi_{t_*}\cap \partial E_i \setminus C$. 
Moreover, by a standard argument, one can see that $y_0 \in \partial^* E_i$.
Indeed, suppose by contradiction that $y_0 \notin \partial^* E_i$. Then, blowing
up $E_i$ at $y_0$, by \cite[Theorem 28.6]{Ma12} we obtain, up to a
subsequence, a \emph{singular} minimizing cone $K$ (i.e.\ with
$\Sigma(K) = \partial K \setminus \partial^* K \neq \emptyset$). On the other
hand, since $E_i$ is contained in a half-space with $y_0$ belonging to the
boundary of that half-space, $K$ is contained in a half-space. This forces
$K$ to coincide with that half-space (see \cite[Theorem 15.5]{Giusti1984}),
hence $\Sigma(K) = \emptyset$, a contradiction. Therefore
$y_0 \in \partial^* E_i$.

By construction, $E_i$ lies on one side of $\Pi_{t_*}$ and touches it at $y_0$, hence by monotonicity of the mean curvature we have that $H_{\mathbf{E}^{v_0}} \geq 0$.
Moreover, if we assume by contradiction that $H_{\mathbf{E}^{v_0}}=0$, the free boundary $\partial^* E_i \setminus C_i$ is a minimal hypersurface. Hence, the strong maximum principle implies that $\partial^* E_i$ coincides with $\Pi_{t_*}$ in a neighbourhood of $y_0$. By analyticity of the regular part of the free boundary, this coincidence extends to the entire connected component of $\partial^* E_i \setminus C$ containing $y_0$. Thus this connected component of the free boundary is equal to $\Pi_{t_*}$, which is incompatible with the fact that $E_i$ has finite perimeter. In particular, since $I_C$ is almost everywhere differentiable and with positive derivative, we have that $I_C$ is strictly increasing.
\end{proof}

The $\Gamma$-convergence established in the previous sections also implies an upper-semicontinuity for the isoperimetric profile with respect to Kuratowski convergence of the obstacles.
    \begin{proposition}\label{prop: semicontinuity isop profile}
        Let $(C_n)_{n\in\mathbb{N}}$ be a sequence of closed convex sets converging in the Kuratowski sense to $C$ and $(v_n)_{n\in\mathbb{N}}$ be a positive sequence converging to $v>0$. Then 
        \begin{equation}
            \limsup\limits_{n \to +\infty} I_{C_n}(v_n) \leq I_C(v)
        \end{equation}
        
    \end{proposition}
    \begin{proof}
    Throughout the proof denote $B_r=B(0,r)$ for every $r>0$. Let $\Lambda > \Lambda_0(\lambda,1)v_n^{-\frac{1}{N}}$ for every $n\in\mathbb{N}$ where $\Lambda_0(\lambda,1)$ is the constant in \Cref{prop: esposito fusco unb}. Let $R>0$, from \Cref{cor: esistenza}, we have that there exists $E\subset B_R\setminus C$ such that
        \begin{equation}
            \mathcal{F}^{v}_{C,\Lambda}(E) = I_{C,R}(v)
        \end{equation}
        Since $C$ is a closed convex set, and $C_n \to C$ in the Kuratowski sense, using \Cref{lemma: local hausdorff} we have that $C_n \cap \widebar{B_{R+1}} \to C \cap \widebar{B_{R+1}}$ in Hausdorff convergence. Using the $\Gamma$-convergence result proved in \Cref{teo: gamma conv} we have that there exists $E_n \subset \mathbb{R}^N \setminus C_n$ such that $E_n \Subset B_{R+1}$, $E_n \to E$ strictly and $\mathcal{F}_{C_n, \Lambda}^{v_n} (E_n)=\mathcal{F}_{C_n\cap \widebar{B_{R+1}}, \Lambda}^{v_n} (E_n) \to \mathcal{F}^{v}_{C \cap \widebar{B_{R+1}},\Lambda}(E) = \mathcal{F}^{v}_{C \cap \widebar{B_{R}},\Lambda}(E)$. Putting all together we have
        \begin{align}
            I_{C,R}(v) &= I_{C\cap \widebar{B}_{R}, R}(v) = \mathcal{F}_{C\cap \widebar{B}_{R}}(E) = \mathcal{F}_{C\cap \widebar{B}_{R},\Lambda}^{v}(E) = \lim_{n\to +\infty} \mathcal{F}_{C_n \cap \widebar{B}_{R}, \Lambda}^{v_n}(E_n) 
            \\ &= \lim_{n\to +\infty} \mathcal{F}_{C_n, \Lambda}^{v_n}(E_n) \geq \limsup_{n\to +\infty} \inf\limits_{F \subset \mathbb{R}^N\setminus C_n} \mathcal{F}_{C_n,\Lambda}^{v_n}(F) = \limsup_{n\to +\infty} I_{C_n}(v_n)
        \end{align}
    \end{proof}
\begin{remark}\label{rmk: controesempio semicont isop}
    One might wonder if the semicontinuity proved in \Cref{prop: semicontinuity isop profile} could be improved to the continuity of the isoperimetric profile.
    This is certainly the case when the obstacles are compact and the convergence is in the sense of Hausdorff, thanks to \Cref{teo: gamma conv}.

    This is not the case, however, in general.
    Indeed, let $v>0$ and $C$ be a compact convex set whose boundary contains $B(0, R)\cap \{x_N=0\}$, let $(x_n)_{n\in \mathbb N}\subset \mathbb R^N$ such that $|x_n|\to \infty$ and let $C_n=C-x_n$.
    It is clear that, if $R$ is large enough, $I_C(v)=I_{C_n}(v)=I_{ \{x_N\leq 0 \} }(v)<I_{\emptyset}(v)$.
    On the other hand, $C_n\to \emptyset$ in the Kuratowski sense.
    
    \Cref{es: semicontinuity isop profile}, which relies heavily on the results of next section, gives an example in which the same lack of continuity happens even without loss of dimension.
\end{remark}

\section{Connectedness properties of generalized minimizers and asymptotic dimension}\label{section: connessione}
This section is devoted to studying the connectedness properties of generalized minimizers and to the study of the relation between the asymptotic geometry of the obstacle and its isoperimetric profile. 
The first result is a universal bound on the number of connected components (cfr.~\Cref{rmk: componenti connesse e scelta del rappresentante}) of a generalized minimizer, independent of both the volume $v$ and the closed convex obstacle $C$. 
As a preliminary step, we revisit the density estimates of \Cref{prop: density est} in the setting of generalized minimizers, proving that the radius $r_0$ can be chosen depending on the total volume rather than on the volumes of the individual components.

   \begin{lemma}\label{lemma: density est gen min}
    Let $v > 0$. 
    Then there exist $c_0 = c_0(\lambda) > 0$ and $r_0 = r_0(\lambda, v) \in (0,1)$ such that the following holds. If $C$ is a closed convex set, $(\mathbf{E};\mathbf{C}) = ((E_0,\dots,E_k);(C,C_1,\dots,C_k))$ is a generalized minimizer with volume constraint $v$, and $\Omega'$ is a connected component of $E_i$ for some $i$, then
    \begin{equation}
        |\Omega' \cap B(x,r)| \geq c_0 r^N
    \end{equation}
    for all $x \in \overline{\Omega'}$ and $r \in (0, r_0)$.
    \end{lemma}
    
    \begin{proof}
    Let $\Lambda = \Lambda_0(\lambda, v)$ be the constant given by \Cref{prop: esposito fusco unb}. By \Cref{prop: gli E_i sono Lambda minimi}, $E_i$ is a $\Lambda$-minimizer of $\mathcal{F}_{C_i}$ for every $i\in \{ 0 , \dots , k\}$. 
    Hence, by \Cref{prop: density est} and the fact that $\Lambda$ depends only on $\lambda$ and $v$, the conclusion follows.
    \end{proof}
    
    We are now ready to prove the bound on the number of connected components.

     \begin{proposition} \label{prop: bound componenti e diametro}
    There exist constants $L=L(\lambda)>0$ and $d=d(\lambda)>0$ such that if $C$ is a closed convex set, $v>0$, and $(\mathbf{E};\mathbf{C})$ is a generalized minimizer with volume constraint $v$, then the number of connected components of $(\mathbf{E};\mathbf{C})$ is less than $L$. Moreover, for every $i=0,\dots,k$, if $\Omega$ is a connected component of $E_i$, then $\operatorname{diam}(\Omega)\leq d\,v^{\frac{1}{N}}.$
    \end{proposition}
    
    \begin{proof}
    Up to rescaling, we may assume $v=1$. For every $i=0,\dots,k$, let $\Omega_1^i,\dots,\Omega_{l_i}^i$ be the connected components of $E_i$, and set
    \begin{equation}
        M=\sum_{i=0}^{k}l_i,
    \end{equation}
    which is the \emph{number of connected components} of $(\mathbf{E};\mathbf{C})$. By the density estimates proved in \Cref{lemma: density est gen min}, there exists a constant $w=w(\lambda)>0$ such that, for every $i=0,\dots,k$ and every $j=1,\dots,l_i$, if we fix a point $x_j^i\in\Omega_j^i$, then, setting $r=\frac{r_0(\lambda,1)}{2}$, we have
    \begin{equation}
        |\Omega_j^i\cap B(x_j^i,r)|\geq w.
    \end{equation}
    Therefore,
    \begin{align}
        1=|\mathbf{E}|=\sum_{i=0}^{k}|E_i|
        &\geq \sum_{i=0}^{k}\sum_{j=1}^{l_i}
        |\Omega_j^i\cap B(x_j^i,r)|\\
        &\geq \sum_{i=0}^{k}l_iw
        =Mw.
    \end{align}
    This proves the bound on the number of connected components. The estimate on the diameter follows from the density estimate proved in \Cref{lemma: density est gen min} by arguing as in \Cref{cor: stima sulle comp connesse e loro diametro}. The general case follows by scaling.
    \end{proof}

The study of the interaction between the asymptotic geometry of the obstacle and its isoperimetric profile is inspired by \cite{RigLargeVol}, where a similar study is performed for the relative perimeter functional, i.e.~$\mathcal F_C$ for $\lambda=0$.
In \Cref{teo: principale d*(C)} we extend this to the general $\lambda\in (-1, 1)$.
To do so, let us recall the definition of the fundamental objects that describe the asymptotic geometry of $C$.
Given a closed convex set $C \subset \mathbb{R}^N$ and a point $x \in C$, the \emph{recession cone} of $C$ is defined by
\begin{equation}
   C_{\infty} := \bigcap_{t>0} t (C - x). 
\end{equation}
\noindent The recession cone $C_\infty$ does not depend on the choice of $x$ and its dimension quantifies the set of directions along which $C$ is unbounded.
The case when the dimension of $C_\infty$ is larger than $N-1$ is somewhat rigid, as the following statement shows.
\begin{lemma}\label{lem: lemma 4.3 Large volume}
    Let $C\subsetneq \mathbb R^N$ be a closed convex set such that $\dim (C_\infty) \geq N-1$. 
    Then there exists an asymptotic cylinder $K\in \mathcal K(C)$ whose boundary contains an hyperplane. 
    In particular, for all $\lambda>0$  
    \begin{equation}
        I_C(v) = I_H(v) \qquad \forall v>0,
    \end{equation}
    where $H$ is a half-space.
    \end{lemma}
    \begin{proof}
        If $\dim(C)=N$, the proof follows from \cite[Lemma 4.3]{RigLargeVol}, by noticing that the existence of such a $K$ is present in the proof of that Lemma. If $\dim(C)<N$, the obvious inequality $\dim(C_\infty)\leq\dim(C)$, together with $\dim(C_\infty)\geq N-1$, implies $\dim(C_\infty)=\dim(C)=N-1$.
        Without loss of generality, assume that $0\in C$ and that $C\subset \{x_N=0 \}$. Let $w\in{\rm int}(C_\infty)$ be an interior point relative to the hyperplane. Then there exists $r>0$ such that the $(N-1)$-dimensional ball centered at $w$ with radius $r$ is contained in $C_\infty$. Since $C_\infty$ is a cone, it follows that the $(N-1)$-dimensional ball centered at $nw$ with radius $nr$ is contained in $C_\infty\subset C$ for every $n\in\mathbb N$. Setting $x_n=nw\in C$, we therefore have $C-x_n\xrightarrow{n\to+\infty}\{ x_N=0 \}$ in the Kuratowski sense.
    \end{proof}
    
    \begin{remark}
        The statement above is slightly different from \cite[Lemma 4.3]{RigLargeVol}, which is only stated for $\dim C=N$ and does not mention the existence of such a $K$.
        The proof of the existence of $K$ is however contained in the proof of \cite[Lemma 4.3]{RigLargeVol}. The proof in the case $\dim C<N$ is immediate.
    \end{remark}

The \emph{asymptotic dimension} $d^*(C)$ of $C$ is defined as
\begin{equation}
    d^*(C) \coloneqq \max \;\dim\left(\lim_n t_n (C - x_n)\right),  
\end{equation}
where the maximum is taken over the sequences $(x_n)_{n\in \mathbb N}\subset C$ and the vanishing positive sequences $(t_n)_{n\in \mathbb N}$; the convergence is understood in the sense of Kuratowski.
By construction, one always has $d^*(C) \geq \dim(C_{\infty})$, and the inequality may be strict in general. 
More in general, for a family $\mathcal H$ of closed convex sets, we define the asymptotic dimension of $\mathcal H$ as
\begin{equation}
    d^*(\mathcal H) \coloneqq  \max \;\dim\left(\lim_n t_n (C_n - x_n)\right) ,  
\end{equation}
where this time the maximum is taken over the sequences $(C_n)_{n\in \mathbb N}\subset \mathcal H$, $(x_n)_{n\in \mathbb N}$ with $x_n\in C_n$ for all $n$ and the vanishing positive sequences $(t_n)_{n\in \mathbb N}$.

\begin{remark}\label{rmk: caratterizzazione di d*(C)}
    In \cite[Remark 3.4]{RigLargeVol}, the authors give the following characterization: $d^{*}(C)$ is the smallest $k \in \mathbb{N} \cup \{ 0 \}$ such that there exist a vector subspace $Z \subset \mathbb{R}^N$ of dimension $k$ and a bounded set $B \subset Z^{\perp}$ for which $C \subset Z + B$, where $Z^{\perp}$ denotes the orthogonal complement of $Z$.
\end{remark}
Using the above characterization, we prove the following useful lemma about the asymptotic dimension of the family given by a closed convex set and the set of its asymptotic cylinders, defined in \Cref{def: asymptotic cylinders}: 
\begin{lemma}\label{lemma: dim(K) <= d*(C)}
Let $C \subset \mathbb{R}^N$ be a closed convex set and let $\mathcal H= \mathcal{K}(C) \cup \{C\}$.
Then $d^*(\mathcal H)=d^*(C)$.

\end{lemma}
\begin{proof}
    Since $C\in \mathcal H$, $d^*(C)\leq d^*(\mathcal H)$. We will prove the other inequality. Throughout the proof, all limits of sets are taken in the Kuratowski sense. Let $K=\lim\limits_n t_n (C_n - x_n)$ where $(t_n)_{n\in \mathbb N}$ is a vanishing sequence, $(C_n)_{n\in \mathbb N}$ is a sequence in $\mathcal H$ and $x_n\in C_n$ for all $n$.
    Without loss of generality, assume $0\in C$. By \Cref{rmk: caratterizzazione di d*(C)}, we can find $Z$, a $d^*(C)$-dimensional subspace of $\mathbb{R}^N$ and a compact set $B \subset Z^{\perp}$ such that $C \subset Z + B$. 
    Since $C_n \in \mathcal H$ there exists a sequence $(y_h^n)_{h\in\mathbb{N}}$ in $C$ such that $C_n = \lim\limits_{h \to +\infty} (C - y_h^n)$. By the splitting above, we can write $y_h^n = z_h^n + b_h^n$ with $z_h^n \in Z$ and $b_h^n \in B$. Hence
    \begin{equation}
        C - y_h^n \subset Z + B - b_h^n \subset Z + (B - B).
    \end{equation}
    In particular, $C_n \subset Z + A$ for some compact set $A \subset Z^{\perp}$ and for every $n\in\mathbb{N}$. Again, writing $x_n = w_n + a_n$ with $w_n \in Z$ and $a_n \in A$, we obtain
    \begin{equation}
        t_n (C_n - x_n) 
        \subset t_n (Z + A - w_n - a_n) = Z + t_n(A-a_n) \subset Z + t_n (A-A) \to Z
    \end{equation}
    thanks to the boundedness of $A$.
    Therefore $K \subset Z$, and we conclude that $\dim (K) \leq \dim(Z) = d^*(C)$.
\end{proof}

Arguing as in the proof of \cite[Theorem~1.1]{RigLargeVol}, we get the following interesting characterization of the asymptotic dimension.
\begin{corollary}\label{cor: caratterizzazione d*}
    Let $C$ be a closed convex set. 
    Then
    \begin{equation}
        d^*(C)=\max_{K\in \{C\}\cup \mathcal K(C)} \dim(K_\infty).
    \end{equation}
\end{corollary}
\begin{proof}
    First, by \Cref{lemma: dim(K) <= d*(C)} and the fact, observed before, that for a closed convex set $L$ it holds $d^{*}(L)\geq \dim (L_\infty)$, we have
    \begin{equation}
        d^*(C)=d^*(\mathcal{K}(C) \cup \{C\})\geq \max_{K\in \{C\}\cup \mathcal K(C)} \dim(K_\infty).
    \end{equation}

    Let us prove the opposite inequality.
    Let $(x_n)_{n\in \mathbb N}\subset C$ and $t_n\to 0$ be such that $t_n(C-x_n)\to K_{max}$ in the sense of Kuratowski with $d^*(C)=\dim(K_{max})$.
    If $(x_n)_{n\in \mathbb N}$ is bounded, $K_{max}=C_\infty$.
    Otherwise, up to subsequences, there exists $K\in \mathcal K(C)$ such that $C-x_n\to K$.
    Since $C$ is convex and $x_n\in C$, for $t>0$ and $n$ large enough, $t_n(C-x_n)\subset t (C-x_n)$.
    This implies $K_{max}\subset t K$ for all $t>0$, thus $K_{max}\subset K_\infty$, hence $\dim (K_\infty)\geq d^*(C)$, concluding the proof.
\end{proof}

The following theorem describes the properties of generalized minimizers of the capillarity energy depending on the asymptotic dimension of $C$.
We are in particular interested in when a generalized minimizer is connected, i.e.\ it consists of a single nonempty entry and such entry is connected.

\begin{theorem}\label{teo: principale d*(C)}
    Let $C$ be a closed convex set and $\lambda \in (-1, 1)$.
    Then the following properties hold:
    \begin{itemize}
        \item[(a)] Suppose that $d^{*}(C) \geq N-1$. Then there exists an asymptotic cylinder of $C$ whose boundary contains a hyperplane. As a consequence,
    \begin{align}
        I_C(v) = I_H(v) \ \text{for every }v>0 \label{eq: I_C = I_H}
    \end{align}
    where $H$ is a half-space, and every generalized minimizer of $I_C(v)$ is isometric to $B^{\lambda}[v]$, and sits on a flat part of the boundary of $C$ or on a flat part of the boundary of an asymptotic cylinder of $C$.
    \item[(b)] Suppose that $d^{*}(C) \leq N-2$. Then
    \begin{align}
        \lim_{v\to\infty}\frac{I_C(v)}{N\,\omega_N^{1/N}\,v^{(N-1)/N}} = 1\,, \label{eq: limite nel caso d*(C) < N-1}
    \end{align}
    Moreover:

    \medskip
    \noindent
    ($\text{b}1$) \textit{(Large Volume Regime)} There exists a constant $v_1(\lambda, C) > 0$ such that any generalized minimizer for $I_C(v)$ is connected for $v > v_1$.

    \medskip
    \noindent
    ($\text{b}2$) \textit{(Small Volume Regime)} There exists a constant $v_0(\lambda, C) > 0$ such that any generalized minimizer for $I_C(v)$ is connected for $v < v_0$.
    \end{itemize}
        
    \end{theorem}

    \begin{proof}
    We divide the proof into three steps.
    
    \vspace{0.2cm}
    \noindent
        \textbf{Step 1: proof of (a).}
    By \Cref{cor: caratterizzazione d*}, there exists $K\in \{C\}\cup \mathcal K(C)$ such that $\dim (K_\infty)\geq N-1.$
    By \Cref{lem: lemma 4.3 Large volume}, there exists $L\in \mathcal K(K)$ whose boundary contains a hyperplane.
    By \Cref{prop: i cilindri asint dei cilindri asint sono cilindri asint}, $L\in \mathcal K(C)$. Consider the asymptotic configuration $(\mathbf E; \mathbf C)=((\emptyset, E_1);(C, L))$, where $E_1 \subset \mathbb{R}^N \setminus L$ is a set isometric to $B^\lambda[v]$ that sits on the flat part of the boundary of $L$.
    By \Cref{prop: isop profile del semispazio}, this is optimal, proving \eqref{eq: I_C = I_H}.
    
    Let $v > 0$ and let $(\mathbf{E};\mathbf{C}) = ((E_0, E_1, \dots, E_k); (C, C_1, \dots, C_k))$ be a generalized minimizer for $I_C(v)$. For every $i = 0, \dots, k$, let $\Omega_{1}^{i}, \dots, \Omega_{l_i}^{i}$ be the connected components of $E_i$, and set $v_j^{i} = |\Omega_{j}^{i}|$. By the density estimate proved in \Cref{lemma: density est gen min} we have that there exists $w=w(\lambda, v)>0$ such that if $E_i \neq \emptyset$, then $v_j^{i} \geq w.$ By the isoperimetric inequality in \Cref{prop: isop profile del semispazio}, we have
    \begin{equation}
        I_H(v) = I_C(v) = \sum_{i=0}^{k} \sum_{j=1}^{l_i} \mathcal{F}_{C_i}(\Omega_{j}^{i})
        \geq \sum_{i=0}^{k} \sum_{j=1}^{l_i} I_H(v_j^{i})
        \geq I_H(v).
    \end{equation}
    Hence all inequalities are equalities and, by item~\ref{it: strict subadditivity} of \Cref{prop: isop profile del semispazio}, this implies that $(\mathbf{E};\mathbf{C})$ has only one connected component, say $\Omega = E_i$ for some $i$. Consequently,
    \begin{equation}
        I_C(v) = \mathcal{F}_{C_i}(E_i) = I_H(v).
    \end{equation}
    Using again the isoperimetric inequality in \Cref{prop: isop profile del semispazio}, this implies that $E_i$ is isometric to $B^{\lambda}[v]$ and lies on a flat part of the boundary of $C_i$.
    
    \vspace{0.2cm}
    \noindent
    \textbf{Step 2: proof of \eqref{eq: limite nel caso d*(C) < N-1} and (b1).} 
    The argument follows a strategy similar to \cite[Theorem 1.1]{RigLargeVol}.
    Using a ball as a competitor, we have that 
    \begin{equation}
        I_C(v) \leq N \omega_{N}^{\frac{1}{N}} v^{\frac{N-1}{N}},
    \end{equation}
    hence, to prove \eqref{eq: limite nel caso d*(C) < N-1}, it is enough to show that for any sequence $v_n \to +\infty$,
    \begin{equation}
        \liminf_{n \to +\infty} \frac{I_C(v_n)}{N \omega_N^{\frac{1}{N}} v_n^{\frac{N-1}{N}}}
        = \liminf_{n \to +\infty} \frac{I_{v_n^{-1 / N}C}(1)}{N \omega_N^{\frac{1}{N}} } \geq 1.
    \end{equation}
    
    Let $(\tilde{\mathbf{E}}^{(n)}; \tilde{\mathbf{C}}^{(n)})$ be a generalized minimizer of $I_C(v_n)$ and let $(\mathbf{E}^{(n)}; \mathbf{C}^{(n)})= v_n^{-\frac{1}{N}} (\tilde{\mathbf{E}}^{(n)}; \tilde{\mathbf{C}}^{(n)})$ be the corresponding rescaled generalized minimizer of $I_{v_n^{-1/N}C}(1)$. 
    By the uniform bound on the number of connected components given by \Cref{prop: bound componenti e diametro}, up to subsequences, we may assume that the length of $(\mathbf{E}^{(n)}; \mathbf{C}^{(n)})$ is $k\in\mathbb{N}$ for every $n$, and that the total number of connected components is $l \in\mathbb{N}$ for every $n$. 
    More precisely, let $(\mathbf{E}^{(n)}; \mathbf{C}^{(n)})=((E^{(n)}_0,\dots,E^{(n)}_k);(C_0^{(n)},\dots,C_k^{(n)}))$ and, for every $i=0,\dots,k$, let $\Omega_{i}^{1,(n)},\dots, \Omega_{i}^{l_i,(n)}$ be the connected components of $E_i^{(n)}$, where $\sum_{i=0}^{k} l_i = l$. Using again \Cref{prop: bound componenti e diametro}, we have that there exists $d=d(\lambda)>0$ such that $\text{diam} (\Omega_{i}^{j,(n)}) < d$. 
    Hence we can find $R>0$ and $x_i^{j,(n)} \in C_i^{(n)} = v_n^{-\frac{1}{N}} \tilde{C}_{i}^{(n)}$ such that $\Omega_{i}^{j,(n)} - x_i^{j,(n)} \subset B(0,R)$.
    Up to subsequences, there exists $\Omega_{i}^{j} \subset B(0,R)$ such that 
    \begin{equation}
        \Omega_{i}^{j,(n)} - x_i^{j,(n)} \xrightarrow[n\to +\infty]{L^1} \Omega_{i}^{j}
    \end{equation}
    Up to extracting further subsequences, we can also find closed convex sets $K_i^j$ such that 
    \begin{equation}
        C_i^{(n)} - x_i^{j,(n)} \xrightarrow[]{n\to +\infty} K_i^j.
    \end{equation}
    in sense of Kuratowski. By the definition of asymptotic dimension and \Cref{lemma: dim(K) <= d*(C)}, we have that $\dim (K_i^{j}) \leq d^{*}(\tilde{C}_i^{(n)}) \leq d^{*}(C)\leq  N-2$, so $\mathcal{F}_{K_i^j}(\cdot)= \mathcal{F}_{\emptyset}(\cdot)=P(\cdot)$. By \Cref{lemma: local hausdorff} and the $\Gamma$-convergence proved in \Cref{teo: gamma conv} we have that
    \begin{equation}
        \mathcal{F}_{K_i^j} (\Omega_{i}^j) \leq \liminf_{n \to +\infty} \mathcal{F}_{C_i^{(n)}-x_i^{j,(n)}}\left(\Omega_{i}^{j,(n)}-x_i^{j,(n)} \right).
    \end{equation}
    
    Putting everything together and using isoperimetric inequality, we conclude
    \begin{align}
        N \omega_N^{\frac{1}{N}} &\geq \liminf_{n \to +\infty} I_{v_n^{- 1/N}C}(1) \geq \sum_{i=0}^{k} \sum_{j=1}^{l_i} \liminf_{n \to +\infty} \mathcal{F}_{C_i^{(n)}-x_i^{j,(n)}}\left( \Omega_{i}^{j,(n)}-x_i^{j,(n)} \right) \geq \sum_{i=0}^{k} \sum_{j=1}^{l_i} \mathcal{F}_{K_i^{j}} (\Omega_{i}^j) 
        \\&= \sum_{i=0}^{k} \sum_{j=1}^{l_i} P (\Omega_{i}^j) \geq \sum_{i=0}^{k} \sum_{j=1}^{l_i} N \omega_N^{\frac{1}{N}} (v_i^j)^{\frac{N-1}{N}} \geq N \omega_N^{\frac{1}{N}}. \label{eq: stretta concavità implica connessione per volumi grandi}
    \end{align}
    This proves \eqref{eq: limite nel caso d*(C) < N-1}. Moreover, arguing by contradiction, suppose that there exists a sequence $v_n \to +\infty$ such that a corresponding generalized minimizer has at least two connected components. Repeating the argument above, by \Cref{lemma: density est gen min} there exists $w=w(\lambda)>0$ such that $\left| \Omega_{i}^{j,(n)} \right| \geq w$, hence $v_i^{j} = |\Omega_{i}^{j}| \geq w$. Therefore, using the strict concavity of $z \mapsto z^{\frac{N-1}{N}}$, the last inequality in \eqref{eq: stretta concavità implica connessione per volumi grandi} becomes strict, yielding a contradiction. This concludes the proof of (b1).

    \vspace{0.2cm}
    \noindent
    \textbf{Step 3: proof of (b2).} 
    We may assume that $C$ is a convex body.
    Indeed, if $\dim(C)\leq N-2$, minimizers are always connected, since they are balls.
    If $\dim(C)=N-1$, it is enough to choose $v_0$ sufficiently small so that a $\lambda$-spherical cap of volume $v<v_0$ can be accommodated on the flat part of $C$.
    Then, by the isoperimetric inequality in \Cref{prop: isop profile del semispazio}, all minimizers with volume constraint $v<v_0$ are of this form.
    
    Up to a translation and a rotation, we may assume that $0\in \partial^* C$ and that $C\subset H = \{ x_N \leq 0 \}$. Suppose by contradiction that there exist a sequence $v_n \to 0$ of volumes and a sequence $(\tilde{\mathbf{E}}^{(n)};\tilde{\mathbf{C}}^{(n)})$ of generalized minimizers for $I_C(v_n)$ with at least two connected components. 
    As before, let $(\mathbf{E}^{(n)};\mathbf{C}^{(n)})=v_n^{-1/N}(\tilde{\mathbf{E}}^{(n)};\tilde{\mathbf{C}}^{(n)})$ be the rescaled generalized minimizers and $C_n=v_n^{-1/N}C$. Note that 
    \begin{equation}
        C_n\to H \ \ \text{in} \ \ L^1_{loc} \ \text{and in the Kuratowski sense.} \label{eq: boh}
    \end{equation} 
    Again, by the uniform bound on the number of connected components from \Cref{prop: bound componenti e diametro}, up to subsequences we may assume that
    \begin{equation}
    (\mathbf{E}^{(n)};\mathbf{C}^{(n)})=((E_0^{(n)},\dots,E_k^{(n)});(C_n,C_1^{(n)},\dots,C_k^{(n)}))
    \end{equation}
    and that $\Omega_{i}^{1,(n)},\dots,\Omega_{i}^{l_i,(n)}$ are the connected components of $E_i^{(n)}$ for every $n\in\mathbb{N}$ and every $i=0,\dots,k$. In this way, the total number of connected components is $l=\sum\limits_{i=0}^{k}l_i$. Set $v_i^{j,n}=|\Omega_{i}^{j,(n)}|$; from \Cref{lemma: density est gen min} we have that there exists $w=w(\lambda)>0$ such that $v_i^{j,n}>w$ whenever $E_i^{(n)} \neq \emptyset$. From \Cref{prop: isop profile del semispazio} there exists $\alpha>0$ such that
    \begin{align}
    I_{C_n}(1)&=\sum_{i=0}^{k}\mathcal{F}_{C_i^{(n)}}(E_i^{(n)})
    = \sum_{i=0}^{k}\sum_{j=1}^{l_i}\mathcal{F}_{C_i^{(n)}}\left(\Omega_{i}^{j,(n)}\right)  
    \\ &\geq \sum_{i=0}^{k}\sum_{j=1}^{l_i}\mathcal{F}_{H}\left(B^{\lambda}\left[v_i^{j,n}\right]\right) > \alpha > I_H(1).
    \end{align}
    Our aim is now to construct a sequence $(F_n)_{n\in\mathbb{N}}$ such that $F_n\subset \mathbb{R}^N\setminus C_n$, $|F_n|=1$ and
    \begin{equation}
    \mathcal{F}_{C_n}(F_n) \to I_H(1) = \mathcal{F}_{H}(B^{\lambda}[1]),
    \end{equation}
    yielding a contradiction. Let $B$ be the ball such that $B\setminus H = B^{\lambda}[1]$. 
    Since $|B|>1$ and $C_n\subset H$, by \eqref{eq: boh} we can find a sequence $(x_n)_{n\in\mathbb{N}}$ such that $x_n\to 0$ and
    \begin{equation}
    |(B+x_n)\setminus C_n|=1.
    \end{equation}
    Let $F_n=(B+x_n)\setminus C_n$, and note that $F_n \to B^{\lambda}[1]$ in $L^1$. We now prove that the two parts of the energy converge:
    \begin{gather*}
    P(F_n,\mathbb{R}^N\setminus C_n)
     =\int_{\partial B+x_n}\chi_{\mathbb{R}^N\setminus C_n}(y)\,d\mathcal{H}^{N-1}(y) 
    \\= \int_{\partial B}\chi_{\mathbb{R}^N\setminus C_n}(y-x_n)\,d\mathcal{H}^{N-1}(y) 
    \longrightarrow \int_{\partial B^{\lambda}[1]}\chi_{\mathbb{R}^N\setminus H}(y)\,d\mathcal{H}^{N-1}(y) = P(B^{\lambda}[1],\mathbb{R}^N\setminus H).
    \end{gather*}
    For the wet part, since $\mathcal{H}^{N-1} \mres\partial^*(C_n-x_n) \xrightharpoonup{*}\mathcal{H}^{N-1}\mres \partial H$, we get
    \begin{gather*}
    P(F_n,C_n) = \mathcal{H}^{N-1}(\partial(C_n-x_n)\cap B ) 
    \\= \mathcal{H}^{N-1}\mres \partial^*(C_n-x_n)(B) \longrightarrow  \mathcal{H}^{N-1}\mres \partial H (B) 
    = P(B^{\lambda}[1],H).
    \end{gather*}
    \end{proof}

\begin{corollary}\label{cor: I_C>I_H e esistenza dei minimi standard}
        Let $C$ be a closed convex set and $\lambda \in (-1, 1)$.
        Then 
        \begin{enumerate}
            \item if $d^{*}(C) \leq N-2$, $I_C(v) > I_H(v)$ for every $v$ sufficiently large, depending on $\lambda$ and $C$.
            \item if $d^{*}(C) \geq N-1$ and $v>0$,  $\mathcal{F}_C$ admits a minimizer with volume constraint $v$ in $\mathbb{R}^N \setminus C$ if and only if $\partial C$ contains a flat portion that can accommodate a spherical cap isometric to $B^{\lambda}[v]$.
        \end{enumerate}
    \end{corollary}
    \begin{proof}
        The proof follows immediately from \Cref{prop: isop profile del semispazio} and \Cref{teo: principale d*(C)}.
    \end{proof} 
    Finally, we conclude with an example showing that the semicontinuity proved in \Cref{prop: semicontinuity isop profile} cannot be improved to continuity, even when there is no loss of dimension in the limit. We construct a convex body $C \subset \mathbb{R}^3$ which is unbounded but asymptotically confined inside a cylinder with bounded base. 
    This geometric constraint implies that $d^*(C)=1$. 
    We then approximate $C$ by a sequence of convex bodies $(C_n)_{n\in\mathbb{N}}$ such that $d^*(C_n)=3$ for every $n$. 
    By \Cref{cor: I_C>I_H e esistenza dei minimi standard}, there exists $v_1>0$ such that $I_C(v) > I_H(v)$ for every $v>v_1$ and one has $I_{C_n}(v)=I_H(v)$ for every $v>0$. The sets $C_n$ are constructed so as to coincide with $C$ up to a given height, and their boundary is modified above that level by replacing the original profile with its tangent lines. 
    This modification preserves convexity while producing sets that are no longer asymptotically confined in any cylinder with bounded base. Despite this, the sequence $(C_n)_{n\in\mathbb{N}}$ converges to $C$ in the sense of Kuratowski. 
    As a consequence, for every $v>v_1$ one has
    \begin{equation}
        \limsup_{n\to\infty} I_{C_n}(v) = I_H(v) < I_C(v).
    \end{equation}
    
    \begin{example}\label{es: semicontinuity isop profile}
        Let
        \begin{equation}
            K := \left\{ (x,y,z)\in \mathbb{R}^3 : x^2+y^2 < 1 \right\}
        \end{equation}
        be the unit cylinder and define
        \begin{equation}
            C := \left\{ (x,y,z)\in \mathbb{R}^3 : x^2+y^2<1, \ 
            z \geq \frac{x^2+y^2}{1-x^2-y^2} \right\}.
        \end{equation}
        Then $C$ is a convex body contained in $K$ and unbounded in the $z$-direction. In particular, $C$ is trapped in the cylinder $K$ and therefore $d^*(C)=1$. For every $n\in \mathbb{N}$, set
        \begin{equation}
            r_n := \sqrt{\frac{n}{n+1}}
        \end{equation}
        and define the function
        \begin{equation}
            f_n(r)
            :=
            \begin{cases}
                \dfrac{r^2}{1-r^2}, & 0\leq r\leq r_n, \\[0.8em]
                n + f'(r_n)(r-r_n), & r\geq r_n,
            \end{cases}
        \end{equation}
        where $f(r)=\frac{r^2}{1-r^2}$ and $f'(r)=\frac{2r}{(1-r^2)^2}$. We then define
        \begin{equation}
            C_n := \left\{ (x,y,z)\in \mathbb{R}^3 : z \geq f_n\big(\sqrt{x^2+y^2}\big) \right\}.
        \end{equation}
        By construction, $C_n$ coincides with $C$ up to height $z=n$, while for $z\geq n$ its boundary is given, in each vertical plane containing the $z$-axis, by the tangent line to the graph of $f$ at $r=r_n$. 
        In particular, each $C_n$ is a closed convex set. It is immediate to see that $C_n \to C$ in the sense of Kuratowski. Finally, we claim that $d^*(C_n)=3$ for every $n$. Indeed, for $r\geq r_n$ the function $f_n$ is affine with strictly positive slope, hence $C_n$ contains a conical region of the form
        \begin{equation}
            \left\{ (x,y,z) : z \geq a_n \sqrt{x^2+y^2} + b_n \right\}
        \end{equation}
        for suitable constants $a_n>0$, $b_n\in \mathbb{R}$. This implies that the dimension of the recession cone of $C_n$ is $3$, and therefore $d^*(C_n)=3$.
    \end{example}

\section*{Acknowledgements}
The authors would like to thank Professor Massimiliano Morini and Professor Filippo Cagnetti for several helpful discussions and valuable suggestions.
Simone Vincini would like to thank Davide Carazzato for the valuable help in the revision of the paper and for suggesting the example in \Cref{rmk: controesempio semicont isop}. Gabriele Nani  was funded from INdAM under the INdAM–GNAMPA Project
2026 \emph{Metodi variazionali per flussi geometrici non locali e problemi isoperimetrici} (CUP: E53C25002010001). Simone Vincini was funded by the Austrian Science Fund (FWF) 10.55776/EFP6.
For open access purposes, the authors have applied a CC BY public copyright license to any author accepted manuscript version arising from this submission.

	\bibliographystyle{siam} 
	\bibliography{main/bib/bibliografia} 

\begin{thebibliography}{10}

\bibitem{AmbrosioCaselles}
{\sc L.~Ambrosio, V.~Caselles, S.~Masnou, and J.-M. Morel}, {\em Connected components of sets of finite perimeter and applications to image processing}, Journal of the European Mathematical Society, 3 (2001), pp.~39--92.

\bibitem{AmFuPa00}
{\sc L.~Ambrosio, N.~Fusco, and D.~Pallara}, {\em Functions of bounded variation and free discontinuity problems}, Oxford Mathematical Monographs, The Clarendon Press, Oxford University Press, New York, 2000.

\bibitem{refId0}
{\sc G.~Antonelli, S.~Nardulli, and M.~Pozzetta}, {\em The isoperimetric problem via direct method in noncompact metric measure spaces with lower ricci bounds}, ESAIM: COCV, 28 (2022), p.~57.

\bibitem{ChoeGhomiRitore}
{\sc J.~Choe, M.~Ghomi, and M.~Ritor{\'e}}, {\em The relative isoperimetric inequality outside convex domains in $\mathbb{R}^n$}, Calculus of Variations and Partial Differential Equations, 29 (2007), pp.~421--429.

\bibitem{dalmasoIntroductionGConvergence1993}
{\sc G.~Dal~Maso}, {\em An {Introduction} to $\Gamma$-{Convergence}}, Birkhäuser, Boston, MA, 1993.

\bibitem{DePhilMaggi}
{\sc G.~De~Philippis and F.~Maggi}, {\em Regularity of free boundaries in anisotropic capillarity problems and the validity of young’s law}, Archive for Rational Mechanics and Analysis, 216 (2014), pp.~473--568.

\bibitem{EspositoFusco}
{\sc L.~Esposito and N.~Fusco}, {\em A remark on a free interface problem with volume constraint}, Journal of Convex Analysis, 18 (2011), pp.~417--426.

\bibitem{Finn}
{\sc R.~Finn}, {\em Equilibrium Capillary Surfaces}, vol.~284 of Grundlehren der mathematischen Wissenschaften [Fundamental Principles of Mathematical Sciences], Springer New York, NY, 1986.

\bibitem{FoFuLeMo23}
{\sc I.~Fonseca, N.~Fusco, G.~Leoni, and M.~Morini}, {\em Global and local energy minimizers for a nanowire growth model}, Annales de l'Institut Henri Poincar\'{e} C. Analyse Non Lin\'{e}aire, 40 (2023), pp.~919--957.

\bibitem{IsopCap}
{\sc N.~Fusco, V.~Julin, M.~Morini, and A.~Pratelli}, {\em The isoperimetric inequality for the capillary energy outside convex sets}.
\newblock Preprint, To appear on Geometric And Functional Analysis (GAFA), 2025.

\bibitem{RigLargeVol}
{\sc N.~Fusco, F.~Maggi, M.~Morini, and M.~Novack}, {\em Rigidity and large volume residues in exterior isoperimetry for convex sets}, Advances in Mathematics, 453 (2024).

\bibitem{FuMo23}
{\sc N.~Fusco and M.~Morini}, {\em Total positive curvature and the equality case in the relative isoperimetric inequality outside convex domains}, Calculus of Variations and Partial Differential Equations, 62 (2023), pp.~Paper No. 102, 32.

\bibitem{Giusti1984}
{\sc E.~Giusti}, {\em Minimal Surfaces and Functions of Bounded Variation}, Monographs in Mathematics, Birkh{\"a}user Boston, MA, 1984.

\bibitem{Ivanov}
{\sc B.~A. Ivanov}, {\em Straight line segments on the boundary of a convex body}, Ukrainskiy Geometricheskiy Sbornik,  (1973), pp.~69--71, 2.

\bibitem{LeRiVe22}
{\sc G.~P. Leonardi, M.~Ritor\'{e}, and E.~Vernadakis}, {\em Isoperimetric inequalities in unbounded convex bodies}, Memoirs of the American Mathematical Society, 276 (2022), pp.~1--86.

\bibitem{maderbaumdicker2025quantitative}
{\sc E.~M{\"a}der-Baumdicker, R.~Neumayer, J.~Park, and M.~Rupflin}, {\em Quantitative estimates for the relative isoperimetric problem and its gradient flow outside convex bodies in the plane}, arXiv preprint arXiv:2508.21198,  (2025).

\bibitem{Ma12}
{\sc F.~Maggi}, {\em Sets of finite perimeter and geometric variational problems}, vol.~135 of Cambridge Studies in Advanced Mathematics, Cambridge University Press, Cambridge, 2012.
\newblock An introduction to geometric measure theory.

\bibitem{MoPoSp22}
{\sc M.~Morini, M.~Ponsiglione, and E.~Spadaro}, {\em Long time behavior of discrete volume preserving mean curvature flows}, Journal für die reine und angewandte Mathematik (Crelles Journal), 2022 (2022), pp.~27--51.

\bibitem{Nardulli2014}
{\sc S.~Nardulli}, {\em Generalized existence of isoperimetric regions in non-compact riemannian manifolds and applications to the isoperimetric profile}, Asian Journal of Mathematics, 18 (2014), pp.~1--28.

\bibitem{PozzettaSurvey}
{\sc M.~Pozzetta}, {\em Isoperimetry on manifolds with ricci bounded below: overview of recent results and methods}, in Anisotropic Isoperimetric Problems and Related Topics, Springer, 2024, pp.~49--89.

\bibitem{ritoreExistenceCharacterizationRegions2004}
{\sc M.~Ritoré and C.~Rosales}, {\em Existence and {Characterization} of {Regions} {Minimizing} {Perimeter} under a {Volume} {Constraint} inside {Euclidean} {Cones}}, Transactions of the American Mathematical Society, 356 (2004), pp.~4601--4622.

\bibitem{Schneider_2013}
{\sc R.~Schneider}, {\em Convex Bodies: The Brunn–Minkowski Theory}, Encyclopedia of Mathematics and its Applications, Cambridge University Press, 2~ed., 2013.

\bibitem{Sternberg-Zumbrun}
{\sc P.~Sternberg and K.~Zumbrun}, {\em On the connectivity of boundaries of sets minimizing perimeter subject to a volume constraint}, Communications in analysis and geometry, 7 (1999), pp.~199--220.

\bibitem{Taylor1977}
{\sc J.~E. Taylor}, {\em Boundary regularity for solutions to various capillarity and free boundary problems}, Communications in Partial Differential Equations, 2 (1977), pp.~323--357.

\end{thebibliography}
\end{document}